\documentclass[11pt]{article}
\usepackage[utf8]{inputenc}
\usepackage[numbers]{natbib}
\title{Jordan permutation groups and limits of $D$-relations}
\author{ Asma Ibrahim Almazaydeh  and Dugald Macpherson}
\date{Septemper, 2020}
\usepackage{pgf}

\usepackage{amssymb,amsbsy,times,fancyhdr,color}
\usepackage{amsmath}
\usepackage[dvips]{epsfig}
\usepackage{graphicx}
\graphicspath{ {images/} }
\usepackage{latexsym}
\usepackage{amsthm}
\usepackage{float,graphicx, enumerate}
\usepackage[margin=3cm]{geometry}
\usepackage[none]{hyphenat}

\usepackage{hyperref}
\hypersetup{pdfborder={0 0 0}}

\usepackage[all]{xy}
\usepackage{amsopn}

\usepackage{color, colortbl}

\usepackage{arabtex}
\usepackage{utf8}

\usepackage{latexsym}

\usepackage{verbatim}
\usepackage{amssymb}

\usepackage{mathrsfs}

\usepackage{amsthm}

\usepackage{inputenc}
\usepackage{tikz}
\usetikzlibrary{trees}
\usepackage{textcomp}
\usepackage{romannum}
\usepackage{wasysym}
\usepackage{color}

\usepackage{lipsum}  
\usepackage{graphicx}
\usepackage{caption}
\usepackage{float}
\usepackage{subcaption}
\usetikzlibrary{decorations.pathreplacing,angles,quotes}
\usepackage{mathtools}
\usepackage{MnSymbol}

\usetikzlibrary{arrows,decorations.markings,plotmarks}

\usepackage{graphicx}
\newtheorem{mydef}{Definition}[section]
\newtheorem{thm}[mydef]{Theorem}
\newtheorem{lem}[mydef]{Lemma}
\newtheorem{prop}[mydef]{Proposition}
\newtheorem{cor}[mydef]{Corollary}

\newtheorem{rem}[mydef]{Remark}

\newtheorem{problem}[mydef]{Problem}

\def\Aut{\mathop{\rm Aut}\nolimits}

\begin{document}

\maketitle
\pagenumbering{arabic}
\begin{abstract}
We construct via Fra\"iss\'e amalgamation an $\omega$-categorical structure whose automorphism group is an infinite oligomorphic Jordan primitive permutation  group preserving a `limit of $D$-relations'. The construction is based on  a semilinear order whose elements are labelled by sets carrying a $D$-relation, with strong coherence conditions governing how these $D$-sets are inter-related.\\

\end{abstract}

\section{Introduction}

A transitive permutation group $G$ on a set $X$ is a {\em Jordan group} if there is $Y\subset X$ with $|Y|>1$ such that the pointwise stabiliser $G_{(X\setminus Y)}$ is transitive on $Y$, together with a non-degeneracy condition which essentially says that this transitivity does not arise just from the degree of transitivity of $G$ on $X$. It follows already from work of Jordan in 1871 that {\em finite} primitive Jordan groups are 2-transitive, and this led to a full classification of such permutation groups by Neumann in \cite{neumann1985some},
with related work around the same time by Kantor in \cite{kantor1985homogeneous},  and by Cherlin, Harrington and Lachlan in \cite{cherlin1985aleph}. For infinite permutation groups, the supply of examples of Jordan groups is much richer -- for example, Aut$({\mathbb Q},<)$ is a primitive but not 2-transitive Jordan group (any proper non-empty open interval is a Jordan set). Other examples include the supergroups of Aut$({\mathbb Q},<)$ in ${\rm Sym}({\mathbb Q})$ -- the ones which are closed in the topology of pointwise convergence on Sym$({\mathbb Q})$ were classified by Cameron in \cite{cameron1976transitivity}. Many other examples arise as automorphism groups of the `treelike' structures explored in \cite{adeleke1998relations}. In addition, there are examples, suggested by finite permutation group theory, consisting of projective and affine groups in their natural actions. Other examples of this type include automorphism groups of  saturated strongly minimal sets (or more generally regular types) arising in model theory. There are also Jordan groups which are {\em highly transitive} (that is, $k$-transitive for all $k$), but little work has been done on these -- the focus has been on Jordan groups which arise as automorphism groups of non-trivial first-order structures.

A  structure theory for infinite primitive Jordan groups has emerged. In 1985, Neumann \cite{neumann1985some} classified the primitive Jordan permutation  groups with cofinite Jordan sets. Primitive Jordan groups with a proper {\em  primitive} Jordan set (that is, the pointwise stabiliser of the complement acts {\em primitively} on the set) were classified by Adeleke and Neumann in \cite{adeleke1996primitive} -- here `classified' means that it was shown that any such group preserves a relational structure of one of a list of types. Finally, in \cite{adeleke1996classification}, the following theorem was proved -- see Definition~\ref{limits} and other definitions in Section 2.
\begin{thm} \label{am-class}
Let $G$ be a primitive but not highly transitive Jordan group on an infinite set $X$. Then $G$ preserves on $X$ a structure of one of the following kinds.
\begin{enumerate}[(i)]
\item A Steiner system.
\item A linear order, circular order, linear betweenness relation, or separation relation.
\item A semilinear order, general betweenness relation, $C$-relation, or $D$-relation.
\item A limit of Steiner systems, general betweenness relations, or $D$-relations.
\end{enumerate}
\end{thm}

The examples of types (i) and (iv) do not have a proper primitive Jordan set so did not arise in \cite{adeleke1996primitive}, and those of type (i) include projective and affine groups, and some constructions arising from strongly minimal sets and regular types.
The structures of type (ii) are essentially those classified by Cameron in \cite{cameron1976transitivity}, and those of type (iii) are described by Adeleke and Neumann in \cite{adeleke1998relations} -- in particular, the relational structures are axiomatised and well-understood.

 The examples of type (iv) are more mysterious, and are the focus of this paper. We do not give the definition of a limit of Steiner systems, but for a limit of $D$-relations (or of general betweenness relations) see Definition \ref{limits} below and the remark following it.
There is an example of an infinite Jordan group preserving a limit of Steiner systems given by Adeleke in \cite{adeleke1995semilinear}. This is developed further by Johnson in \cite{johnson2002constructions}, where for every $k\geq 2$ there is a construction of a $k$-transitive but not $(k+1)$-transitive example. An example of an infinite primitive Jordan group preserving a limit of betweenness relations is given by Bhattacharjee and Macpherson in  \cite{bhattmacph2006jordan}.  The group acts on an $\omega$-categorical structure which is built by a Fra\"iss\'e construction. Another example of an infinite Jordan permutation group preserving a limit of betweenness relations is given by Adeleke in his work \cite{adeleke2013irregular} (work which, despite its later publication date,  was done much earlier than \cite{bhattmacph2006jordan}, in the early 1990s, and which inspired \cite{bhattmacph2006jordan} and the present work).  Adeleke in \cite{adeleke2013irregular} also gives an example of an infinite primitive Jordan permutation group preserving a limit of $D$-relations.

Recall that a countably infinite first order structure is {\em $\omega$-categorical} if it is determined up to isomorphism by its cardinality and its first order theory. By the Ryll-Nardzewski Theorem, this is equivalent to its automorphism group being {\em oligomorphic}, that is, having finitely many orbits on $k$-tuples for all $k$. The group preserving a limit of betweenness relations constructed by Adeleke in  \cite{adeleke2013irregular} is not oligomorphic, but that constructed in   \cite{bhattmacph2006jordan} is. It is expected, but not verified, that the group preserving a limit of $D$-relations constructed  in Adeleke  \cite{adeleke2013irregular} is not oligomorphic. Adeleke and Macpherson, in the end of their paper \cite{adeleke1996classification}, posed the problem of explicitly  classifying  {\em oligomorphic} primitive Jordan permutation  groups, and also asked whether it is possible for an infinite primitive oligomorphic Jordan permutation group to preserve a limit of betweenness relations or $D$-relations. With Theorem~\ref{mainthm} below, together with that in \cite{bhattmacph2006jordan}, a  positive answer  has now been found in both cases.  Furthermore, in the Adeleke paper \cite{adeleke2013irregular}, the Jordan group is built as a direct limit of an increasing chain of permutation groups, but no invariant relational structure is made explicit. In our construction here, the Jordan group is the automorphism group of a relational structure which can reasonably be claimed to be a `new' treelike structure, essentially distinct from those occurring Theorem~\ref{am-class}(iii) or described in \cite{adeleke1998relations}. 

 Our main theorem is the following. The overall strategy of the proof of Theorem~\ref{mainthm} is analogous to that in \cite{bhattmacph2006jordan}, but there are significant differences.

\begin{thm}\label{mainthm} 
There is an $\omega$-categorical structure $M$ whose automorphism group is a primitive Jordan group which preserves a limit of $D$-relations
but does not preserve a structure of types (i), (ii), or (iii) of Theorem~\ref{am-class}.
\end{thm}

Some introductory background is given in Section \ref{Definitions}. 
 In Section \ref{trees of D-sets}, we build a class of finite structures, each of which is essentially a finite lower semilinear order with vertices labelled by finite graph-theoretic unrooted trees, with coherence conditions. These are viewed as structures in a relational language with relations $L,L',S,S',Q,R$. We describe possible one-point extensions of such structures,  prove an amalgamation theorem, and thereby obtain by  Fra{\" {\i}ss\'{e}}'s Theorem a  countably infinite structure  $M$ (the `Fra{\" {\i}ss\'{e}} limit').    
In Section \ref{Fraisse}, we describe in detail the structure $M$ and its automorphism group. We show that there is an associated dense lower semilinear order (a meet semilattice), again with vertices labelled by (dense) $D$-sets, again with coherence conditions.
Adapting an iterated wreath product construction described by Cameron in \cite{cameron1987some} which is based on Hall's wreath power, we show in Section \ref{Jordan} that $\text{Aut}(M)$  is a Jordan group with a `pre-direction' as a Jordan set.  Then we find, by properties of Jordan sets, that a `pre-$D$-set' is also a Jordan set for Aut$(M)$. Finally we prove that the Jordan group $G = \text{Aut}(M)$ preserves a limit of $D$-relations, the main result of this paper.

We believe that our construction, and its companion in \cite{bhattmacph2006jordan}, opens the possibility to give a much more enlightening  description of type (iv) in Theorem~\ref{am-class} by requiring that there is an invariant combinatorial structure satisfying certain explicit axioms. The constructions may also have interest for other test questions on homogeneous and $\omega$-categorical structures, and may be open for further generalisation. This is explored briefly in Section 6.

   

We conclude with some remarks concerning the wider context and motivation. Structural results on Jordan groups have had a number of applications. First, Cherlin, Harrington and Lachlan in \cite{cherlin1985aleph} used structural results on finite Jordan groups in model theory to classify $\omega$-categorical {\em strictly minimal sets}, and thereby to develop a powerful structure theory for $\omega$-categorical $\omega$-stable structures -- this paper was fundamental to the development of geometric stability theory in model theory. Neumann  \cite{neumann1985some} used essentially the same result to describe primitive permutation groups on a countably infinite set which have no countable orbits on the set of infinite co-infinite subsets. The paper \cite{adeleke1996infinite} uses results on primitive Jordan groups with primitive proper Jordan sets to obtain structural results on primitive groups on an uncountable set which contain a non-identity element of `small' support. This is analogous to the result of Wielandt that an infinite primitive permutation group with a non-identity element of finite support contains the finite alternating group, and Macpherson and Praeger in \cite{macpherson-praeger-cycle} used the full structure theory for primitive Jordan groups to show that a primitive permutation group realising a certain cycle type (a  single infinite cycle, finitely many and at least one non-trivial finite cycles, and infinitely many fixed points) must be highly transitive. Several authors have used Jordan groups to show that certain automorphism groups are `maximal-closed' in the symmetric group: Kaplan and Simon \cite{kaplan2016affine} showed that ${\rm AGL}_n({\mathbb Q})$ (for $n\geq 2$) and ${\rm PGL}_n({\mathbb Q})$ (for $n\geq 3$) are maximal closed; Bradley-Williams in \cite{bradleywilliams} described the closed supergroups of the automorphism group of certain semilinear orders, and Bodirsky and Macpherson \cite{bodirsky-macpherson} exhibited an uncountable {\em non-oligomorphic} maximal-closed permutation group acting on a countable set.

Semilinear orders, $C$-relations, general  betweenness relations, and $D$-relations can naturally be viewed as `treelike'. The classification in \cite{adeleke1996primitive} of infinite primitive Jordan groups with primitive proper Jordan sets suggests that these are the only treelike structures. However, we would argue that the highly symmetric structure $M$ constructed in Theorem~\ref{mainthm} (and its cousin in \cite{bhattmacph2006jordan}) involves all the above structures, but its automorphism group does not preserve any of the above structures, and thus it can claim to be a new treelike structure. This makes it potentially interesting in other ways -- see for example Problem~\ref{prob7} below. 

The methods in the paper are mainly combinatorial and permutation group-theoretic. We assume familiarity with some basic concepts from model theory such as relational structures, amalgamation and Fra\"iss\'e limits, $\omega$-categoricity, but give some explanation -- see e.g. Theorem~\ref{ryll}, Definition~\ref{defnice} and Theorem~\ref{general}.

The research in this paper was the main part of the PhD thesis \cite{asmathesis} at the University of Leeds by the first author. This thesis was funded by Tafila Technical University in Jordan. The authors thank Meenaxi Bhattacharjee for very helpful initial  conversations around 2000. 

\section{Definitions}\label{Definitions}

Throughout the paper, we shall denote by $(G,X)$ a permutation group $G$ acting on a set $X$, and we say $X$ is a {\em $G$-space}. For a natural number $k$, a $G$-space $X$ is said to be {\em $k$-transitive} if $G$ is transitive on the set of ordered $k$-subsets of $X$. If $G$ is transitive on the set of {\em unordered} $k$-subsets of $X$, then it is called $k$-{\em{homogeneous}}. If $G$ is $k$-transitive (respectively, $k$-homogeneous ) on $X$ for every $k\in \mathbb N$, then $G$ is said to be  {\em{highly transitive}} (respectively, {\em{highly homogeneous}}).

   For $Y \subset X$, the {\em {setwise stabiliser}} of $Y$ in $G$ is denoted by  $G_{\{ Y \} }$, and the {\em pointwise stabiliser} of $Y $ in $G$ is $G_{(Y)}$. The stabiliser of a point $x \in X$ is denoted by $G_x$.
  
  A group $G$ acting on a set $X$ is said to be {\em oligomorphic} in its action on $X$ if $G$ has finitely many orbits on $X^k$, the set of all $k$-tuples of $X$, for every natural number $k$.
 For more about oligomorphic groups,  see \cite{cameron1990oligomorphic}. A structure $M$ is $\omega$-{\em categorical} if $M$ is countably infinite  and any countable structure $N$ which satisfies the same first order theory as $M$ is isomorphic to $M$. The connection between these two notions lies in the following theorem.
\begin{thm}\label{ryll}[Ryll- Nardzewski 1959, Engeler 1959, Svenonius 1959].
Let $ M$ be a countably infinite first order structure. Then $M$ is $\omega$-categorical if and only if  $\text{Aut}( M)$ is oligomorphic on $ M$.
 \end{thm}

\begin{mydef} \em Let $Y \cup Z  $ form a partition of a transitive $G$-space $X$ with $\mid Z \mid > 1$. If the pointwise stabiliser $G_{(Y)}$ of $Y$ in  $G$  is transitive on $Z$, then $Z$ is called a {\em Jordan set} for $(G,X)$ and $Y$ is called a {\em Jordan complement}. The Jordan set $Z$ is {\em improper} if, for some $k\in {\mathbb N}$,  $(G,X)$ is $(k+1)$-transitive and $|Y|=k$; it is {\em proper} otherwise. We say that  $Z$ is a {\em primitive Jordan set} if $G_{(Y)}$ is primitive on $Z$, and an {\em imprimitive} Jordan set otherwise. A {\em Jordan group} is a transitive permutation group with a proper Jordan set. 
\end{mydef}

The following definition is taken from  \cite{adeleke1996primitive}. The  subsequent lemma is  heavily used in the classification results in \cite{adeleke1996primitive} and \cite{adeleke1996classification}, since many arguments apply properties of the family of all Jordan sets, or of an orbit on Jordan sets.
\begin{mydef} {\label{connected} \em{
\begin{enumerate}[(a)]
    \item A {\em{typical pair}} is a pair of subsets $Y_1, \ Y_2$ of $X$ such that $Y_1 \not \subseteq Y_2, \ Y_2 \not \subseteq Y_1$, and $Y_1 \cap Y_2 \neq \emptyset$.
    \item  A family of sets $\{ Y_i: \ i \in I \}$ will be said to be {\em{connected}} if for any $i, i' \in I$ there exists $j_0, \dots, j_l \in I$ such that $j_0=i, j_l= i'$ and $Y_{j_{r-1}} \cap Y_{j_r} \neq \emptyset$ for all $1 \leq r \leq l$.
\end{enumerate}
     }}
\end{mydef}

\begin{lem} \label{typicalpair}
\begin{enumerate}[(i)]
\item \cite[Lemma 3.2]{adeleke1996primitive}  \label{connected Jordan}
Suppose that $(G,X)$ is a transitive  $G$-space and that $\{ Z_i: \ i\in I\}$ is   a connected system of Jordan sets. Then $\bigcup_{i\in I} Z_i$ is a Jordan set for $(G,X)$.
\item \label{typical}(\cite{adeleke1996primitive}, Lemma 3.1) The union of any typical pair of Jordan sets is a Jordan set.
\end{enumerate}
\end{lem}

We now introduce, very briefly, some of the relational structures that arise in this paper. For further details see \cite{adeleke1998relations}.

First, recall  (adopting the conventions of \cite{adeleke1996classification}) that an {\em $n$-Steiner system} on $X$ is a family ${\mathcal B}$ of subsets of $X$ called {\em blocks}, all of the same size (possibly infinite), such that any $n$ distinct elements of $X$ lie in a unique block. We shall assume $n$-Steiner systems to be {\em non-trivial} in the sense that there is more than one block, and blocks have size greater than $n$. Jordan groups arising from projective and affine groups in their natural actions preserve Steiner systems. 

Recall also that a {\em separation relation} (see Cameron \cite{cameron1976transitivity}) is the natural arity 4 relation induced on a circularly ordered set indicating that two elements lie in distinct segments with respect to two other elements. 

  Let $(X, \leq)$ be a partially ordered set. Then $X$ is said to be a {\em (lower) semilinearly ordered} set if for any $a$ in $X$ the set 
$\{x\in X:x\leq a\}$ is totally ordered by $\leq$,
 any two elements have a common lower bound, but the set $X$ itself is not totally ordered. Given a lower semilinear order $(X,\leq)$, let $p\in X$ and put
 $Y_p:=\{x \in X: x>p\} $. Define an equivalence relation $E_p$ on $Y_p$, putting $ x E_p y \Leftrightarrow \exists z (p <z\leq x \wedge p<z\leq y).$ 
Then $E_p$ is preserved by $({\text{Aut}(X, \leq))}_p$. The equivalence classes of the equivalence relation $E_p$ at the point $p$ are called the {\em cones} at $p$.

From now on, by {\em semilinear order} we always   mean a {\em lower} semilinear order.

\begin{mydef} \label{Drelation} \em
A quaternary relation $D(x, y; z, w)$ on $X$ is a {\em D-relation} if for all $x,y,z,w \in X$ (D1)-(D4) hold:
\begin{enumerate}[(D1)]
    \item $D(x, y; z, w)\Rightarrow D(y, x; z, w) \wedge D(x, y; w, z) \wedge D(z,w; x, y)$;
    \item  $D(x, y; z, w)\Rightarrow \neg D(x, z; y, w);$
    \item $D(x, y; z, w) \Rightarrow (\forall a \in X ) (D(a, y; z, w) \vee D(x, y; z, a))$;
    \item $(x\neq z \wedge y\neq z) \Rightarrow D(x, y; z, z )$;
\end{enumerate}
We say it is a {\em proper} $D$-set if in addition (D5) holds.\\
\noindent
 (D5) $(x, y, z$ distinct) $\Rightarrow(\exists t) (t \neq z \wedge D(x, y; z, t))$.\\
\noindent
The $D$-set is said to be {\em dense} if \\
\noindent
(D6) $ D(x, y;z, w) \Rightarrow (\exists a \in X)D(a, y; z, w)\wedge D(x, a; z, w) \wedge D(x, y; a, w) \wedge D(x, y; z, a)$.

\end{mydef}

There are further tree-like structures that we mention without detail, as they play a more peripheral role here; for example, a {\em{general betweenness relation}} is, informally, a ternary relation $B(x;y,z)$ on a set $X$ which expresses that $x$ lies on the path between $y$ and $z$ (we usually omit the word `general'). If $(X,\leq)$ is a lower semilinear order, then one can define a general betweenness relation $B$ on $X$, putting $B(x;y,z)$ for any $x,y,z \in X$ if one of the following holds:
\begin{enumerate}[(i)]
    \item $y \geq x \wedge \neg(z \geq x)$.
    \item $ z \geq x \wedge \neg (y\geq x)$.
    \item $x=\text{glb}\{y,z\}$, where glb denotes the {\em{greatest lower bound}} (if it exists). 
\end{enumerate} 

If $T$ is an unrooted graph-theoretic tree, then there is a natural general betweenness relation
on its vertices -- $B(x;y,z)$ holds if and only if $x$ lies on the $yz$-geodesic. It is easy to imagine an analogous relational structure with a notion of betweenness where edges are replaced, for example, by the real interval [0,1]; and indeed, an ${\mathbb R}$-tree carries a natural general betweenness relation defined via geodesics as above -- but a set with a general betweenness relation does not in general have any automorphism-invariant metric.

A $C$-{\em{relation}} is a ternary relation which can be viewed as  describing the behaviour of the maximal chains of a semilinear order $(X,\leq)$: if $x,y,z$ are maximal chains of $X$ then $C(x;y,z)$ holds if $x \cap y = x \cap z \subset y \cap z$. Much more detail, including axioms,  can be found in \cite{adeleke1998relations}, and there is an overview also in \cite{adeleke1996classification}. Note that if $(Y,C)$ satisfies the axioms of a $C$-relation, then there is a lower semilinear order $(X,\leq)$ such that $Y$ can be identified with a `dense' set of maximal chains of $X$ with $C$ interpreted as above -- here the density means that every $a\in X$ lies in some maximal chain of $Y$.

We remark that if $T$ is a finite graph-theoretic tree then there is a $D$-relation on the set of leaves of $T$: put $D(x,y;z,w)$ if
$x=y\not\in \{z,w\}$ or $z=w\not\in \{x,y\}$ or $x,y,z,w$ are distinct and the path from $x$ to $y$ is disjoint from the path from $z$ to $w$. If $T$ is an infinite tree then there is a similar definition of a $D$-relation on the set of {\em ends} of $T$. If $B$ is a general betweenness relation on $X$ then there is a concept of {\em direction} of $(X,B)$, analogous to an end,  and corresponding $D$-relation on the set of directions (see Section 16 and Theorem 23.2 of \cite{adeleke1998relations}). Conversely, if $D$ is a $D$-relation on $X$ then it is possible to interpret a general betweenness relation in the structure $(X,D)$ - see \cite[Theorem 25.3]{adeleke1998relations}. If $(X,C)$ is a $C$-relation then there is a natural $D$-relation on the set of elements of $X$:  for example, $$ (\forall x, y, z, w \in X)D(x, y; z, w)\Leftrightarrow (C(x; z, w) \wedge C(y; z, w)) \vee (C(z; x, y) \wedge C(w;x, y))$$   See \cite{adeleke1998relations}, Theorem 23.4, and Theorem 23.5.

We now introduce the definition of a {\em  limit of $D$-relations}, to give meaning to Theorems~\ref{am-class}(iv) and \ref{mainthm}. We have reversed the ordering on $J$ compared to presentations given previously -- this seems to fit more naturally with our construction. 
\begin{mydef}(\cite{adeleke1996classification}, Definition 2.1.9) \label{limits} \em If $(G, X)$ is an infinite Jordan group we say that $G$ preserves a {\em limit of $D$-relations} if
there are: a linearly ordered set $(J, \leq)$ with no least element, a  chain $(Y _i :i\in J)$
of subsets of $X$ and  chain $(H_i: i\in J)$ of subgroups of $G$ with $Y_i\supset Y_j$ and $H_i>H_j$ whenever $i<j$, such that the following hold:
\begin{enumerate}[(i)]
    \item for each $i, H_i =G_{(X \backslash Y_i) }$, and $H_i$ is transitive on $Y_i$ and has a unique non-trivial
maximal congruence $\sigma_i$ on $Y_i$;
\item for each $i$, $(H_i, Y_i / \sigma_i)$ is a 2-transitive but not 3-transitive Jordan group preserving a $D$-relation;
\item $\bigcup(Y_i: i \in J)= X$;
\item  $(\bigcup(H_i: i \in J), X)$ is a 2-primitive but not 3-transitive Jordan group;
\item $\sigma_j\supseteq \sigma_i |_{Y_j}$ if $i< j$;
\item $\bigcap(\sigma_i: i\in J)$ is equality in $X$;
\item $(\forall g\in G) (\exists i_0 \in J) (\forall i<i_0)(\exists j\in J) (Y_i ^g=Y_j \wedge g^{-1} H_{i} g= H_j)$;
\item for any $x \in X, G_x$ preserves a $C$-relation on $X\setminus \{ x \}$.
\end{enumerate}

\end{mydef}

The notion of preserving a limit of general betweenness relations is essentially the same, but with a general betweenness relation replacing the $D$-relation in (ii).  
Note that we do not define limits of Steiner systems, since the concept is not used here.

 \section{Trees of D-sets}\label{trees of D-sets}

In this section we construct the $\omega$-categorical structure $M$ whose automorphism group preserves a limit of $D$-relations. The structure $M$ is  a Fra\"iss\'e limit of a class of finite structures (`trees of $D$-sets'), which, informally, may be viewed as rooted lower semilinear orders with each vertex labelled by a finite $D$-relation (so essentially by a finite graph-theoretic tree) with additional coherence conditions. We first introduce the key concept of a finite {\em tree of $D$-sets}.

{\bf{Notation.}} 
  Let $(T,\leq)$ be a finite lower semilinear order with a root $\rho$. Label each vertex $\nu$ of $ T$ by a finite $D$-set $D(\nu)$ with a $D$-relation $D_\nu$ defined on $D(\nu)$. We view $D(\nu)$ as the set of leaves of a finite unrooted tree $\overline{D(\nu)}$ (in the graph-theoretic sense) without dyadic vertices (vertices of degree $2$), and with $D_\nu$ defined in the natural way described in Section 2. (The correspondence between finite $D$-sets and such trees follows from Proposition 3.1 and Section 9 in \cite{cameron1987some}, noting that the author uses different notation for the relation $D$.) We refer to vertices  of $\overline {D(\nu)}$ as {\em nodes}, and those of degree at least three are called {\em ramification points}, with the set of these denoted by ${\rm Ram}(\overline{D(\nu)})$; if the ramification point $r$ lies on the geodesic between any two of the distinct nodes $x,y,z$, we write $r={\rm ram}(x,y,z)$, and any three distinct leaves of $D(\nu)$ determine a unique such ramification point, such that the $xy$-path, the $xz$-path, and the $yz$-path all pass through ram$(x,y,z)$. 
By a {\em{successor}} of a vertex $\nu \in T$ we mean a vertex $\mu\in T$ such that $\nu<\mu \wedge \neg \exists \lambda(\nu<\lambda<\mu)$;  we write succ$(\nu)$ for the set of successors of $\nu$.
For each ramification point $r$ of $\overline{D(\nu)}$ there is an equivalence relation $E_r$ on $D(\nu)$ such that two leaves $w_1, w_2$ of $D(\nu)$ are $E_r$-equivalent if the unique paths from $r$ to $w_1$ and from $r$ to $w_2$ have at least two common nodes (or equivalently, if the unique $w_1w_2$-path of $\overline{D(\nu)}$ does not pass through $r$). The $E_r$-classes will be called {\em{branches}} at $r$. For each $r \in \overline{D(\nu)}$, one of the branches at $r$ will be distinguished, and called the {\em{special branch at $r$}}.
  
  We shall use the Roman letters $x,y,z,w,u,v, \dots$ for leaves of a $D$-set, and the letters $r,r',r''$ or $ r_1, r_2, \dots$ for the ramification points. The Greek letters $\alpha, \nu, \mu, \dots$ refers to the vertices of the tree while we retain the letter $\rho$ for the root. This notation will persist in Section 4, where everything is infinite and the labelling graph-theoretic trees are replaced by general betweenness relations.

 For each $\nu\in T$, we assume there is a fixed  bijection $f_\nu :\text{succ}(\nu) \rightarrow \text{Ram} (D(\nu))$ from the set of successors of 
the vertex $\nu$ in $T$ to the set of ramification points of the $D$-set $D(\nu)$. For $r\in \text{Ram}(D(\nu))$, if $ \omega= f_\nu ^{-1} (r)$ then there is a bijection
$g_{\omega \nu} $ from the $D$-set  $D(\omega)$ to the set of {\em non-special} branches at $r$ (in the $D$-set $D(\nu)$).

Let $\nu_0, \dots, \nu_m$ be vertices of the semilinear order  $T$ such that $\nu_0< \dots < \nu_m$. Then $(\nu_0, \dots, \nu_m)$ is a {\em{chain of successors}} if  $\nu_{i+1} \in \text{succ}(\nu_i)$ for each $i \in \{0, \dots, m-1\}$.
Given the chain $(\nu_0, \dots, \nu_m)$, there is a map $g_{\nu_m \nu_0}$ which we define by induction such that it maps each leaf of the $D$-set $D(\nu_m)$ to a union of branches at a fixed ramification point of $D(\nu_0)$. 
Let $a\in D(\nu_m)$, define $$g_{\nu_m \nu_0}(a):= \{ x \in D(\nu_0): \exists y \in g_{\nu_m \nu_{m-1}}(a) (x\in g_{\nu_{m-1} \nu_0}(y)) \}.$$

The structure above, consisting of the labelled semilinear order and the maps $f_\nu$ and $g_{\mu \nu}$, will be called a (finite) {\em tree of $D$-sets}, and we use symbols $\tau, \tau'$ to denote such structures, and refer to $T$ as its {\em structure tree}. We have not yet described how to parse  a tree of $D$-sets as a first order structure. 

 Let $\tau, \tau '$ be two trees of $D$-sets. An {\em{isomorphism between trees of $D$-sets} } is an isomorphism between the corresponding two lower semilinear orders $\phi:(T,\leq) \rightarrow (T',\leq)$ together with, for any vertex $\nu\in T$, a graph isomorphism $\psi_\nu$ from $\overline{D(\nu)}$ to $\overline{D(\phi(\nu))}$. The maps $\psi_\nu$ are required to map the  special branch at any ramification point $r$ to the special branch at $\psi_\nu(r)$, and to commute with the maps $f_\nu$ and $g_{ \omega \nu}$.
 
If $\tau$ is a tree of $D$-sets with vertices $\mu<\nu$, then we say that the $D$-set $D(\nu)$ {\em omits} the element $u\in D(\mu)$ if there is no $x\in D(\nu)$ such that $u\in g_{\nu\mu}(x)$. If $\nu$ is an immediate successor of $\mu$ this means that $u$ lies in the special branch of the ramification point $f_\mu(\nu)$ of $D(\mu)$.

We shall view a finite tree of $D$-sets $\tau$ as a first order structure in a  language $\mathscr L$ which has a ternary relation $L$, two quaternary relations $L'$ and $S$, a 5-ary relation $S'$, a 6-ary relation $R$ and a 7-ary relation  $Q$. The universe of the structure will be 
on the domain of the root $D$-set of $\tau$  (i.e. the set of {\em leaves} of $\overline{D(\rho)}$, where $\rho$ as usual denotes the root of the structure tree), and the relations are interpreted on $D(\rho)$ as follows. 
 \begin{enumerate}[(i)]\label{enu}
\item $L(x;y,z)$ holds in $\tau$ if either
\begin{enumerate}
    \item \label{L1a} $ x , y, z$ lie in distinct branches at node $r$ of the root $D$-set $D(\rho)$, and the branch containing $x$ is special at $r$ (see Figure \ref{fig:L 1(x;y,z)}),  or
    \tikzset{middlearrow/.style={
        decoration={markings,
            mark= at position 0.75 with {\arrow{#1}} ,
        },
        postaction={decorate}
    }
}
\begin{figure}[H]
    \centering
   \begin{tikzpicture}

[scale=1]
\node at (0,0)[below]{$r$};
\node (x) at (1,1){$x$};
\node (y) at (1,-1) {$y$};
\node (z) at (-1,0){$z$};

\draw [middlearrow={stealth}](x)--(0,0);
\draw (y)--(0,0);
\draw (z)--(0,0);
    
    \end{tikzpicture}
    \caption{$L(x;y,z)$}
    \label{fig:L 1(x;y,z)}
\end{figure}

\item \label{L1b} there is a $D$-set $D(\nu)$ with a ramification point $r$, and leaves $\bar x, \bar y, \bar z$ lying in distinct branches at $r$ with  $\bar x $ lying in the special branch at $r$, such that $x \in g_{\nu\rho}(\bar x), y \in g_{\nu\rho}(\bar y), z \in g_{\nu \rho}(\bar z)$ .
\end{enumerate}

We say in (\ref{L1a}) that $D(\rho)$ {\em{witnesses}} $L(x;y,z)$, and in (\ref{L1b}) that $D(\nu)$ {\em{witnesses}} $L(x;y,z)$. We use the semi-colon to distinguish the special branch in the first argument, while there is symmetry between the other two arguments. Note the diagrammatic convention introduced in Figure \ref{fig:L 1(x;y,z)}: the arrow on the path between the leaf $x$ and the ramification point $r$ indicates that $x$ lies in the special branch at $r$.

\item \label{Sdef}Let $x,y,z,w \in D(\rho)$ be distinct. Then $S(x,y;z,w)$ holds, written $\tau \models S(x,y;z,w)$, if one of the following holds

\begin{enumerate}
    \item \label{a} In the root $D$-set, with universe denoted  $D(\rho)$, and a $D$-relation denoted $D_\rho$ we have $D_\rho(x,y;z,w)$.
    
    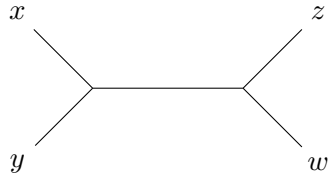
\begin{figure}[H]
        \centering
        \begin{tikzpicture}[scale=2]
\node(x) at (0,1){$x$};
\node(z) at (2,1){$z$};
\node(y) at (0,0){$y$};
\node(w) at (2,0){$w$};

\draw (x)--(0.5, 0.5)--(1.5,0.5)--(z);
\draw (y)--(0.5, 0.5);
\draw(1.5,0.5)--(w);
\end{tikzpicture}
        \caption{$S(x,y;z,w)$}
        \label{fig:S1(x,y;z,w)}
    \end{figure}

    \item \label{b} $x,y,z,w$ lie in distinct non-special branches at node $r$ of $D(\rho)$, and there is some vertex $\nu \geq f_\rho ^{-1}(r)$ such that $D(\nu)$ contains distinct $\bar x , \bar y, \bar z, \bar w$ such that $D_\nu( \bar x, \bar y; \bar z, \bar w)$ holds in $D(\nu)$, and $x \in g_{\nu\rho}(\bar x), y \in g_{\nu\rho}(\bar y), z \in g_{\nu\rho}(\bar z), w \in g_{\nu\rho}(\bar w) $.

\end{enumerate}

We say in (\ref{a}) that $D(\rho)$ {\em{witnesses}} $S(x,y;z,w)$, and in (\ref{b}) that $D(\nu)$ {\em{witnesses}} $S(x,y;z,w)$.


{\bf{Note.}}
1. It is easy to see that given a tree of $D$-sets $\tau$ and $x,y,z\in D(\rho)$, the  relation $L(x;y,z)$ can be witnessed in at most one $D$-set of $\tau$; to see this, observe that if it is witnessed in $D$-sets $D(\mu)$ and $D(\nu)$ then it cannot happen that $\mu<\nu$ or that $\mu$ and $\nu$ are incomparable. Likewise, if $x,y,z,w\in D(\rho)$ then $S(x,y;z,w)$ is witnessed in at most one $D$-set of $\tau$. We omit the details.

2.  The relation $S$ captures the behaviour of $D$-relations except that in $S$ we do not allow equality among its parameters, i.e. axiom $(D4)$ of Definition \ref{Drelation} does not hold for $S$. We use the semi-colon to reflect the symmetry between the first two arguments and the last two.


\item  $Q(x,y;z,w:p;q,s)$ holds in $\tau$  if there is some $D$-set in which the relations $S(x,y;z,w)$ and $L(p;q,s)$ are both witnessed. We interpret this as $S(x,y;z,w)$ and $L(p;q,s)${\em{ happen in the same $D$-set}}. 

\item $R(x;y,z:p;q,s)$ holds in $\tau$  if there is some $D$-set in which the relations $L(x;y,z)$ and $L(p;q,s)$ are both witnessed. Again, we interpret this as saying that the two $L$-relations {\em{happen in the same $D$-set}}.

\item $L' (x;y,z;u)$ holds in $\tau$  if in the $D$-set $D(\nu)$ witnessing $L(x;y,z)$, the element $u$ is omitted, that is, there is no $\bar{u}\in D(\nu)$ with $u\in g_{\nu\rho}(\bar{u})$. We then say that the $D$-set $D(\nu)$ {\em{witnesses}} the relation $L' (x;y,z;u)$.
We use the first semi-colon as in $L$ above, and the second one to distinguish the omitted element.
 \item $S'(x,y;z,w;t)$ holds in $\tau$  if in the $D$-set $D(\nu)$ witnessing $S(x,y;z,w)$, the element  $t$ is omitted (in the same sense as in (v)). We then say that the $D$-set $D(\nu)$ witnesses the relation $S'(x,y;z,w;t)$. 
 The second semi-colon indicates that the last argument is distinguished.
\end{enumerate}

Note that, by the definition, the relations $L'$, $S'$ cannot be witnessed in the root $D$-set.

\medskip

{\bf{Note}}: When we say that one of the above relations holds in the structure $A$ we mean it is witnessed in some $D$-set of $A$. We may thus view a finite tree of $D$-sets as an $\mathscr L$-structure whose universe is the set of directions of the root $D$-set. We use symbols like $A,B,C,E, \dots$ (rather than $\tau, \tau'$) for such finite $\mathscr L$-structures, and write $\tau_A$ for the corresponding structure tree (which will be seen in  Lemma~\ref{isotreeDset} to be determined up to isomorphism by $A$). Also we write $A<B$ if $A$ is an $\mathscr L$-substructure of $B$ in the sense of model theory. Occasionally, we write $L\{x,y,z\}$, as an abbreviation for $L(x;y,z)\vee L(y;x,z)\vee L(z;x,y)$, and we may say that $L\{x,y,z\}$ is {\em witnessed} in a specific $D$-set.\\

Let $\mathscr D$ be the collection of all finite $\mathscr L$-structures arising from finite trees of $D$-sets as described above. If $A\in \mathscr D$, we write $D_A$ (or just $D_\rho$) for the root $D$-set of $A$, sometimes not distinguishing between  the universe of this $D$-set and the $D$-relation.

\begin{lem}\label{3set}
Let $A\in \mathscr D$. and let $x,y,z$ be distinct elements of $A$. Then 
\begin{enumerate}[(i)]
\item If $x,y,z$ are distinct elements of $A$ then $L\{x,y,z\}$ holds in $A$, and
\item any substructure $A'$ of $A$ of size at most 3 lies in $\mathscr D$.
\end{enumerate}
\end{lem}
\begin{proof}
\begin{enumerate}[(i)]
\item Let $\mu$ be a vertex of the structure tree of $A$ maximal  such that there are distinct $\bar{x},\bar{y},\bar{z} \in D(\mu)$ with $x\in g_{\mu\rho}(\bar{x})$, $y\in g_{\mu\rho}(\bar{y})$ and $z\in g_{\mu\rho}(\bar{z})$. Then, by maximality, one of $\bar{x}, \bar{y},\bar{z}$ lies in the special branch at the ramification point ${\rm ram}(\bar{x},\bar{y},\bar{z})$,  and hence $D(\mu)$ withnesses $L\{x,y,z\}$. 
\item Suppose first $|A'|\leq 2$. Then the elements of $A'$ do not determine any ramification point of $D(\rho)$ (the root $D$-set of $A$). It follows that the structure on $A'$ corresponds to a tree of $D$-sets with just one vertex (the root $\rho$ of that of $A$) with the corresponding $D$-set consisting just of the vertices of $A'$ (with an edge between them if $|A'|=2$). If $|A'|=3$ with $A'=\{x,y,z\}$ then as in (i) we may suppose $A\models L(x;y,z)$. Now $A'$ is a structure arising from a tree of $D$-sets  with two vertices $\rho$ (the root) and its successor $\nu$. Here $D(\rho)$ has a ramification point $r$ joined to just the three leaves $x,y,z$ with $x$ special, and $D(\nu)$  consists of  just an edge joining the two vertices $g_{\nu\rho}^{-1}(y)$ and $g_{\nu\rho}^{-1}(z)$. 
\end{enumerate}
\end{proof}

\begin{lem}\label{D_rho}
Let $A\in \mathscr D$ have a root $\rho$. Then the relation $D_\rho$ on $D(\rho)$ satisfies the following:
\noindent
for all $x,y,z,w \in D(\rho), \   
 D_\rho (x,y;z,w)\Leftrightarrow [(((x=y) \vee (z=w)) \wedge \{x,y\} \cap \{z,w \}=\emptyset) \vee (x,y,\\ z,w $ are all distinct $\wedge S(x,y;z,w)\wedge(\forall t) (\neg S'(x,y;z,w;t)))].$
 \end{lem}
 \begin{proof} We may suppose that $x,y,z,w$ are distinct.

$\Rightarrow$. Suppose $D_\rho (x,y;z,w)$. Then $S(x,y;z,w)$ holds, witnessed in the root $D$-set. As $D(\rho)$ contains all elements of $A$ and is the only $D$-set witnessing $S(x,y;z,w)$, we have $(\forall t) \neg S'(x,y;z,w;t)$.
  
 $\Leftarrow$. Suppose $S(x,y;z,w)\wedge(\forall t) (\neg S'(x,y;z,w;t))$ holds. For a contradiction, we will assume that the $D$-set witnessing $S(x,y;z,w)$ is not the root. Then there is a lower $D$-set in which $x,y,z,w$ lie in distinct branches $\bar x, \bar y, \bar z, \bar w$ respectively at a ramification point $r$. As $S(x,y,z,w)$ is witnessed higher up, none of $x,y,z,w$ lies in a special branch at $r$. Since each ramification point has a special branch, some $t\in A$ lies in the special branch at $r$. Then $ S'(x,y;z,w;t)$ holds, which is a contradiction. 
\end{proof}
Consider an $\mathscr L$-structure $A\in \mathscr D$ with tree of $D$-sets $\tau$ with root $\rho$ and let $\nu$ be a successor of $\rho$. We define an $\mathscr L$-structure $A_\nu$ whose domain is the set of leaves of the $D$-set $D(\nu)$, that is, the set of non-special branches  in $D(\rho)$ at $f_\rho(\nu)$. To define the relations to $A_\nu$, suppose first that $\bar{a},\bar{b},\bar{c}\in A_\nu$ are distinct, and let $a\in g_{\nu\rho}(\bar{a})$, $b\in g_{\nu\rho}(\bar{b})$ and $c\in g_{\nu\rho}(\bar{c})$. Then 
$A_\nu\models L(\bar{a};\bar{b},\bar{c})$ if and only if $A\models L(a;b,c)$. It is easily checked that this is well--defined, i.e. independent of the choice of $a,b,c$. The relations $S,L',S', Q,R$ are defined similarly on $A_\nu$. If $r=f_\rho(\nu)$, we also sometimes write $A_\nu$ as $A_r$.

Given a structure tree $\tau$, we define the height $h(\tau)$  of $\tau$ to be the number of vertices in the longest path from a leaf of $\tau$ to the root $\rho$. If $\mu$ is a vertex of $\tau$, then $\tau_\mu$ is the subtree of $\tau$ induced on $\{\sigma \in \tau: \sigma \geq \mu\}$. If $A\in \mathscr D$ with structure tree $\tau$, we put $h(A):=h(\tau)$. 

\begin{lem}\label{indtreeDset}
Let $A\in \mathscr D$ with tree of $D$-sets $\tau$ having  root $\rho$ with a successor $\nu$. Then the following hold.
\begin{enumerate}[(i)]
\item $A_\nu$ is isomorphic to a substructure of $A$.
\item $A_\nu \in \mathscr D$, with structure tree $\tau_\nu$.
\item $h(\tau_\nu)<h(\tau)$.
\item $|A_\nu|<|A|$.
\end{enumerate}
\end{lem}

\begin{proof} All parts are elementary, and we omit the details. 
\end{proof}

\begin{rem} \rm It follows from the last lemma that the above construction can be iterated for a successor of $\nu$. Thus, inductively, for any vertex $\mu$ of the structure tree $\tau$ of $A$, there is a corresponding $\mathscr L$-structure $A_\mu$, and all parts of Lemma~\ref{indtreeDset} hold with $\mu$ replacing $\nu$. This is a convenient tool for inductive arguments.
\end{rem}

\begin{prop}\label{isotreeDset}
Suppose that $\tau_1,\tau_2$ are trees of $D$-sets with corresponding $\mathscr L$-structures $A_1,A_2$ and let $\chi:A_1\to A_2$ be an isomorphism. Then $\chi$ induces an isomorphism $\phi:\tau_1\to \tau_2$ of trees of $D$-sets. 
\end{prop}

\begin{proof} We apply induction on $h(\tau_1)$. Let $\rho_1,\rho_2$ be the roots of $\tau_1,\tau_2$ respectively, and put $\phi(\rho_1)=\rho_2$. 

 For the base case, suppose $h(\tau_1)=1$. Then $\tau_1$ has just the root $\rho_1$ and $D(\rho_1)$ has no ramification points, so at most two leaves. Thus $A_1$ consists of a set of size at most 2 with none of the $\mathscr L$-relations holding, and as $\chi$ is an isomorphism the same holds for $A_2$. Hence $\tau_2$ has one vertex $\rho_2$, and $\chi$ induces a unique isomorphism $\tau_1\to \tau_2$ and $D(\rho_1)\to D(\rho_2)$.

For the inductive step, suppose $m:=h(\tau_1)\geq 2$. By Lemma~\ref{D_rho}, $\chi$ determines an isomorphism of $D$-structures $D(\rho_1)\to D(\rho_2)$. This extends to a unique graph isomorphism (which we denote by $\bar{\chi}$) $\overline{D(\rho_1)}\to \overline{D(\rho_2)}$ taking the  ramification points of $\overline{D(\rho_1)}$ to the  ramification points  of $\overline{D(\rho_2)}$.

For each $r\in {\rm Ram}(D(\rho_1))$ put $\phi(f^{-1}_{\rho_1}(r))=f^{-1}_{\rho_2}(\bar{\chi}(r))$ to obtain a bijection ${\rm succ}(\rho_1)\to {\rm succ}(\rho_2)$. Let $\nu_1$ be the successor of $\rho_1$ corresponding to $r$ and $\nu_2$ the successor of $\rho_2$ corresponding to $\bar{\chi}(r)$. We claim that $\chi$ induces an isomorphism $\phi_\nu$ from the $\mathscr L$-structure $A_{\nu_1}$ to the $\mathscr L$-structure $A_{\nu_2}$. Indeed, $\bar{\chi}$ gives a bijection $D(\nu_1)\to D(\nu_2)$, and the fact that it is an isomorphism of $\mathscr L$-structures  follows from the definition of the $A_{\nu_i}$.

Since $h(A_{\nu_1})<h(A)$ (by Lemma~\ref{indtreeDset}(iii)), it follows by induction that $\chi$ induces  a unique  isomorphism from the tree of $D$-sets corresponding to $A_{\nu_1}$ to that corresponding to $A_{\nu_2}$. This holds for all successors of $\rho_1$, and the result follows. 
\end{proof}

\subsection {One-point extensions}\label{one point extension} 

Fix an $\mathscr L$-structure $A\in \mathscr D$. We want to specify the possible forms of a one-point extension $E=A\cup \{e\}$
of $A$ such that $E \in \mathscr D$ and $A$ is an $\mathscr L$-substructure of $E$. We  first describe some one-point extensions.\\ 
{\bf Type \Romannum{1} }(Star-like): To obtain $\tau_E$, which is the structure tree on the $\mathscr L$-structure $E$, from $\tau_A$, we add a new root $\rho _E$ under the root $\rho_A$ of the structure tree $\tau_A$, such that $D(\rho_E)$ looks like a star with one ramification point (the centre) and non-special branches each containing a single leaf and corresponding to the leaves in the root $D$-set $D(\rho_A)$ of $A$, and a special branch $e$. We shall use  the word {\em star} to describe a tree of this form (a node connected to a finite collection of leaves).  See Figure 3 for an example.
\tikzset{middlearrow/.style={
        decoration={markings,
            mark= at position 0.75 with {\arrow{#1}} ,
        },
        postaction={decorate}
    }
}
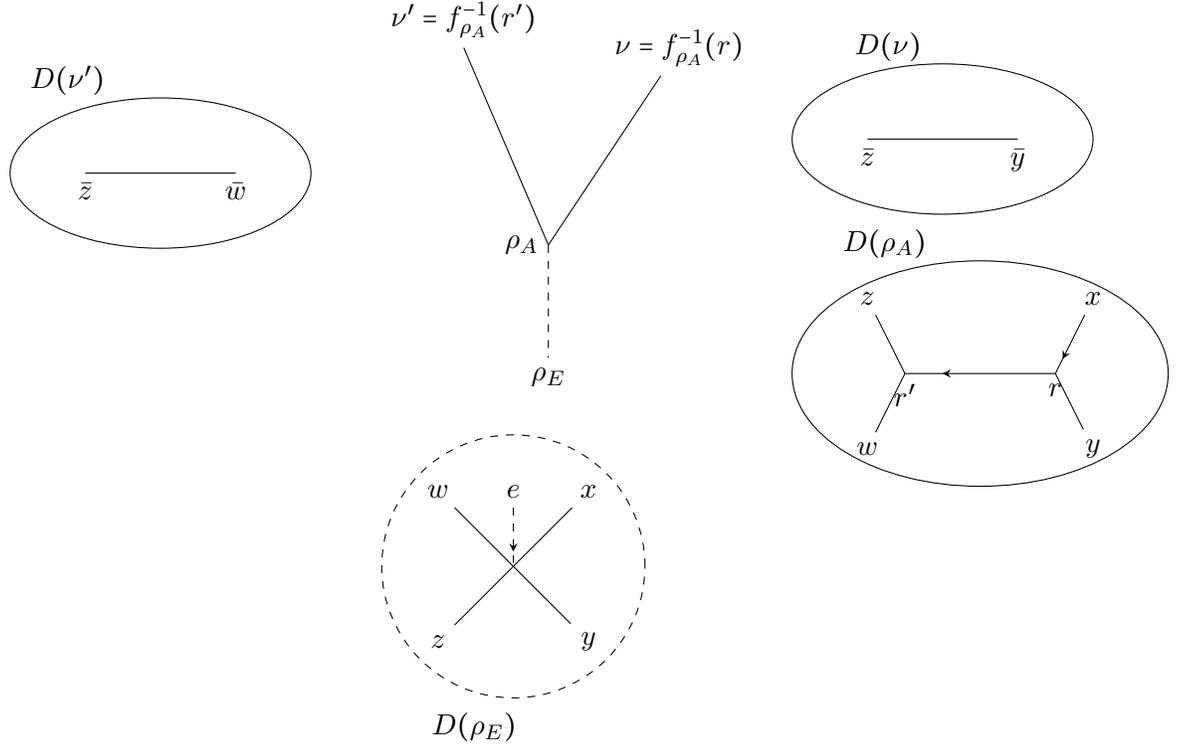
\begin{figure}[H]
     \centering
     \begin{subfigure}[b]{0.3\textwidth}
         \centering
        
     \begin{subfigure}[b]{0.33\textwidth}
         \centering
         $D(\nu ')$
         \begin{tikzpicture}[]
     \draw (0,0) ellipse (2cm and 1cm);
    \node (w) at (1,0)[below]{$\bar w$};
    \node (z) at (-1,0)[below]{$\bar z$};
    
    \draw (1,0)--(-1,0);

    \end{tikzpicture}
    \qquad
     \end{subfigure}

           \end{subfigure}
     \hfill
     \begin{subfigure}{0.3\textwidth}
         \centering
         \begin{subfigure}[b]{0.55\textwidth}
         \centering
         
         \begin{tikzpicture}
         [scale=0.75]
\node (x) at (0,0)[left]{$\rho_A$};
\node at (0,-2) [below] {$\rho_E$};
\node (v) [right] at (1,3.5){$\nu=f^{-1}_{\rho_A}(r)$};
\node at (3.5,7){};

\node at (-0.5,5.25){};
\node at (-0.5,7){};

\node at (2,3.5){};
\node at (1.75, 6.125){};
\node at (1,8){};
\node at (-1.5,3.5)[above]{$\nu '= f_{\rho_A}^{-1}(r')$};

\draw [dashed] (0,0)--(0,-2);
\draw (v)--(0,0);
\draw (-1.5,3.5)--(0,0);
    
    \end{tikzpicture}
            
        \qquad
     \end{subfigure}
          \begin{subfigure}[b]{0.55\textwidth}
 \centering
        \begin{tikzpicture}[]
    \draw (0,0) [dashed] circle (1.75cm);
 \node (x) at (1,1){$x$};
 \node (y) at (1,-1){$y$};
 \node (z) at (-1,-1){$z$};
 \node (w) at (-1,1){$w$};
 \node (e) at (0,1){$e$};
   
 \draw  [middlearrow={stealth}][dashed] (e)--(0,0);
  \draw (x)--(0,0);
 \draw (y)--(0,0);
 \draw (z)--(0,0);
 \draw (w)--(0,0);
    
 \end{tikzpicture}
   $D(\rho_E)$
 \qquad
 \end{subfigure}

         \begin{tikzpicture}
\end{tikzpicture}
        
     \end{subfigure}
     \hfill
     \begin{subfigure}{0.3\textwidth}
         \centering
         
         \begin{subfigure}[b]{0.55\textwidth}
         \centering
         $D(\nu)$
         \begin{tikzpicture}
         
     \draw (0,0) ellipse (2cm and 1cm);
    \node (y) at (1,0)[below]{$\bar y$};
    \node (z) at (-1,0)[below]{$\bar z$};
    
    \draw (1,0)--(-1,0);

    \end{tikzpicture}
            $D(\rho_A)$
        \qquad
     \end{subfigure}
         
          \begin{subfigure}[b]{0.55\textwidth}
         \centering
         \begin{tikzpicture}
      \draw (0,0) ellipse (2.5cm and 1.5cm);
    \node (r) at (1,0)[below]{$r$};
    \node (r') at (-1,0)[below]{$r'$};
    \node (x) at (1.5,1){$x$};
    \node (y) at (1.5,-1){$y$};
    \node (z) at (-1.5,1){$z$};
    \node (w) at (-1.5,-1){$w$};
   
    \draw  [middlearrow={stealth}] (1,0)--(-1,0);
    \draw  [middlearrow={stealth}] (x)--(1,0);
    \draw (y)--(1,0);
    \draw (z)--(-1,0);
    \draw (w)--(-1,0);
    
    \end{tikzpicture}
   
    \qquad
     \end{subfigure}
     
     \end{subfigure}
        \caption{One-point extension:Type \Romannum{1}}
        \label{Type 1}
\end{figure}

Since there is only one ramification point in $D(\rho_E)$, it will have the form $f_{\rho _E}(\rho_A)$, where $\rho_A$ is the immediate successor of $\rho_E$. The $D$-set $D_E = D(\rho_E)$ is a star whose centre is $f_{\rho_{E}} (\rho_A)$ where the branches (here identified with leaves) are of the form $g_{\rho_A\rho_E}(x),\ x$ a leaf in $D(\rho_A)$. The relations on $A$ will also hold in $E:=A \cup \{e \}$. Thus, if $a,b,c \in A$ and $A \models L(a;b,c)$ then $E \models L(a;b,c)$; however this is not witnessed in the root $D$-set of $E$, and indeed, $E\models L'(a;b,c;e)$. Likewise if $a,b,c,d, \in A$ and $A\models S(a,b;c,d)$, then $E\models S'(a,b;c,d;e)$.\\ 

{\bf Type \Romannum{2} }: In this type, we assume that the two structure tree  roots for the two structures $A$ and $E$ are the same, denoted by $\rho$, and we add the new leaf $e$ to the existing $D$-set $D_A$ of the root $\rho$ of $\tau_A$ to obtain the root $D$-set $D_E$ of $E$. We can do this in two ways:
\begin{enumerate}[(a)]
    \item Add a new leaf $e$ adjacent to an existing ramification point $r$ in $\overline{D(\rho)}$, with $e$ non-special at $r$. So the $D$-set $D(\nu)$ corresponding to that ramification point  (i.e. with $\nu=f_\rho^{-1}(r)$) gains a new 
leaf, namely $g_{\nu\rho}^{-1}(e)$. This process iterates through the structure tree, so the definition is inductive on $|A|$. 
\item Create a new ramification point by adding a  node on an existing edge in $\overline{D(\rho)}$, then add a leaf $e$ at this node.
Here we consider two cases:\\
 (\romannum{1}) $e$ is the special branch at this new ramification point.\\
 (\romannum{2}) $e$ is not the special branch at this ramification point.\\
In  both cases a new successor is added to the structure tree, but the $D$-set labelling the new successor has just two endpoints,
and hence there are no modifications higher in the structure tree.
\end{enumerate}

The following lemma is almost immediate and we omit the proof.
\begin{lem} \label{in D}
If $A\in \mathscr D$ and $E$ is a one-point extension of $A$ of Type I or Type II, then $E\in \mathscr D$.
\end{lem}

\begin{lem} \label{lema} If $A, E \in \mathscr D$ with $A < E$, and $a,b,c,d \in A$ are all distinct elements, then $D_E(a,b;c,d)\rightarrow D_A(a,b;c,d)$.
\end{lem}
\begin{proof}
As $a,b,c,d$ are distinct, Lemma~\ref{D_rho} yields  $D_E(a,b;c,d)\Rightarrow S(a,b;c,d) \wedge(\forall t\in E)\neg S'(a,b;c,d;t)\Rightarrow $\\
              $S(a,b;c,d) \wedge(\forall t\in A)\neg S'(a,b;c,d;t)\Rightarrow D_A(a,b;c,d)$.
              
\end{proof} 



\begin{lem} \label{lemb} Suppose $A, E \in \mathscr D$ with $A<E$, and there is no $e \in E$ such that $A\cup \{ e \}$ is a Type $\Romannum{1}$ extension of $A$. Then the root $D$-relation $D_A$ of $A$, and the relation $D_E(A)$ induced on $A$ by the root $D$-relation $D_E$ of $E$, are the same.  
 \end{lem}
 {\bf{Note.}} We do not here assume that $\lvert E \setminus A \rvert= 1$.
 
 \begin{proof} Let $a,b,c,d \in A$ and assume $D_E(A)(a,b;c,d)$. Then $D_E(a,b;c,d)$. We may suppose that $a,b,c,d$ are distinct. By Lemma \ref{lema} $D_A(a,b;c,d)$.
 
Conversely, let $a,b,c,d \in A$ are distinct, and suppose that $D_A(a,b;c,d)$ but $\neg D_E(A)(a,b;c,d)$.  Let $r,r'$ be as in Figure~\ref{fig:D_A(a,b;c,d)}  in $D_A$. 
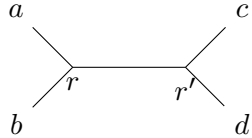
\begin{figure}[H]
\centering
\begin{tikzpicture}[scale=1.5]
\usetikzlibrary{calc}
\node [below](alpha) at (0.5,0.5){$r$};
\node [below]at (1.5, 0.5) {$r'$};
\node(a) at (0,1){$a$};
\node(c) at (2,1){$c$};
\node(b) at (0,0){$b$};
\node(d) at (2,0){$d$};

\draw (a)--(0.5, 0.5)--(1.5,0.5)--(c);
\draw (b)--(0.5, 0.5);
\draw(1.5,0.5)--(d);
\end{tikzpicture}
\caption{$D_A(a,b;c,d)$} \label{fig:D_A(a,b;c,d)}
\end{figure}

 As $A \models S(a,b;c,d)$ and this is not witnessed in the root $D$-set of $E$, it follows from Lemma~\ref{D_rho} that there is $e \in E \setminus A$ such that $E \models S'(a,b;c,d;e)$. 
We have the following  picture in $D_E$, with $e$  special at the shown ramification point $s$ in Figure \ref{fig:$D_E$}.
\tikzset{middlearrow/.style={
        decoration={markings,
            mark= at position 0.75 with {\arrow{#1}} ,
        },
        postaction={decorate}
    }
}
\begin{figure}[H]
\centering
\begin{tikzpicture}
[scale=1]
\node at (0,0)[below]{$s$};
\node(a) at (1,1){$a$};
\node(c) at (1,-1){$c$};
\node(b) at (-1,1){$b$};
\node(d) at (-1,-1){$d$};
\node(w) at (1,0){$e$};

\draw (a)--(0,0);
\draw (b)--(0,0);
\draw (c)--(0,0);
\draw (d)--(0,0);
\draw[middlearrow={stealth}] (w)--(0,0);
\end{tikzpicture}
\caption{$D_E$} \label{fig:$D_E$}
\end{figure}

This picture is a star, and we assume that $A \cup \{e \}$ is not a Type \Romannum{1} extension of $A$. This means that there must be some $x \in A$ (hence in $E$)  witnessing that $A \cup \{ e \}$ is not a Type \Romannum{1} extension of $A$. We consider the various possible positions of $x$ with respect to $a,b,c,d,e$ in $D_E$.\\

{\em{Case ({1})}}. Suppose $x$ lies in the same branch at $s$ as $c$ (with the argument similar if $x$ lies in the same branch as $a,b$ or $d$).  Since $S(a,d;c,x) \wedge (\forall w \in E) \neg S'(a,d;c,x;w)$ holds in $E$ and hence in $A$, $x$ must lie in the same branch as $c$ at $r'$ in $D_A$. Put $r'':=\text{ram}(x,c,d)$. Let $x'\in A$ be in the special branch at $r''$. We assume for convenience that $A\models L(x';c,d)$ (other  cases being similar). Now $Q(a,b;c,d:x';c,d)$ and $Q(a,d;c,x: x';c,d)$ both hold in $A$ and hence both hold in $E$, 
so $S(a,d;c,x)$ and $L(x';c,d)$ are witnessed in the same $D$-set of $E$, which must be the root $D$-set (as $S(a,d;c,x)$ holds there). Likewise, $S(a,b;c,d)$ and $L(x';c,d)$ are witnessd in the same $D$-set of $E$, so $S(a,b;c,d)$ is witnessed in the root $D$-set of $E$, a contradiction.

 {\em{Case {(2)}}}. If $x$ is in the same branch at $s$  as the special branch $e$ in $E$, then we will see $S'(a,b;c,d;x)$ holds in $E$ and hence in $A$. This is impossible, since we have $D_A(a,b;c,d)$, so $A\models S(a,b;c,d) \wedge (\forall t)\neg S'(a,b;c,d;t)$.
 
 {\em{Case {(3)}}}. Suppose neither of (1), (2) holds, but that (to ensure $A\cup\{e\}$ is not a Type I extension of $A$) there exists $x'\in A$ in the same branch as $x$ at $s$ in $D_E$ (distinct from the branches containing $a,b,c,d,e$). Since $S(x,x';u,v)$ (for any distinct $u,v\in \{a,b,c,d\}$) holds in the root $D$-set of $E$, the same holds in $A$. Let $t$ be the unique ramification point of $A$ of form ${\rm ram}(x,x',u)$ for all $u \in \{a,b,c,d\}$. Also let $x''\in A$ lie in the special branch at $t$. We shall suppose 
$L(x'';x,a)$ (there are other similar cases, if say $x''=x$). Now $Q(a,d;c,d:x'';x,a)$ holds in $A$ and hence in $E$,
as does $Q(u,v;x,x':x'';x,a)$ for any distinct $u,v\in \{a,b,c,d\}$. Since such relations $S(u,v;x,x')$ are witnessed in the root $D$-set of $E$, so is $L(x'';x,a)$ and hence so also is $S(a,b;c,d)$, a contradiction. 

\end{proof}
Similarly, it is readily seen that if $A,E\in \mathscr D$ with $A<E$, and there is no $e \in E\setminus A$ such that $A<A \cup \{e \}$ is of Type \Romannum{1}, and $a,b,c \in A$ are distinct, then $L\{a,b,c\}$ is witnessed in $D_A$ if and only if it is witnessed in $D_E$ (recall that we mean by $L\{a,b,c\}$ the  disjunction $L(a;b,c) \vee L(b;a,c) \vee L(c;a,b)$). 

  
 \begin{lem}\label{lemtype} If $A, E \in \mathscr D$ and $E$ is a one-point extension of $A$ with $E=A\cup \{ e\}$, then $(A,E)$ is of Type {\Romannum {1}} or of Type {\Romannum {2}}.
\end{lem}
 \begin{proof}
Assume that the extension is not of Type {\Romannum {1}}, so the structure tree of $E$ has no new root under $\rho_A$ with a star $D$-set. By Lemma \ref{lemb}, $D_E(A)=D(A)$, and  the root $D$-set $D_A$ of $A$ is a substructure of $D_E$, and hence we can identify $\overline {D_A}$ with a subset of $\overline{D_E}$. Furthermore, for $a,b,c \in A$, $L(a;b,c)$ is witnessed in $D_A$ if and only if it is witnessed in $D_E$ by the above paragraph. Thus, either $e$ is added (in $D_E$) as a new non-special leaf to an existing ramification point $r$ of $D_A$, or $e$ is added on a new ramification point $r'$ of an edge of $D_A$. To prove it is of Type
 {\Romannum {2}}, we consider the following cases:
 
 {\em Case (i)}. Suppose that $e$ is added as a new non-special leaf to an existing ramification point $r$ of $D_A$. By induction, as $\lvert A_r \rvert < \lvert A \rvert$, $A_r< E_r$ is an extension of Type I or Type II.  It follows that $A<E$ is a Type \Romannum{2} extension.
 
 {\em Case (ii)}. Suppose that $e$ is added on a new ramification point $r'$ of an edge of $D_A$. In this case, for $E$, $\rho_A$ obtains a new successor $\rho_{r'}$ whose $D$-set has size 2. The structure is otherwise unchanged, and $E$ is an extension of $A$ of Type \Romannum{2}(b). 
 

\end{proof} 

\begin{lem} \label{lem1} Let $A<E$ with $A,E \in \mathscr D$. Then there is an element $e \in E\setminus A $ such that $A \cup \{ e \} \in \mathscr D$.
\end {lem}
\begin{proof} 
We just showed, by Lemma \ref{lemtype},  that extending a substructure $A$ of $E$ by one element,  so that the result lies in $\mathscr D$, can be done by only two ways: Type \Romannum{1} or Type \Romannum{2}.

Firstly, adding $e$ from $E \setminus A$ to $A$ to get $A \cup \{e \} \in\mathscr D$ by a Type \Romannum{1} extension will give the result. Thus, we may suppose there is no such $e$, so $D_E(A)=D_A$ by Lemma~\ref{lemb}.

Suppose there is an edge of $D_A$ such that $E$ has a ramification point $r$ on the edge and there is $e \in E \setminus A$ and $a,b \in A$ such that $a,b,e$ lie in distinct branches at $r$. We may suppose (by careful choice of $e$) that one of $a,b,e$ lies in the special branch at $r$ in $E$. In this case $A \cup \{e \}\in \mathscr D$, a one-point extension of $A$ of Type \Romannum{2}(b).

Suppose the configuration of the last paragraph does not occur. Since $D_E(A)=D_A$, there is a ramification point $r$ of $D_A$ and some $e \in E\setminus A$ lying in a new non-special branch at $r$ of $E$, with $e$ adjacent to $r$ in $\overline{D_E}$. Arguing inductively on $|A|$, we may choose $e$ here (among elements of $E$ lying in a new branch at $r$), so that if $E':=A \cup \{e\}$ then $E'_r\in \mathscr D$ and $A_r<E'_r$ is an extension of Type I or II.  Then $E'\in \mathscr D$ and is a one-point extension of $A$ of Type \Romannum{2}(a).

\end{proof}

The next lemma enables us to reduce proving amalgamation to the special case of amalgamating one-point extensions. 
\begin{lem}{\label{thm1}} Assume $A< E$ with $ A, E \in \mathscr D$. Then we may enumerate $E\setminus A$ as $\{ e_1, e_2, \dots , e_n \}$ such that for each $i=1, \dots ,n$, if $E_i$ is the $\mathscr L$-substructure of $E$ on $A\cup \{ e_1, \dots, e_i\}$ then $E_i \in \mathscr D$. 
\end{lem}
\begin{proof}
Fix $n$. We prove by induction on $m<n$ that there are distinct $e_1, \dots, e_m \in E \setminus A$ such that for each $i=0, \dots,m$ the $\mathscr L$-structure $E_i$ induced on $A \cup \{e_1 ,\dots, e_i\}$ lies in $\mathscr D$ (where $E_0=A$).

The base case $m=0$ is trivial. Assume the result holds for $m$. Then, by Lemma \ref{lem1}, there is some $e\in E \setminus E_m$ such that $E_m \cup \{e \} \in \mathscr D$. Put $e_{m+1}:=e$. 
\end{proof}

\subsection{ Amalgamation Property}\label{amalgamation}
 Fra{\" {\i}ss\'{e}}'s method is based on building a countable structure $M$ as a union of a sequence of finite structures, each itself an amalgam of substructures.  The following is a general lemma that holds for any class of finite structures.
We say that a class $\mathscr C$ has the {\em amalgamation property} if, whenever $A,E_1,E_2 \in \mathscr C$ and $f_i:A\to E_i$ are embeddings (for $i=1,2$) there is $D\in \mathscr C$ and embeddinsg $g_i:E_i\to D$ such that $g_1\circ f_1=g_2 \circ f_2$. We say that $\mathscr C$ has the {\em amalgamation property for one-point extensions} if the above holds whenever $|E_1\setminus f_1(A)|=|E_2\setminus f_2(A)|=1$. 

\begin{lem} \label{lem3}Let $\mathscr C$ be a class of finite structures, and suppose that the following hold:
\begin{enumerate}[(i)]
    \item the class $\mathscr C$ has the amalgamation property for one-point extensions. 
    \item for any $A, E \in \mathscr C$ with $A<E$, we may write $E \setminus A =\{e_1, \dots, e_n \}$ so that if $E_i$ is the induced substructure of $E$ on $A \cup \{ e_1, \dots, e_i \}$ ( for each $i=1, \dots, n$), then $E_i \in \mathscr C$.
\end{enumerate}
Then the class $\mathscr C$ has the amalgamation property.

\end{lem}
\begin{proof} 
See the last three paragraphs of the proof of Lemma 3.7 of  \cite{bhattmacph2006jordan}.


\end{proof}

\begin{lem}
The class $ \mathscr D$ has the amalgamation property.
\end{lem}
\begin{proof} We will prove the amalgamation property for one-point extensions, and the result then follows from Lemmas \ref{thm1} and  \ref{lem3}. Assume $A< E_1 \ \text {and} \ A< E_2 $ with
$A, E_1, E_2 \in \mathscr D \ \text{ such that } \ E_1 \setminus A = \{ e_1 \} \ \text {and} \ E_2 \setminus A= \{e_2 \}$.
We may assume that $e_1 \ \text {and} \ e_2 $ are distinct, and we want to define a structure $E$ on $E_1\cup E_2$, inducing each $E_i$, in such a way that  $E \in \mathscr D$. Let $\tau_i$ be the structure tree corresponding to $E_i$ with root $\rho_i$ where $i=1,2$. We will consider three cases.\\ 
{\em{ Case \romannum{1}}}. Suppose that $E_1 \ \text {and} \ E_2$ are Type \Romannum{1} extensions of $A$. Then, in the structure tree of $E$,  place the root $\rho_2$ beneath the root $\rho_1$ such that $e_2$ is special in $D(\rho_2)$ with $e_1$  non-special, and in $D(\rho_1)$ the element $e_1$ is special. Here $E$ is a Type 1 extension of $E_1$.\\ 
{\em{ Case \romannum{2}}}. Suppose that one of the $E_i$, say $E_1$ is of Type  \Romannum{1}, and $E_2$ is of Type \Romannum{2}. Then define the structure tree of $E$ by placing the root $\rho_1$ under $\rho_{2}$ such that $D(\rho_1)$ is a star in which $e_1$ is special and $e_2$ is not.\\  
{\em {Case \romannum{3}}}. Suppose that  $E_1 \ \text {and} \ E_2$ are of Type \Romannum{2} over $A$. Then we will consider the following four sub-cases.
\begin{enumerate}[(1)]
    \item Asssume that $e_1, e_2$ are added to the same  ramification point $r$ of $D(\rho_A)$ to get $E_1, E_2$ respectively. Keep them distinct in $E$. Then neither of $e_1, e_2$ is special in the root $D$-sets $D(\rho_1)$ and $D(\rho_2)$. In the root $D$-set of $E$,  $e_1, e_2$ will lie in  non-special branches at $r$.  Then higher up two new end-points are added to the same $D$-set $D(f_{\rho_E}^{-1}(r))$, and we finish inductively since $\lvert A_r\rvert < \lvert A\rvert$, so $(E_1)_r$ and $(E_2)_2$ can be amalgamated over $A_r$.
    \item Suppose that $e_1\ \text{and} \ e_2$ are added to distinct  ramification points $r_1$ and $r_2$ of $D(\rho_A)$. Then again, when building $E$,  a leaf 
will be added to each of the $D$-sets corresponding to these ramification points. The structures $E_{r_1}$ and $(E_1)_{r_1}$ will be isomorphic, and $E_{r_2}$ will be isomorphic to $(E_2)_{r_2}$.
\item Suppose that the branch $e_1$ is added to an old ramification  point $r$ of $D(\rho_A)$, and $e_2$ creates a new ramification point $s$,  i.e. $A<E_1$ is of Type \Romannum{2} (a), and $A<E_2$ has Type \Romannum{2}(b). Then a new successor $f_{\rho_E}^{-1}(s)$ has trivial $D$-set in $E$ (i.e. with only 2 elements joined by an edge), and $D(f_{\rho_E}^{-1}(r))$ 
is isomorphic to $D(f_{\rho_{E_1}}^{-1}(r))$.

\item Assume that both $e_1\ \text{and} \ e_2 $ create new ramification points, that is, both give Type \Romannum{2}(b) extensions. Then keep these ramification points  distinct in $E$. Hence $D(\rho_E)$ will have two new ramification points (compared to $D_A$)  and $\rho_E$ has  two new successors with labelling $D$-sets of just two elements.
\end{enumerate}
   
\end{proof} 
\begin{lem} The class $\mathscr D$ has the joint embedding property.
\end{lem}
\begin{proof}
Take two finite structures $A, B \in \mathscr{D}$ with $n, m$ points respectively. Consider their structure trees $\tau_A$ and $\tau_B$  with roots $\rho_A, \rho_B$ respectively. Build a new tree $\tau$ with root $\rho$ such that $D(\rho)$ contains two ramification points $r$ and $r'$ with $n+1$ branches at $r$, and  $m+1$ branches at $r'$, with special branches as shown in the Figure \ref{fig:JEP}. The resulting structure $E$ will have $E_r$ isomorphic to $A$ and $E_{r'}$ isomorphic to $B$. 
\tikzset{middlearrow/.style={
        decoration={markings,
            mark= at position 0.75 with {\arrow{#1}} ,
        },
        postaction={decorate}
    }
}
\begin{figure}[H]
\centering
\begin{tikzpicture}
\node at (2,0) [below]{$r$};
\node at (-2,0) [below]{$r'$};
\node at (3,1){};
\node at (3,0){};
\node at (3,-1){};
\node at (3,0){};
\node at (-3,1){};
\node at (-3,0){};
\node at (-3,-1){};
\node at (4,0)[right]{$n+1$};
\node at (-4,0)[left]{$m+1$};
\draw [middlearrow={stealth}](-2,0)--(2,0);
\draw(2,0)--(3,1);
\draw (2,0)--(3,-1);
\draw [middlearrow={stealth}](2,0)--(-2,0);
\draw (-2,0)--(-3,0);
\draw (-2,0)--(-3,1);
\draw (-2,0)--(-3,-1);
\draw  (3,0)--(2,0);
 \filldraw [black] (3.5,0.5) circle (0.5pt);
 \filldraw [black] (3.5,0) circle (0.5pt);
  \filldraw [black] (3.5,-0.5) circle (0.5pt);
  \filldraw [black] (-3.5,0.5) circle (0.5pt);
\filldraw [black] (-3.5,0) circle (0.5pt);
\filldraw [black] (-3.5,-0.5) circle (0.5pt);
\draw[decoration={brace,mirror,raise=8pt},decorate,thick]
  (3.5,-1.5) -- node[above=8pt] {}(3.5,1.5) ;
  \draw[decoration={brace,mirror,raise=8pt},decorate,thick]
  (-3.5,1.5) -- node[above=8pt] {}(-3.5,-1.5) ;
\end{tikzpicture}
\caption{ } \label{fig:JEP}
\end{figure}
\end{proof} 

The class $\mathscr D$ is not closed under  substructure (see Remark~\ref{semi homo def}), and we therefore  use a standard  modified version of Fra{\" {\i}ss\'{e}}'s Theorem. Below we follow the presentation in Evans \cite{evans1994examples}. The approach is also described in  \cite{hodges1997shorter}, with the resulting Fra\"iss\'e limit described as {\em weakly homogeneous}. 
\begin{mydef}{\label{defnice}}\em{
Let $\mathscr L^*$ be a finite relational language and let $\mathscr C$ be a class of finite $\mathscr L^*$-structures. Define a collection $\mathcal E$ of embeddings
 $f:A\rightarrow D$ where $A,D \in \mathscr C $ such that
\begin{enumerate}[(i)]
     \item any isomorphism is in $\mathcal E$;
    \item $\mathcal E$ is closed under composition;
    \item if $f: A \rightarrow D$ is in $\mathcal E$ and $B,  D\in \mathscr C$ with $B<D$ and $f(A)\subseteq B$, then the map obtained by restricting the range of $f$ to $B$ is also in $\mathcal E$.

\end{enumerate}
 Then we call this collection a {\em{class of $\mathscr C$-embeddings}} }   
\end{mydef}

Consider the following modification for the joint embedding property and the amalgamation property :\\
(JEP$ '$) If $A, B \in \mathscr C$, there exists $C\in \mathscr C$ and embeddings $f: A\rightarrow C$ and $g: B \rightarrow C$ such that $f, g \in \mathcal E$.\\
(AP$ '$) Suppose $A, D_1, D_2 \in \mathscr C$ and $f_i: A\rightarrow D_i$ are embeddings in $\mathcal E$. Then there exists $D \in \mathscr C$ and embeddings $g_i: D_i \rightarrow D$ in $\mathcal E$ such that $g_1f_1=g_2f_2$.\\

Let $\mathcal E$ be a class of $\mathscr C$-embeddings. For an $\mathscr L^*$-structure $M$, and a finite substructure $A \in \mathscr C$, we say that $A$ is $\mathcal E$-{\em{embedded}} in $M$ if whenever $B\in \mathscr C$ is a finite substructure of $M$ which contains $A$, the inclusion map from $A$ to $B$ is in $\mathcal E$.

Then we use the following version of Fra{\" {\i}ss\'{e}}'s Theorem (Theorem 2.10 of \cite{evans1994examples}):
\begin{thm}{\label{general}}
Suppose that $\mathcal C$ is a collection of finite $\mathcal L^*$-structures in which the number of isomorphism types of any finite size is finite. Suppose $\mathcal E$ is a class of $\mathcal C$-embeddings which satisfies JEP$'$ and AP$'$. Then there exists a countable $\mathcal L^*$-structure $M$ with the following properties:
\begin{enumerate}[(i)]
     \item the class of $\mathcal E$-embedded substructures of $M$ is equal to $\mathcal C$;
     \item $M$ is a union of a chain of finite $\mathcal E$-embedded substructures;
     \item if $A\leq M$ and $\alpha: A\rightarrow B$ is in $\mathcal E$ then there exists $C\leq M$ containing $A$ and an isomorphism $\beta: B \rightarrow C$ lying in $\mathcal E$ such that $\beta \alpha (a) = a$ for all $a \in A$.
\end{enumerate}
\end{thm}
Let $M$ be the $\mathscr L$-structure built by applying Theorem \ref{general} to the collection $\mathscr D$ and the collection $\mathcal E$ of embeddings between members of $\mathscr D$.
\begin{lem}\label{extend iso}
Any isomorphism between finite  substructures of $M$ which lie in $\mathscr D$ extends to an automorphism of $M$.
\end{lem}
\begin{proof}
This follows immediately from Theorem \ref{general}(iii).
\end{proof}

We will refer to the property in the last lemma as {\em semi-homogeneity} of $M$, and frequently just say that an automorphism exists, or that a tuple of $M$ has a given finite extension in $M$, `by semi-homogeneity'. 

\begin{rem}\label{semi homo def}{\em{
The collection $\mathscr D$  does not have the hereditary property, that is, it is not closed under substructure. 
For example, consider $C\in \mathscr D$ with elements $x,y,z,w,p$. Let $r={\rm ram}(x,y,z)$, $r':={\rm ram}(x,z,w)$ and suppose $S(x,y;z,w)$ holds, and
$L(x;y,z)\wedge L(x;y,w)\wedge L(x,y,p)$ hold at $r$, all in the root $D$-set $D(\rho_C)$ as in the Figure below. Let $\nu:=f_{\rho_C}^{-1}(r)$ and in $D(\nu)$ suppose that the relation $L(p;y,z)$ is witnessed in the unique ramification point  $r''$. Also let $\nu':=f_{\rho_C}^{-1}(r')$ and $\nu_1:=f_{\nu}^{-1}(r'')$. The two labelling $D$-sets $D(\nu')$ and $D(\nu_1)$ each have just two elements. Put $A=C\setminus\{x\}$. Then $A\not\in \mathscr D$; indeed, otherwise, a short argument would show that $A\models Q(p,y;z,w:p;y,z)$, contradicting the fact that clearly $C\models \neg  Q(p,y;z,w:p;y,z)$.}}
\tikzset{middlearrow/.style={
        decoration={markings,
            mark= at position 0.75 with {\arrow{#1}} ,
        },
        postaction={decorate}
    }
}
\begin{figure}[H]
    \centering
    \begin{tikzpicture}[scale=1.5]

\node [below](alpha) at (0.5,0.5){$r$};
\node [below]at (1.5, 0.5) {$r'$};
\node(x) at (0,1){$x$};
\node(z) at (2,1){$z$};
\node(y) at (0,0){$y$};
\node(w) at (2,0){$w$};
\node (p) at (1,1){$p$};

\draw [middlearrow={stealth}](z)--(1.5,0.5);
\draw [middlearrow={stealth}](x)--(0.5,0.5);
\draw (p)--(0.5,0.5);
\draw (x)--(0.5, 0.5)--(1.5,0.5)--(z);
\draw (y)--(0.5, 0.5);
\draw(1.5,0.5)--(w);
\end{tikzpicture}
    \caption{$\rho_C$}
    \label{fig:rho C}
\end{figure}

\end{rem}

\subsection{Oligomorphicity of {\em{M}}}

In this section we will show  that the automorphism group of $M$ is oligomorphic and hence by  Theorem~\ref{ryll} (the  Ryll-Nardzewski  Theorem) that $M$ is $\omega$-categorical.

To ensure that oligomorphicity of $\text{Aut}(M)$ follows from Lemma~\ref{extend iso}, we need to eliminate situations such as the following. 
 For example, suppose it happened that $M$ has finite substructures $E_i$ (for $i \in \mathbb N$) in the class $\mathscr D$, and suppose $\lvert E_1 \rvert < \lvert E_2 \rvert < \lvert E_3 \rvert < \dots $ and that $E_i$ is a substructure of $M$ of smallest size subject to lying in $\mathscr D$ and containing $a_i,b_i$. Then the pairs $(a_i, b_i)$ all lie in distinct orbits of Aut($M$) on $M^2$.
Our next lemma eliminates this possibility. First we note the following lemma, a standard result easily proved by induction.
\begin{lem}\label{number of ram}
Let $T$ be a graph-theoretic tree with $n$ leaves, where $n\geq 3$. Then $T$ has at most $n-2$ ramification points.
\end{lem}

\begin{lem}\label{F bound}
Define  $f: \mathbb{N} \rightarrow \mathbb{N}$ by $f(n)= (n-2)+ (n-2)(n-3)+\dots +(n-2)(n-3)\dots 2 =\sum_{i=1}^{n-3}\frac{(n-2)!}{i!} $.  Then for every finite $A \subset M$ there is $F \in \mathscr D$ with $A \leq  F \leq M$ and $\lvert F \rvert \leq f(\lvert A \rvert )$. 
\end{lem}
\begin{proof}
By Theorem \ref{general}, $A$ lies in a finite substructure $E$ of $M$ lying in $\mathscr D$. We aim to choose $F$ inside $E$, of minimal size. Let $\rho$ be the root of the structure tree of $E$, $D_E$ be the corresponding $D$-set, let $D_A$ be the induced $D$-set structure on $A$, and $\overline{D_E}$, $\overline{D_A}$ be the corresponding tree structures. Let $n:=\lvert A \rvert $. We shall build $F$ as the union of a finite sequence $A=F_0\subseteq F_1 \subseteq F_2 \subseteq \dots \subseteq E$. We may suppose that $E$ is chosen minimally, that is, there is no proper substructure $E'$ of $E$ with $E'\in \mathscr D$ and $A \leq E'<E$. 

Let $|A|=n$. We have $\lvert \text{Ram}(D_A)\rvert\leq n-2$ (by Lemma \ref{number of ram}). We form $F_1$ by adding to $A$, for each ramification point $r$ of $D_A$ such that the special branch of $E$ at $r$ contains no member of $A$, a member of that special branch. Then $\lvert F_1 \rvert \leq \lvert A \rvert +n-2$, and $F_1$ contains a special branch at each such ramification point $r$ of $D_A$, and has no additional ramification points. 

Next, for each such ramification point $r$ of $D_A$, let $\sigma$ be the corresponding successor in the structure tree of $E$.  (We note here that by minimality of $E$ it cannot happen that the elements of $A$ all lie in distinct non-special branches at the same ramification point $r$ of $D_E$, and thus indeed $\lvert D_\sigma (A) \rvert < \lvert A \rvert =n$). There are at most $n-2$ such $\sigma$, and the $D$-set $D_\sigma$ of $E$ contains at most $n-1$ elements with representatives in $A$, giving a $D$-set $D_\sigma(A)$ of size at most $n-1$, so with at most $n-3$ ramification points. We build $F_2$ to ensure that it contains a special branch at each ramification point of $D_\sigma (A)$, for each $\sigma$. This requires adding at most $(n-2)(n-3)$ points to $F_1$, so $\lvert F_2 \rvert \leq \lvert F_1 \rvert +(n-2)(n-3)$.

We iterate this process. To build $F_3$ from $F_2$, we consider the at most $(n-2)(n-3)$ ramification points of $F_2$ (of $D$-sets of successors of $\rho)$, and the corresponding  $(n-2)(n-3)$ vertices $\lambda$ of height $3$ in the structure tree of $E$. Each  $D$-set $D_\lambda(E)$ contains at most $(n-2)$ elements with representatives in $A$, so the corresponding $D$-set $D_\lambda (A)$ has at most $(n-4)$ ramification points.

Continuing this process, we find that for $F_i$, each $D$-set of height $i$ (where $\rho$ has height $1$) has at most $n-1-i$ ramification points, and that for $j\leq i$, each $D$-set of $F_i$ at height $j$ has a special branch at each ramification point. Thus, putting $F:=F_{n-3}$, we find that $F$ has a special branch at each ramification point of each $D$-set, so $F\in \mathscr D$. Finally, we see inductively that for each $i$, $\lvert F_i \rvert = \lvert F_{i-1}\rvert+ (n-2)(n-3)\dots (n-(i+1))$. Thus, $|F|\leq f(|A|)$, and the result  follows.

\end{proof}
  \begin{lem}\label{olig}
  Let $M$ be the Fra{\" {\i}ss\'{e}} limit of a class $\mathcal C$ of finite structures in a finite relational language,  in the sense of Theorem \ref{general}. Suppose there is a function $f: \mathbb N \rightarrow \mathbb N$ such that for every finite subset $A$ of $M$ there is $F<M$ with $F\in \mathcal C$ and $\lvert F \rvert  \leq f(\lvert A \rvert)$. Then $M$ is $\omega$-categorical.
  \end{lem}
  \begin{proof}
  Suppose that $A$ is a finite subset of $M$ with $k$ elements. Every such $A$ lies in a member $F$ of $\mathcal C$ which is a substructure of $M$ as given in the statement. As the language is finite, and using the bound provided by $f$, there are finitely many choices of such $F$. Any two such structures $F \in \mathcal C$ which are isomorphic lie in the same orbit of Aut$(M)$ on sets. As the choices of $F$ are finite then there are finitely many orbits on such sets $F$. Therefore, as each $F$ has a finite subset isomorphic to $A$ then the number of Aut$(M)$-orbits on $M^k$ is finite for any $k$. By Theorem~\ref{ryll}, $M$ is $\omega$-categorical.
  
  \end{proof}
  
\begin{cor}\label{oligcoro}
The structure $M$  built from $\mathscr D$ via Theorem~\ref{general} is $\omega$-categorical and has oligomorphic automorphism group.
\end{cor}

\begin{proof} 
This is immediate from Lemmas~\ref{F bound} and \ref{olig}.
\end{proof}

\section{Analysing the Fra{\" {\i}ss\'{e}} Limit} \label{Fraisse}
Throughout this section, we  let $M$ be the structure built in Section  \ref{trees of D-sets}, and put $G=\text{Aut}(M)$. In the previous section, we defined finite trees of $D$-sets. Here, we  show that $M$ itself  can  be viewed as a ``tree of $D$-sets''. We have to construct the structure tree of $M$ - in the language of model theory, we interpret it in $M$.  It will be a dense semilinear order without maximal or minimal elements, so in particular there will be no notion of `root' or of `successor'. The vertices of the structure tree are labelled by classes of an equivalence relation on triples arising from the relation symbol $R$, and correspond to certain $D$-sets. The elements of each $D$-set are equivalence classes of a further equivalence relation $E_{xyz}$ defined on a subset of $M$ that corresponds to the vertex $\langle xyz\rangle$ of the structure tree.

  \subsection{Automorphisms of {\em{M}}} \label{auto M}
In this section  we collect in Lemma~\ref{Aut(M)} some basic symmetry properties of $G=\Aut(M)$ . 
  As the language $\mathscr L$ consists of six relations, it is convenient  first to  show that the relations $L',S',Q,R$ are $\emptyset$-definable in $M$ in  terms of $L, S$.
  \begin{lem} \label{LL'SS'}

  \begin{enumerate}[(i)]
      \item $M \models (\forall x,y,z,w) L'(x;y,z;w) \leftrightarrow [L(x;y,z) \wedge  L(w;y,z)  \wedge L(w;x,z)  \wedge L(w;x,y) \wedge \neg S(x,w;y,z)] $.
      
       \item \label{R in terms S}  $M \models R(x;y,z:p;q,s)\  \leftrightarrow \ [L(x;y,z) \wedge L(p;q,s) \wedge (\forall t)(L'(x;y,z;t) \Leftrightarrow L'(p;q,s;t))].$
\item \label{S'F} 
$ M \models S'(x,y;z,w;t) \leftrightarrow \bigdoublewedge \limits_{\substack{u,v \in \{x,y,z,w\}\wedge  u\neq v}}  R(t;x,y:t;u,v)\  \wedge  \bigdoublewedge \limits_ {\substack{u,v,s \in \{x,y,z,w\}\wedge  L(u;v,s)}}\neg R(t;x,y:u;v,s) \wedge  S(x,y;z,w).$

\item \label{Q in S term}  $M \models Q(x,y;z,w:p;q,s)\  \leftrightarrow [S(x,y;z,w) \wedge L(p;q,s) \wedge (\forall t)(S'(x,y;z,w;t) \Leftrightarrow  L'(p;q,s;t))]$.
  \end{enumerate}
 \end{lem}
 
 \begin{proof}
 \begin{enumerate}[(i)]
     \item  ($\Rightarrow)$ Suppose that $L'(x;y,z;w)$ holds in $M$. Pick a finite substructure $A \in \mathscr D$ such that $x,y,z,w \in A<M$. Then there is a $D$-set of $A$ containing $x,y,z,w$ with $w$ lying in the special branch at the ramification point $r=\text{ram}( x,y,z)$ (so all of $x,y,z,w$ lie in distinct branches  at the same ramification point $r$ of this $D$-set). We may assume that this $D$-set is the root $D$-set $D(\rho)$ of $A$ where $\rho$ is the root of the structure tree on $A$. So $L(w;y,z) \wedge L(w;x,z) \wedge L(w;x,y)$ are witnessed in this root $D$-set. Then the  labelling $D$-set of the vertex  $f_\rho^{-1}(r)$ (or one above it)  witnesses $L(x;y,z)$ and omits $w$, and clearly $A \models  \neg S(x,w;y,z)$, so $M\models \neg S(x,w;y,z)$.
     
($\Leftarrow$) In a finite structure $A\in \mathscr D$ with $x,y,z,w \in A <M$, suppose that  $ L(x;y,z) \wedge L(w;y,z) \wedge L(w;x,z) \wedge L(w;x,y) \wedge \neg S(x,w;y,z) $. We aim to show $M \models L'(x;y,z;w)$. We may suppose (by choosing $A$ as small as possible) that the root $D$-set $D(\rho)$ of $A$ is the only one containing $x,y,z,w$ as distinct elements, i.e. lying in distinct directions (that is, in any higher $D$-set of $A$, either some of these will be omitted, or some element corresponds to a union of branches of the root $D$-set containing more than one of $x,y,z,w$) .

Suppose first that $S(x,y;z,w)$ is witnessed in $D(\rho)$ (the argument is similar if $S(x,z;y,w)$ is witnessed in $D(\rho)$).  Let $r_1=\text{ram}(x,y,z)$, and $r_2=\text{ram}(x,z,w)$. See Figure \ref{fig:L' holds in M}.
\tikzset{middlearrow/.style={
        decoration={markings,
            mark= at position 0.75 with {\arrow{#1}} ,
        },
        postaction={decorate}
    }
}
 \begin{figure}[H]
\centering
\begin{tikzpicture}
[scale=2]
\node at (-0.5,0)[below]{$r_1$};
\node(x) at (-1,0.5){$x$};
\node(y) at (-1,-0.5){$y$};
\node at (0.5,0)[below]{$r_2$}; 
\node (z) at (1,0.5)[above]{$z$};
\node (w) at (1,-0.5)[below]{$w$};
\node (u) at (0,0.5)[above]{$u$};

\draw [middlearrow={stealth}](u)--(-0.5,0);
\draw (x)--(-0.5,0);
\draw (-0.5,0)--(y);
\draw (z)--(0.5,0);
\draw (w)--(0.5,0);
\draw (-0.5,0)--(0.5,0);
\end{tikzpicture}
\caption{} \label{fig:L' holds in M}

\end{figure}
Since $L(w;x,y)$, we see that $x$ (and $y$) cannot be special at $r_1$. And since $L(x;y,z)$, we see that $w$ cannot be special at $r_1$. Thus, some other direction $u$ (as depicted) must be special at $r_1$. Then since $z$ and $w$ are identified in $f_{\rho}^{-1}(r_1)$  we cannot have $L(x;y,z) \wedge L(w;x,y)$, a contradiction.

Thus, $x,y,z,w$ all lie in different branches at the same ramification point $r$ of $D(\rho)$. We may suppose further (by the minimality of the choice of $A$) that one of $x,y,z,w$ is special at $r$. Since $L(w;y,z) \wedge L(w;x,z) \wedge L(w;x,y)$, this must be $w$, with $L(x;y,z)$ witnessed in a higher $D$-set of $A$. Thus $A \models L'(x;y,z;w)$, so $M \models L'(x;y,z;w)$.

 \item  $\Rightarrow$) Suppose that $M \models R(x;y,z:p;q,s)$, and let $A \in \mathscr D$ be any finite substructure of $M$ containing $x,y,z,p,q,s$. Then $A \models R(x;y,z:p;q,s)$, so from the way $R$ was defined, $L(x;y,z)$ and $L(p;q,s)$ must be witnessed in the same $D$-set of $A$. It follows immediately that $A \models (\forall t) (L'(x;y,z;t) \Leftrightarrow L'(p;q,s;t))$. Since this hold for any such $A$, it holds in $M$. 
    
    $\Leftarrow$) Suppose that $M$ satisfies 
$$L(x;y,z)\ \wedge \ L(p;q,s) \ \wedge \ (\forall t)(L'(x;y,z;t) \Leftrightarrow L'(p;q,s;\\t)),$$ and let $A \in \mathscr D$ be a finite substructure of $M$ containing $x,y,z,p,q,s$. Then as $M\models L(p;q,s) \wedge L(x;y,z)$, these $L$-relations are witnessed in distinct comparable $D$-sets of $A$ (here comparability is  with respect to the structure tree ordering), or incomparable $D$-sets of $A$, or in the same $D$-set of $A$.
    
    If $L(x;y,z)$ and $L(p;q,s)$ are witnessed in distinct comparable $D$-sets of $A$, say $L(p;q,s)$ below $L(x;y,z)$, then there is some $t \in A$    with $A\models L'(x;y,z;t) \wedge \neg L'(p;q,s;t)$, a contradiction.
    
    Suppose that $L(p;q,s)$ and $L(x;y,z)$ are witnessed in incomparable $D$-sets of $A$. Then we may suppose (replacing $A$ by a substructure if necessary) that in the root $D$-set $D(\rho)$ of $A$, there are distinct ramification points $r_1$ and $r_2$ such that $x,y,z$ lie in distinct branches at $r_1$ and $p,q,s$ lie in distinct branches at $r_2$. We now see that for all possible choices of special branches at $r_1$ and $r_2$, $A \models( \exists t) \neg (L'(x;y,z;t) \Leftrightarrow L'(p;q,s;t))$, again a contradiction. 

Thus, $L(p;q,s)$ and $L(x;y,z)$ are witnessed in the same $D$-set of $A$, so $A \models R(x;y,z:p;q,s)$, and hence $M \models R(x;y,z:p;q,s)$ as required.
    \item $\Rightarrow$ Assume $M\models S'(x,y;z,w;t)$, and let $A\in \mathscr D$ be a substructure of $M$ containing $x,y,z,w$ in distinct non-special branches of some ramification point $r$ of the root $D$-set, and $t\in A$ in a special branch at $r$.
As $A\models S'(x,y;z,w;t)$ there is a $D$-set of $A$ witnessing $S(x,y;z,w)$ and omitting $t$. In particular $M\models S(x,y;z,w)$. 

In the root $D$-set of $A$ we see that  $\bigdoublewedge \limits_{\substack{u,v \in \{x,y,z,w\}\wedge  u\neq v}}  R(t;x,y:t;u,v)$ holds in $A$ (and hence in $M$). Also, $L(t;x,y)$ is witnessed in the root $D$-set of $A$, which cannot witness $L(u;v,s)$ for $u,v,s\in \{x,y,z,w\}$. Thus $A \models  \bigdoublewedge \limits_ {\substack{u,v,s \in \{x,y,z,w\}\wedge  L(u;v,s)}}\neg R(t;x,y:u;v,s)$, and hence this holds also in $M$.

$\Leftarrow$ Assume, for a contradiction, that the right hand side holds and that $M\models \neg S'(x,y;z,w;t)$. Then there is a finite $A\in \mathscr D$ with $x,y,z,w,t\in A\leq M$ and $A\models \neg S'(x,y;z,w;t)$. As $M\models S(x,y;z,w)$, there is a $D$-set of $A$ witnessing $S(x,y;z,w)$ and containing $t$. Careful analysis of the possible positions of $t$, and possible choices of special branches, shows that  $\bigdoublewedge \limits_{\substack{u,v \in \{x,y,z,w\}\wedge  u\neq v}}  R(t;x,y:t;u,v)\  \wedge  \bigdoublewedge \limits_ {\substack{u,v,s \in \{x,y,z,w\}\wedge  L(u;v,s)}}\neg R(t;x,y:u;v,s)$ cannot hold in $A$, so cannot hold in $M$.

 \item This is similar to $(ii)$.
 \end{enumerate}
 
 \end{proof}
 
 It follows from Lemma~\ref{LL'SS'} that $G$ is the automorphism group of the reduct of $M$ to the language with just the relation symbols $L$ and $S$. 
 
\begin{lem}\label{Aut(M)}
\begin{enumerate}[(i)]

    The group $G$ in its action on $M$ has the following properties: it is
    \item \label{3 homom} $3$-homogeneous,
    \item $2$-transitive,
    \item primitive,
    \item $2$-primitive.
    
\end{enumerate}
 
\end{lem}
  
   \begin{proof}
   \begin{enumerate}[(i)] 
   \item Let $A=\{x,y,z\}$ and $A'=\{x',y',z'\}$ be $3$-element subsets of $M$. Observe that the induced structures on $A$ and $A'$ lie in $\mathscr D$, since any $3$-element substructure of any member of $\mathscr D$ lies in $\mathscr D$ (Lemma~\ref{3set}(ii)), and $M$ is a union of a chain of members of $\mathscr D$. By Lemma~\ref{3set}(i) we have $A \models L\{ x,y,z \}$ and $A' \models L \{x',y',z' \}$. Without loss of generality, assume $A\models L(x;y,z)$ and  $A'\models L(x';y',z')$. It is easily seen that the map $g:A \rightarrow A'$ with $(x,y,z)^g = (x',y',z')$ is an isomorphism. Hence, by Lemma \ref{extend iso}, $g$ extends to some element of $G$.  
  
 \item Suppose $x,y,x',y' \in M$ with $x\neq y$ and $x'\neq y'$. Let $A$ be the induced structure on $\{x,y \}$, and $A'$ be that on $\{x',y'\}$. Then $A, A' \in \mathscr D$ by Lemma~\ref{3set}(ii), and the map $g:A \rightarrow A'$ given by $(x,y)^g =(x',y')$ is an isomorphism. By Lemma \ref{extend iso} $g$ extends to an element of $G$, as required. 
 
    \item This follows from ({\em{ii}}). 
  \item Since $G$ is $2$-transitive, it remains to check that for $a\in M$ the group $G_a$ is primitive on $M \setminus \{ a \}$. Thus we must show that there is no proper non-trivial $G_a$-congruence on $M \setminus \{a\}$. So for the fixed $a$ it suffices, using semi-homogeneity, to show the following are not equivalence relations:
  \begin{enumerate}[(a)]
      \item  $E_{a}(x,y) \Leftrightarrow L(a;x,y) \vee x=y$. This relation is not transitive. Indeed, assume $L(a;x,y) \wedge L(a;y,z)$. Working in a finite substructure of $M$ lying in $\mathscr D$, we may choose $a,y,x$ to be in distinct branches distinct at a ramification point $r$ with $a$ in the special branch, and $z$ lying in the same branch as $x$ as in the following picture.
Therefore, $\neg L(a;x,z)\wedge x \neq z$, so $E_a$ is not a transitive relation.
 \tikzset{middlearrow/.style={
        decoration={markings,
            mark= at position 0.75 with {\arrow{#1}} ,
        },
        postaction={decorate}
    }
}
  \begin{figure}[H]
\centering
\begin{tikzpicture}
[scale=1]
\node at (0,0)[below]{$r$};
\node(x) at (1,1){$x$};
\node(a) at (1,-1){$a$};
\node(y) at (-1,1){$y$};
\node at (0.5,0.5){};
\node (z) at (0,1)[above]{$z$};
\node at (-1,-1){};

\draw (0,0)-- (-1,-1);
\draw (x)--(0,0);
\draw [middlearrow={stealth}](a)--(0,0);
\draw (y)--(0,0);
\draw  [middlearrow={stealth}](0,1)--(0.5,0.5);
\end{tikzpicture}
\caption{} \label{fig:2-primitive}
\end{figure}

\item $F_{a}(x,y) \Leftrightarrow L(x;a,y) \vee x=y$. This relation is not symmetric, since $L(x;a,y) \rightarrow \neg L(y;a,x)$ (in any finite substructure and hence in $M$).

\item $F'_a(x,y)\Leftrightarrow L(x;a,y) \vee L(y;a,x) \vee x=y$. This is not transitive, as in the configuration below we have $F'_a (x,y)\wedge F'_a(y,z) \wedge \neg F'_a(x,z)$.
\begin{figure}[H]
    \centering
    \begin{tikzpicture}
[scale=1]
\node at (0,0)[left]{$r$};
\node(z) at (1,1){$z$};
\node(x) at (1,-1){$x$};
\node(y) at (1,-2){$y$};
\node at (0,-1){};
\node (a) at (0,-2){$a$};

\draw [middlearrow={stealth}](y)--(0,-1);
\draw (z)--(0,0);
\draw (a)--(0,0);
\draw [middlearrow={stealth}](a)--(0,0);
\draw (x)--(0,0);

\end{tikzpicture}
    \caption{}
    \label{fig:F'(x,y)}
\end{figure}

  \end{enumerate}
 \end{enumerate}
     \end{proof}
   
   \subsection{Construction}
   In this section we aim to recover a notion of  structure tree for $M$,  using the relations $L$ and $S$. First, define
$K^*:=\{ (x,y,z) \in M^3: M \models L(x;y,z)\}$. Using the definition of $R$ on members of $\mathscr D$, it is immediate that $R$ defines an equivalence relation on $K^*$, and the structure tree will have universe $K^*/R$.  We refer to the equivalence classes of $R$ on $K^*$ as {\em{vertices}}, and denote the $R$-class containing $(x,y,z)$ as $ \langle xyz\rangle $.
   \begin{mydef}\label{Dsetdef}\rm
     Let $p,q,s \in M$ with $M\models L(p;q,s)$. Define 
$$J_{pqs}:=
\{ j: R(p;q,s:j;q,s) \vee R(p;q,s: p;j,s) \vee R(p;q,s:p;j,q)
\}.$$ 

\end{mydef}

 


\begin{lem} \label{QRIJ}
Let $x,y,z,p,q,s \in M$ with $M\models L(x;y,z)\wedge L(p;q,s)$. Then
\begin{enumerate}[(i)]   
\item  \label{RJJ}$M \models R(x;y,z: p;q,s) \Leftrightarrow J_{xyz} =J_{pqs}$.
\item $J_{xyz}=J_{pqs}\Leftrightarrow \langle xyz\rangle =\langle pqs\rangle$.
\end{enumerate}

\end{lem}
\begin{proof}
\begin{enumerate}[(i)]
\item      $\Rightarrow )$ Assume $M \models R(x;y,z: p;q,s)$.  To show that $J_{xyz} \subset J_{pqs}$, let $b \in J_{xyz}$, so we want $ b \in J_{pqs}$. There is a finite
$\mathscr L$-substructure $A<M$  containing $x,y,z,p,q,s,b$ with $A \in \mathscr D$. Since $R$ is  a symbol of $\mathscr L$,  $A\models R(x;y,z:p;q,s)$, so $L(x;y,z)$ and   $L(p;q,s)$ are witnessed in the same $D$-set of $A$. But $b \in J_{xyz}$ so, without loss of generality, let $R(x;y,z:b;y,z)$ hold. Thus, we may suppose  $L(x;y,z)$ and   $L(p;q,s)$ are witnessed in the root  $D$-set of $A$. By considering possible configurations in $A$, we see $R(p;q,s:b;q,s) \vee R(p;q,s: p;b,s) \vee R(p;q,s:p;b,q)$, so $b\in J_{pqs}$.

    $ \Leftarrow$) Assume, for a contradiction $J_{xyz}=J_{pqs}$ and $\neg R(x;y,z:p;q,s)$. Then there is a finite substructure $A \in \mathscr D$ such that $x,y,z,p,q,s \in A<M$. Since $A \models \neg R(x;y,z:p;q,s)$, $L(x;y,z)$ and $L(p;q,s)$ are witnessed  in different $D$-sets of $A$. We suppose first that these $D$-sets are comparable, so (without loss of generality) there is $t \in A$ such that $A\models L'(x;y,z;t)$ and $t$ lies in the $D$-set of $A$ witnessing $L(p;q,s)$. We see easily that $t \in J_{pqs}\setminus J_{xyz}$. 
    On the other hand, if $L(p;q,s)$ and $L(x;y,z)$ happen in two incomparable $D$-sets, then a lower $D$-set contains $p,q,s$ in distinct branches at a ramification point, $r$ say, with $t$ in the special branch and $x,y,z$  in distinct branches at another ramification point, $r'$ say, with $t'$ in the special branch. There are several possible configurations to consider, but in each case we find $J_{xyz}\neq J_{pqs}$. 
\item This is immediate from (i). 

\end{enumerate} 
\end{proof}

For convenience we also note the following lemma.
\begin{lem}\label{conven}
Suppose $M\models L(x;y,z)\wedge L(p;q,s)$ and let $A\in \mathscr D$ with $x,y,z,p,q,s\in A\leq M$. Then the following are equivalent.
\begin{enumerate}[(i)]
\item The $D$-set of $A$ witnessing $L(p;q,s)$ is below or equal to that witnessing $L(x;y,z)$;
\item $J_{xyz}\subseteq J_{pqs}$.
\end{enumerate}
\end{lem}
\begin{proof}
The direction (i) $\Rightarrow$ (ii) is immediate, and (ii) $\Rightarrow$ (i) is a small adaptation of the proof of the $\Leftarrow$ direction  for Lemma~\ref{QRIJ}(i).
\end{proof}

 We define a partial order $\leq$ on $K^*/R$ by reverse inclusion; that is $\langle xyz\rangle \leq \langle pqr \rangle $ if and only if $J_{xyz}\supseteq J_{pqr}$. It is immediate from Lemma \ref{QRIJ} that this is well defined. We shall show later (Lemma \ref{semi2}) that $(K^*/R, \leq )$ is a dense semilinear order without maximal or minimal elements. First, we aim  to associate a $D$-set with each vertex of $K^*/R$.

\begin{mydef}  \label{defE}\rm
Define a relation $E_{pqs}$ on $J_{pqs}$ such that $u E_{pqs} v$ holds if and only if
$$ (\forall m) (\forall n)  (R(p;q,s:m;n,u) \leftrightarrow 
R(p;q,s:m;n,v)) \wedge (R(p;q,s:u;m,n) \leftrightarrow 
R(p;q,s:v;m,n)).$$ 
\end{mydef}

Observe that  if $J_{xyz}=J_{x'y'z'}$, then $R(x;y,z:x';y',z')$ by Lemma~\ref{QRIJ}, and it follows that $E_{xyz}=E_{x'y'z'}$. Indeed, if
$uE_{xyz}v$ then 
$$ (\forall m) (\forall n)  (R(x;y,z:m;n,u) \leftrightarrow 
R(x;y,z:m;n,v)) \wedge (R(x;y,z:u;m,n) \leftrightarrow 
R(x;y,z:v;m,n)),$$ 
so (as $R$ is an equivalence relation on $K^*$)
$$ (\forall m) (\forall n)  (R(x';y',z':m;n,u) \leftrightarrow 
R(x';y',z':m;n,v)) \wedge (R(x';y',z':u;m,n) \leftrightarrow 
R(x';y',z':v;m,n)),$$ 
so $uE_{x'y'z'}v$. Also,
    $E_{xyz}$ is an equivalence relation on $J_{xyz}$, and is invariant under $G_{\{J_{xyz}\}}$. 

\begin{mydef}\label{longdef}{\em{
 \begin{enumerate}[(i)]
     \item Given $J_{xyz}$ and  $E_{xyz}$, define $R_{xyz}$ to be the quotient $J_{xyz}/E_{xyz}$, so elements of $R_{xyz}$ are $E_{xyz}$-classes of elements of $M$. We use the notation $[m]$ to refer to the element of $R_{xyz}$ containing the element $m \in M$ (when the underlying equivalence relation $E_{xyz}$ is clear). We call such objects $[m]$  {\em{directions}} when viewed as elements of $R_{xyz}$, and {\em{pre-directions}} when viewed as subsets of $M$. We shall refer to the subset $J_{xyz}$ of $M$ as a {\em pre-$D$-set}. 
     \item \label{D_xyzw} Let $ [u], [v] ,[t],[s]\in R_{xyz}$. Write  
$D_{xyz} \ ([u]\ ,[v]\ ;[t]\ ,[s])\Leftrightarrow \  ([u]=[v]\ \wedge  \ [u]\notin  \{[s], [t]\}) \vee ([t]=[s]\ \wedge \ [t]\notin \{[u] ,[v]\}) \vee Q(u,v;t,s:x;y,z)$.
 \end{enumerate}

 }}
\end{mydef}

By considering finite substructures, it can be checked that any such subset $[m]$ of $M$, as in (i) above,  is a direction of a {\em unique} $D$-set $R_{xyz}$.

 We denote the structure tree by $(K^*/R, \leq)$.


\begin{lem}\label{4} 
\begin{enumerate}[(i)]
\item The relation $D_{xyz}$ is well-defined on $R_{xyz}$.
\item If $\langle xyz\rangle=\langle x'y'z'\rangle$ then $D_{xyz}=D_{x'y'z'}$. 
    \item The structure $(R_{xyz}, D_{xyz})$ is a dense proper $D$-set.
   \item \label{D preserved}The relation $D_{xyz}$ is $G_{\{J_{xyz}\}}$-invariant.
    
\end{enumerate}
\end{lem}
\begin{proof}
\begin{enumerate}[(i)]
    \item Suppose $[u],[v],[s], [t] \in R_{xyz}$ are distinct, and $u'\in [u], v' \in [v], s'\in [s]$, and $t'\in [t]$. We must show $D_{xyz}(u,v;s,t) \leftrightarrow D_{xyz}(u',v';s',t')$, so suppose $M\models D_{xyz}(u,v;s,t)$. We may assume $M\models Q(u,v;s,t:x;y,z)$, so must show $M\models Q(u',v';s',t':x;y,z)$. Choose any large finite $A\in \mathscr D$ with $A\leq M$ containing $x,y,z,u,u',v,v',s,s',t,t'$, so $A\models  Q(u,v;s,t:x;y,z)$.

By the definition of $E_{xyz}$, and using symmetry conditions on the variables in $R$, we have in $M$ and hence in $A$: 
$$ (\forall m \forall n) (R(x;y,z: u;m,n) \leftrightarrow R(x;y,z:u';m,n) )$$
$$ (\forall m \forall n ) (R(x;y,z: m;v,n) \leftrightarrow R(x;y,z:m;v',n) )$$
$$ (\forall m \forall n ) (R(x;y,z: m;n,t) \leftrightarrow R(x;y,z:m;n,t') )$$
Consider the $D$-set $D_\nu$ of $A$ witnessing $S(u,v;s,t)$ and $L(x;y,z)$. If $u,u'$ correspond to distinct leaves of $D_\nu$, then there is some $w\in A$ so that $L\{u,u',w\}$ is witnessed in $D_\nu$, say with $L(u;u',w)$ witnessed in $D_\nu$. Thus, $A\models R(x;y,z:u;u',w)$, so $A\models R(x;y,z:u';u',w)$, which is clearly impossible. Thus, $u,u'$ correspond to the same leaf of $D_\nu$, and likewise, $v,v'$ correspond to the same leaf of $D_\nu$, as do $s,s'$, and also $t,t'$. It follows that 
$A\models Q(u',v';s',t':x;y,z)$, so $M\models Q(u',v';s',t':x;y,z)$, as required.

\item Since $\langle xyz\rangle=\langle x'y'z'\rangle$ we have $R(x;y,z:x';y',z')$. So if $D_{xyz}(u,v;t,s)$, then $M\models Q(u,v;t,s:x;y,z)$, so any sufficiently large finite substructure $A$ of $M$ satisfies  $R(x;y,z:x';y',z') \wedge Q(u,v;t,s:x;y,z)$ so satisfies $Q(u,v;t,s:x';y',z')$, so $M\models  Q(u,v;t,s:x';y',z')$. Hence $D_{x'y'z'}(u,v;t,s)$ holds. 

\item We want to show that conditions $(D1)-(D6)$ of Definition \ref{Drelation} hold.  Axioms $(D1)$, $(D2)$, $(D3)$ and $(D4)$ follow immediately from corresponding conditions on $S$, inherited via $Q$.  For $(D5)$, suppose that $[u], [v], [t] \in R_{xyz}$ are distinct. Pick finite $A \in \mathscr D$ with $x,y,z,u,v,t \in A<M$. We may suppose that $L(x;y,z)$ is witnessed in the root $D$-set of $A$. By semi-homogeneity, $A$ has a Type \Romannum{2}(b) extension $A<A'=A \cup \{s\}$ such that $S(u,v;t,s)$ is witnessed in the root $D$-set of $A'$, with $A'<M$. Then $M \models Q(u,v;t,s:x;y,z)$, and we have $D_{xyz}(u,v;t,s)$.

The argument is similar for $(D6)$. Suppose $[u],[v],[t],[s] \in R_{xyz}$ with $D_{xyz}([u],[v];[t],[s])$, and for convenience we suppose them distinct. Pick $A \in \mathscr D$ with $x,y,z,u,v,t,s \in A<M$. Then $A\models Q( u,v;t,s:x;y,z)$, and we may suppose $L(x;y,z)$ and $S(u,v;t,s)$ are witnessed in the root $D$-set of $A$. Now, by semi-homogeneity  $A$ has a Type \Romannum{2}(b) extension $A<A'=A \cup \{a \}<M$, as depicted in Figure \ref{fig:}.

\begin{figure}[H]
\centering
\begin{tikzpicture}
[scale=2]
\node(u) at (0,1){$u$};
\node(t) at (2,1){$t$};
\node(v) at (0,0){$v$};
\node(s) at (2,0){$s$};
\node(a) at (1,1) {$a$};

\draw (u)--(0.5, 0.5)--(1.5,0.5)--(t);
\draw (v)--(0.5, 0.5);
\draw(1.5,0.5)--(s);
\draw[dashed] (1, 0.5)--(a);
\end{tikzpicture}
\caption{} \label{fig:}
\end{figure}

Then $$A\models S(a,v;t,s) \wedge S(u,a;t,s) \wedge S(u,v;a,s) \wedge S(u,v;t,a),$$ all witnessed in the root $D$-set of $A$, so in $M$ we have (putting $D=D_{xyz}$ and arguing via $Q$) $$D([a],[u];[t],[s]) \wedge D([u],[a];[t],[s]) \wedge D([u],[v];[a],[s]) \wedge D([u],[v];[t],[a])$$ as required.
\item This follows immediately from (ii) and Lemma~\ref{QRIJ} (ii).
\end{enumerate}
\end{proof}

\begin{lem}\label{LSQ}
\begin{enumerate}[(i)]
    \item If $L(p;q,s)$ holds in $M$ then there are $ x,y,z,w \in M$ such that $M\models Q(x,y;z,w:p;q,s)$.
    \item If $S(x,y;z,w)$ holds in $M$ then there are $p,q,s \in M$  such that $Q(x,y;z,w:p;q,s)$. 
\end{enumerate}
\end{lem}
\begin{proof}
\begin{enumerate}[(i)]
    \item First observe that the induced $\mathscr L$-structure on $\{ p,q,s \}$ lies in $\mathscr D$. Pick $A<M$ with $A \in \mathscr D$, containing distinct elements $p',q',s',x',y',z',w'$ such that the relations $L(p';q',s')$ and $S(x',y';z',w')$ are witnessed in the root $D$-set of $A$. Then $A \models Q(x',y';z',w':p';q',s')$ so $M \models Q(x',y';z',w':p';q',s')$. By $3$-homogeneity (Lemma \ref{Aut(M)}(\ref{3 homom})) there is $g\in G$ with $(p',q',s')^g= (p,q,s)$. Put $x:={x'}^g, y:={y'}^g, z={z'}^g, w:={w'}^g$. Then $M \models Q(x,y;z,w:p;q,s)$, as required.
    
    \item Similar to (\romannum{1}).
\end{enumerate}
\end{proof}

The notions of {\em ramification point} and {\em  branch} (at a ramification point), as introduced at the start of Section 3 for finite $D$-sets, make sense also for infinite $D$-sets, interpreted in the obvious way. We do not define them formally here, but refer to \cite{adeleke1998relations}. A $D$-set $(R,D)$ determines a corresponding general betweenness relation (see  \cite[Theorem 25.3]{adeleke1998relations}), whose elements form the set Ram$(R)$ of ramification points of $R$. If $r\in {\rm Ram}(R)$, then $r$ corresponds to a {\em structural partition} of $R$, whose {\em sectors} are the branches at $r$ (\cite[Section 24]{adeleke1998relations}. 

We shall say that $L(p;q,s)$ is {\em witnessed} in $R_{xyz}$ if $M\models R(x;y,z:p;q,s)$, and likewise that $S(u,v;t,s)$ is witnessed in $R_{xyz}$ is $M\models Q(u,v;t,s:x;y,z)$. 
 If $L(p;q,s)$ is witnessed in $R_{xyz}$, and the $E_{xyz}$-classes of $p,q,s$ lie in distinct branches at the ramification point $r$, we say that the branch at $r$ containing $p/E_{xyz}$ is the {\em{special branch }} at $r$. Let Ram$(R_{xyz})$ denote the set of ramification points of $R_{xyz}$. If $a_1, \dots , a_n \in R_{xyz}$ (for $n\geq 3)$ lie in distinct branches at $r \in \text{Ram}(R_{xyz})$, we write $r=\text{ram}(a_1, \dots, a_n)$. 
\begin{lem} \label{semi2}

   The partial order $(K^*/R, \leq)$ is  a lower semilinear order. Furthermore,  if $\langle xyz \rangle , \langle pqs \rangle $ are incomparable elements there is a vertex $\langle abc \rangle $ such that $\langle abc \rangle = \text{\em{inf}} \{ \langle xyz\rangle, \langle pqs \rangle \}  $, so  $K^*/R$ is a meet-semilattice. In addition, it has no maximal or minimal elements, and is dense (that is, satisfies $\forall u\forall v(u<v\to \exists w(u<w<v))$).
\end{lem}
\begin{proof}
We show the semilinearity via Claims $1$ and $2$ below:

{\em{Claim 1.}} Given two sets $J_{xyz}$ and $J_{pqs}$ such that neither contains the other then there is another set $J_{abc}$ containing both.

\begin{proof}
Let $A<M$ with $x,y,z, p,q,s \in A \in \mathscr D$. Then by Lemma~\ref{conven}  $L(x;y,z)$ and $L(p;q,s)$ are witnessed in incomparable $D$-sets of $A$, and we may suppose these lie in cones (at the root $\rho$ of $A$) corresponding to distinct ramification points $r_1, r_2$ of $D(\rho)$. There are $a,b,c\in A$ with $L(a;b,c)$ witnessed in $D(\rho)$, such that $x,y,z$ are in  distinct branches at the ramification point $r_1$ and $p,q,s$ are in distinct distinct branches at $r_2$, one configuration being as depicted below. Furthermore, by Lemma~\ref{conven} we have 
$J_{xyz} \cup J_{pqs}\subset J_{abc}$, and furthermore, by considering any sequence of one-point extensions of $A$,  we see that $\langle abc \rangle $ is the infimum in $K^*/R$ of $\langle xyz \rangle$ and $\langle pqs \rangle$. 
\begin{figure}[H]
\centering
\begin{tikzpicture}
[scale=2]
\node (r1) at (0.5,0.5)[below]{$r_1$};
\node (r2) at (1.5,0.5)[below]{$r_2$};
\node(a) at (0,1){$a$};
\node(c) at (2,1){$c$};
\node(b) at (0,0){$b$};
\node (x) at (0,0.5){$x$};
\node (y) at (1,1){$y$};
\node (z) at (0.5,0){$z$};
\node (p) at (2,0.5){$p$};
\node (q) at (1.25, 1){$q$};
\node (s) at (1.25,0){$s$};
 
\draw (s)--(1.5,0.5);
\draw (q)--(1.5,0.5);
\draw (p)--(1.5,0.5);
\draw (z)--(0.5,0.5);
\draw (x)--(0.5,0.5);
\draw (y)--(0.5,0.5);
\draw (a)--(0.5, 0.5)--(1.5,0.5)--(c);
\draw (b)--(0.5, 0.5);

\end{tikzpicture}
\caption{$D_\rho$} \label{fig:Claim 1}
\end{figure}


\end{proof}

{\em{Claim 2.}} Assume $\langle xyz \rangle $ and $\langle pqs \rangle $ are incomparable. There is no $J_{lmn}$ contained in both $J_{pqs}$ and $J_{xyz}$.

\begin{proof}
Assume there are $l,m,n\in M$ with $J_{lmn} \subseteq J_{xyz} \cap J_{pqs}$. By assumption, there are $a,b \in M$ with $a \in J_{xyz}\setminus J_{pqs}$ and $b \in J_{pqs}\setminus J_{xyz}$. Let $A<M$ be finite with $x,y,z,p,q,s,l,m,n\in A \in \mathscr D$. Let $L(x;y,z)$ and $L(p;q,s)$ be witnessed in $A$ by $D$-sets $D_{\nu_1}$ and $D_{\nu_2}$ respectively, and $L(l;m,n)$ by $D_\mu$. Then $\nu_1$, $\nu_2$ are incomparable (due to the existence of $a, b$, and using Lemma~\ref{conven}) but $\mu \geq \nu_1$ and $\mu \geq \nu_2$ (as $J_{lmn}\subseteq J_{xyz}\cap J_{pqs}$). This is impossible, as the structure tree of $A$ is semilinearly ordered.
\end{proof}

{\em{Claim 3.}} The semilinear  order $(K^* /R, \leq)$ has no greatest or least element, and  is dense.

\begin{proof}
Let $\langle xyz \rangle \in K^*/R$. Let $B$ be a minimal substructure of $M$ in $\mathscr D$ containing $x,y,z$, so $B=\{x,y,z\}$. Choose a structure $A \in \mathscr D$ containing $x',y',z',p,q,s$ as depicted (in the root $D$-set),
\tikzset{middlearrow/.style={
        decoration={markings,
            mark= at position 0.75 with {\arrow{#1}} ,
        },
        postaction={decorate}
    }
}
\begin{figure}[H]
    \centering
    \begin{tikzpicture}
    \node (r)at  (-2,0)[left] {$r$};
    \node (r1) at (0,0){};
    \node (r2) at (2,0){};
    \node (x') at (-3,1)[above]{$x'$};
    \node (y') at (-1.75,1) {$y'$};
    \node (z') at (-3,-1) {$z'$};
    \node (s) at (-1,1) {$s$};
    \node (p)at (3,1){$p$};
    \node (q) at (3,-1){$q$};
    
    \draw (-3,1)--(-2,0);
    \draw (y')--(-2,0);
    \draw (z')--(-2,0);
    \draw (s)--(0,0);
    \draw [middlearrow={stealth}] (0,0)--(-2,0);
    \draw [middlearrow={stealth}] (-2,0)--(0,0);
    \draw [middlearrow={stealth}] (0,0)--(2,0);
    \draw (p)--(2,0);
    \draw (q)--(2,0);
    \end{tikzpicture}
    \caption{$D_{\rho_A}$}
    \label{fig:I_xyzw I_pqst}
\end{figure}
\noindent so that the map $(x',y',z')\longmapsto (x,y,z)$ is an $(L,S)$-isomorphism, $L(x';y',z')$ is witnessed in the successor $D$-set $D_\mu$ of the root $D_\rho$ of $A$ corresponding to the ramification point $r$, and $p,q,s$ are as shown in $D_\rho$. We may suppose also $A_r \cong B$, via an isomorphism $\phi$ inducing $(x',y',z') \longmapsto (x,y,z)$.

We may suppose $A\leq M$. Then by semi-homogeneity, $\phi$ extends to some $g \in G$. Clearly $\langle spq\rangle < \langle x'y'z'\rangle$, and it follows that $\langle s^gp^gq^g\rangle < \langle xyz \rangle $, as required. A similar argument shows that $K^*/R$ has no greatest element under $\leq$.

The argument for density is a similar application of semi-homogeneity. Assume $\langle xyz \rangle < \langle pqs\rangle$. We may find a finite substructure $A$ of $M$ containing $x,y,z,p,q,s$, such that $L(x;y,z)$ is witnessed in the root $D$-set $D_\rho$ of $A$, which has a ramification point $r$ at which $p,q,s$ lie in distinct non-special branches. Using semihomogeneity, we may suppose that there are $l,m,n\in A$ such that $p,q,s,l,m,n$ all lie in distinct non-special branches at $r$, that $L(l;m,n)$ is witnessed in the successor $D_\mu$ corresponding to $r$, and that $p,q,s$ lie in distinct non-special branches at a ramification point of $D_\mu$. It follows by Lemma~\ref{conven} that $\langle xyz \rangle < \langle lmn \rangle < \langle pqs\rangle $, as required.

\end{proof}
\end{proof}

Our next task is to identify analogues for $M$ of the maps $f_\mu$ and $g_{\mu \nu}$ for members of $\mathscr D$.
The map $f_{\langle xyz \rangle}$ determines a bijection between the set of cones of $(K^*/R,\leq)$ at $\langle xyz \rangle$ and Ram$(R_{xyz})$, and this is the context of the following lemma.
\begin{lem}\label{A}
Suppose $x,y,z,p,q,s \in M$ and $\langle xyz\rangle <\langle pqs \rangle$. Then
 
\begin{enumerate}[(i)]
    \item \label{E refines} $E_{xyz}\big | _{J_{pqs}}$ refines  $E_{pqs}$.
    \item \label{u v E} Let $p,q,s$ lie in distinct branches at the ramification point $r$ of $R_{xyz}$, and let $u,v \in J_{pqs}$ be $E_{pqs}$-inequivalent. Then $u,v$ lie in distinct branches at $r$. 
\end{enumerate}
\end{lem}
\begin{proof}
\begin{enumerate}[(i)]
    \item Let $u,v \in J_{pqs}$ with $uE_{xyz}v$. Pick $A\in \mathscr D$ with $x,y,z,p,q,s,u,v\in A \leq M$. Then $A$ satisfies
$$(\forall m) (\forall n)  (R(x;y,z:m;n,u) \leftrightarrow 
R(x;y,z:m;n,v)) \wedge (R(x;y,z:u;m,n) \leftrightarrow 
R(x;y,z:v;m,n)).$$
Thus, $u,v$ lie in the same leaf of the $D$-set of $A$ witnessing $L(x;y,z)$. Thus, $u,v$ also lie in the same leaf of the $D$-set of $A$ witnessing $L(p;q,s)$, so $A$ satisfies
 $$ (\forall m) (\forall n)  (R(p;q,s:m;n,u) \leftrightarrow 
R(p;q,s:m;n,v)) \wedge (R(p;q,s:u;m,n) \leftrightarrow 
R(p;q,s:v;m,n)).$$
It follows that $u E_{pqs} v$ holds in $M$.
    \item We prove the contrapositive, so suppose we have in $R_{xyz}$ a diagram such as the following.
    \begin{figure}[H]
        \centering
        \begin{tikzpicture}
        \node at (0,0)[above]{$r$};
        \node (p) at (-1,1){$p$};
        \node (q) at (-1,-1){$q$};
        \node (s) at (1,-1) {$s$};
        \node (u) at (3,1) {$u$};
        \node (v) at (3,-1){$v$};
        
        \draw (0,0)--(2,0);
        \draw (p)--(0,0);
        \draw (q)--(0,0);
        \draw (s)--(0,0);
        \draw (u)--(2,0);
        \draw (v)--(2,0);
         \end{tikzpicture}
        \caption{}
        \label{fig:uEv}
    \end{figure}
Using $Q$ and $R$, we see that if $A\in \mathscr D$ with $x,y,z,p,q,s,u,v\in A\leq M$, then the $D$-set of $A$ witnessing $L(x;y,z)$ has the same picture. Thus, the $D$-set of $A$ witnessing $L(p;q,s)$ has $u,v$ in the same leaf, and it follows via the definition of $E_{pqs}$  that $uE_{pqs}v$ holds.
\end{enumerate}
\end{proof}

Now suppose $\langle xyz \rangle < \langle pqs \rangle$. Then by the last lemma $p,q,s$ are inequivalent modulo $E_{xyz}$, so there is a ramification point $r$ of $R_{xyz}$ such that the $E_{xyz}$-classes of $p,q,s$ lie in distinct branches at $r$. Put $f_{\langle xyz \rangle }(\langle pqs \rangle )=r$.

\begin{lem}
\begin{enumerate}[(i)]
    \item In the above notation, the value of $f_{\langle xyz\rangle }(\langle pqs \rangle )$ depends only on the cone at $\langle xyz \rangle$ containing $\langle pqs \rangle$. 
    \item \label{1-1 P classes} $f_{\langle xyz \rangle }$ determines a bijection between the set of cones at $\langle xyz \rangle $ and the set of ramification points of $R_{xyz}$.
   
\end{enumerate}
 \end{lem}
\begin{proof}
\begin{enumerate}[(i)]
    \item
    {\em{Claim 1.}} If $a,b \in J_{pqs}$ are inequivalent modulo $E_{pqs}$, then they lie in distinct branches at $r$.
    
    \begin{proof}
    This is immediate from Lemma \ref{A}(\ref{u v E}).
    \end{proof}
    
    {\em{Claim 2.}} If $\langle xyz\rangle  <  \langle pqs\rangle  < \langle p'q's'\rangle $ then $f_{\langle xyz\rangle}(\langle pqs \rangle)=f_{\langle xyz\rangle}(\langle p'q's' \rangle)$.
    
    \begin{proof}
    In this situation, $p',q',s'$ are inequivalent modulo $E_{p'q's'}$ and hence modulo $E_{pqs}$ (Lemma \ref{A}(\ref{E refines})) so lie in distinct branches at $r$ (by Claim 1).
    \end{proof}
    
   {\em{ Claim 3.}} If $ \langle pqs \rangle $ and $\langle p'q's' \rangle $ are incomparable but in the same cone at $\langle xyz \rangle$, then $f_{\langle xyz\rangle}(\langle pqs \rangle)=f_{\langle xyz\rangle}(\langle p'q's' \rangle)$.  
    
    \begin{proof}
    Pick $p'',q'',s'' $ with $\langle xyz\rangle <$  $\langle p''q''s'' \rangle  < \langle pqs \rangle$, and $\langle xyz \rangle < \langle p''q''s'' \rangle < \langle p'q's' \rangle$. By Claim 2, $f_{\langle xyz \rangle } (\langle pqs \rangle) =f_{\langle xyz \rangle } (\langle p''q''s'' \rangle )=f_{\langle xyz \rangle } (\langle p'q's' \rangle )$. 
    \end{proof}
    
    Part (i) follows.
    
    \item To see that $f_{\langle xyz \rangle }$ is surjective, let $r \in \text{Ram}(R_{xyz})$ and choose $p,q,s \in M$ such that modulo $E_{xyz}$ they lie in distinct non-special branches at $r$. Then $\langle xyz \rangle < \langle pqs \rangle$, by considering finite substructures of $M$. It follows that $f_{\langle xyz \rangle } (\langle pqs \rangle )=r$.
    
    For injectivity, suppose $\langle pqs \rangle, \langle p'q's'\rangle\in K^*/R$. Suppose that there is a finite $A \in \mathscr D$ with $x,y,z,p,q,s,p',q',s' \in A <M$ such that $p,q,s$ and $p',q',s'$ meet at the same ramification point in the $D$-set of $A$  in which $L(x;y,z)$ holds. Then this holds in any $A'\in \mathscr D$ with $A<A'<M$   (e.g.  consider a sequence of one-point extensions between $A$ and $A'$). It follows by semi-homogeneity that there are $u,v,w \in M$ with $\langle xyz \rangle < \langle uvw \rangle$ and $\langle uvw \rangle < \langle pqs \rangle $ and $\langle uvw \rangle < \langle p'q's'\rangle$, so $\langle pqs \rangle$ and $\langle p'q's'\rangle$ lie in the same cone of $K^*/R$ at $\langle xyz \rangle$.
\end{enumerate}
\end{proof}

\begin{lem}\label{union of branches}
Let $\langle xyz \rangle < \langle pqs \rangle$, and let $[m]$ be a pre-direction of $R_{pqs}$ and  $r:=f_{\langle xyz \rangle}(\langle pqs \rangle) \in {\rm Ram} (R_{xyz})$  Then  there is a unique set $t$ of branches of $R_{xyz}$ at $r$ such that $ [m] = \cup \cup t$.  
\end{lem}
\begin{proof}
Consider a finite structure $A \in \mathscr D$ with $x,y,z,p,q,s,m\in A\leq M$.  Consider the vertex $\langle xyz\rangle <\langle pqs \rangle$, and $r:=f_{\langle xyz \rangle }(\langle pqs \rangle)$
(we abuse notation here by interpreting this in $A$ rather than $M$). Write $[m]_A$ for the pre-direction of $m$ in the finite $D$-set labelling the structure tree vertex  $\langle pqs\rangle$ in $A$.
Let $t$ be the set of branches $\{t_1, \dots, t_n \}$ of $A$  at $r$ that corresponds in $A$ to the direction $[m]_A$. Each pre-branch of $t_i, \ i \in \{1, \dots, n\}$ consists of a collection of  pre-directions, say

  $$t_1 =\{u^{(1)}_1,u^{(1)}_2, \dots, u^{(1)}_{m_1}\}$$ 
 $$ t_2 =\{u^{(2)}_1,u^{(2)}_2, \dots, u^{(2)}_{m_2}\}$$
   $$\vdots $$
  $$t_n =\{u^{(n)}_1,u^{(n)}_2, \dots, u^{(n)}_{m_n}\} 
$$
so that $$\bigcup t=\bigcup_{i=1}^{n} t_{i}= t_1 \cup \dots \cup t_n =\{u^{(1)}_1,u^{(1)}_2, \dots, u^{(1)}_{m_1}, \dots, u^{(n)}_1,u^{(n)}_2, \dots, u^{(n)}_{m_n}\}$$
hence $[m]_A=\bigcup \bigcup \{t_1, \dots, t_n \}$. Since this holds for any $A'\in \mathscr D$ with $A<A'<M$, and branches of $M$ at $r$ arise as unions of branches of the corresponding finite structures,  the result follows.
\end{proof}

Define $g_{  \langle pqs \rangle \langle xyz \rangle}([m])=t$ where $[m]=\cup \cup t$ as above. So $g_{ \langle pqs \rangle \langle xyz \rangle }$ is a  map from the directions of $R_{pqs}$ to the power set of the set of non-special branches at $r$. From the definition we see that if  $[m]\neq [m']$ then $g_{ \langle pqs \rangle \langle xyz \rangle}([m]) \cap g_{ \langle pqs \rangle  \langle xyz \rangle}([m']) = \emptyset$.
\begin{lem}\label{gxyz}
The map $g_{\langle xyz \rangle \langle pqs \rangle }$  depends only on  $\langle xyz \rangle, \langle pqs \rangle$, not on the choice of $x,y,z,p,q,s$.
 \end{lem}
 \begin{proof}
 The point essentially is that in any finite structure $A \in \mathscr D$ with  $x,y,z,p,q,s,m \in A<M$, the set $t$ of branches depends just on the direction of $m$ in the $D$-set witnessing $L(p;q,s)$, and on the map $g_{\mu \nu}$ where $\mu$ codes in $A$ the $D$-set witnessing $L(p;q,s)$, and $\nu$ the $D$-set witnessing $L(x;y,z)$.    
     
  \end{proof}

\begin{prop}\label{aboutG}
\begin{enumerate}[(i)]
   
   \item  The group $G_{\langle xyz \rangle }$ is transitive on  ${\rm Ram}(R_{xyz})$.
   \item  \begin{enumerate}[(a)]
   \item The stabiliser $G_{\langle xyz \rangle }=G_{\{J_{xyz}\}}$ induces a transitive group on the subset $J_{xyz}$ of $M$.
   \item The group $G$ is transitive on the semilinear order $(K^*/R,\leq)$.
  
   \end{enumerate}
   \item The group $G_{\{J_{xyz}\}}$ induces a $2$-transitive group on the set of directions of $R_{xyz}$, i.e. is transitive on the set of pairs of distinct directions.
 \item The group $G$ is transitive on the set  $\mathscr X$, where $\mathscr X= \bigcup R_{xyz} $, the union of all the sets of directions in the structure $M$.
\item  The group $G_{\{J_{xyz}\}}$ is transitive on the set of non-special branches of $R_{xyz}$, and for each $r\in {\rm Ram}(R_{xyz})$ and branch $U$ at $r$, the group $G_{\{J_{xyz}\},U}$ induces a transitive group on $U$.
   \item The equivalence relation $E_{xyz}$ is the unique maximal $G_{ \{J_{xyz} \}}$-congruence on $J_{xyz}$.
   
   \end{enumerate}
   
      \end{prop}

\begin{proof} 
\begin{enumerate}[(i)]
    \item  Assume $r , r'$ are two ramification points of $R_{xyz}$and $x,y,z$ and $p,q,s$ are  triples lying in distinct branches at $r,r'$  respectively with $L(p;q,s)$ witnessed in $R_{xyz}$. We want to find some $g \in G_{\langle xyz \rangle }$ such that $r ^g = r'$. Now $R(x;y,z: p;q,s)$ holds. By $3$ -homogeneity there is $g \in G$ such that $(x,y,z)^g =(p,q,s)$. By Lemma \ref{QRIJ}, $J_{xyz}= J_{pqs}$ so $\langle xyz \rangle = \langle pqs \rangle$, so $g \in G_{\langle xyz \rangle }$, and $g$ preserves the $D$-relation on $R_{xyz}$ (Lemma~\ref{4}(iv)), so $r^g=r'$.
\item 
\begin{enumerate}[(a)]
    \item Let $u\in J_{xyz}$, so $R(x;y,z:u;y,z)\vee R(x;y,z:x;u,z)\vee R(x;y,z:x;y,u)$ holds. Let $r:={\rm ram}(x,y,z)$.   To show the transitivity we want to find $g \in G_{\langle xyz\rangle }$ such that $u^g =x$. There are $3$ cases to consider.

{\em{Case 1.}} If $u$ is in the same branch  as $x$ at $r$ in  $R_{xyz}$, then $L(x;y,z)$ and  $L(u;y,z)$ are witnessed in $R_{xyz}$, and hence $\langle xyz\rangle \ =\ \langle uyz\rangle$. By semi-homogeneity of $G$ there is $g \in G$ such that $(u,y,z)^g=(x,y,z)$. As $J_{xyz}=J_{uyz}$ then the $D$-set is fixed and hence, by Lemma \ref{QRIJ}(ii),  $\langle xyz \rangle $ is fixed so $g \in G_{\langle xyz \rangle}$.

{\em{Case 2.}} If $x,y,z,u$ are in distinct branches at a ramification point $r$, then $L(x;y,z) \wedge L(x;u,y) \wedge L(x;u,z)$ hold and by semi-homogeneity there exists $w$ in the same branch as $u$ at $r$ such that $L(u;w,z)$ is witnessed at $r':={\rm ram}(u,w,z)$  in the same $D$-set $R_{xyz}$,  i.e. $R(x;y,z: u;w,z)$ holds. Therefore, $L(x;y,z) \wedge L(u;w,z)$ holds, so  there is $g \in G$ such that $(u,w,z)^g = \ (x,y,z)$. As $R_{uwz}=R_{xyz}$, $g$ fixes the $D$-set and hence fixes $\langle xyz \rangle$, so $g \in G_{\langle xyz \rangle}$.

{\em{Case 3.}} Suppose that $u$ is in the same branch as $z$ (the argument is the same if it is in the same branch as $y$). If $u $ is  special at $\text{ram}(u,y,z)$, then there is some $g \in G$ fixing $y,z$ and taking $u$ to $x$, and as $R_{xyz}=R_{uyz}$, the $D$-set is fixed by $g$, hence $\langle xyz \rangle$ is fixed, so $g \in G_{\langle xyz\rangle}$. Otherwise,  by semi-homogeneity, there is some $w$ such that $L(u;w,y)$ is witnessed in $R_{xyz}$. Again, there is $g \in G$ with $(u,w,y)^g=(x,y,z)$, as required.
\item  Let $\langle xyz \rangle $, $\langle pqs \rangle \in (K^*/R, \leq)$. Then $M \models L(x;y,z) \wedge L(p;q,s)$ so by semi-homogeneity there is $g \in G$ with $(p,q,s)^g=(x,y,z)$. Then $\langle pqs \rangle^g =\langle xyz \rangle$.

\end{enumerate}

\item Let $[ p] \neq [q]$ be distinct directions of $R_{xyz}$ with $[ p]=p/E_{xyz}, [q]=q/E_{xyz}$ and put $[ x]= x/E_{xyz}, [y]=y/E_{xyz}$. It suffices to show there is $g \in G_{\{J_{xyz}\}}$ with $([x], [y])^g=([p], [q])$. Choose $s \in M$ such that $R(x;y,z: p;q,s)$ holds - this exists by semi-homogeneity. Using $3$-homogeneity (Lemma \ref{Aut(M)}(\ref{3 homom})) there is $g \in G$ with $(x,y,z)^g=(p,q,s)$. Since $R(x;y,z:p;q,s)$ holds, $g$ fixes $J_{xyz}$ setwise, so $g$ preserves $E_{xyz}$ so fixes $R_{xyz}$ setwise, and clearly $([ x], [y])^g= ([ p], [ q])$.

\item This follows from(ii)$(b)$ and (iii). 

\item Let $U$ be the branch containing $y$ at $r:={\rm ram}(x,y,z)$ in $R_{xyz}$, and let $L(p;q,s)$ be witnessed in $R_{xyz}$. Put $r':={\rm ram}(p,q,s)$, and let $V$ be the branch at $r'$ containing $q$. It sufices to show that  there is $g\in G_{\{J_{xyz}\}}$ with $V^g=U$. But this is immediate -- we may choose any $g$ with $(p,q,s)^g=(x,y,z)$, as exists by semi-homogeneity.

For the second assertion, with $x,y,z,r,U$ as above, let $w\in U$. Then $L(x;z,w)$ is witnessed in $R_{xyz}$, and by semi-homogeneity there is  $g\in G$ with $(x,y,z)^g=(x,w,z)$. Clearly $g$ fixes $U$ setwise.

\item

Maximality of $E_{xyz}$ follows immediately from $2$-transitivity of $G_{\{J_{xyz}\}}$ on $R_{xyz}=J_{xyz}/E_{xyz}$, and this was proved in (\romannum{3}).

It remains to prove that $E_{xyz}$ is the {\em{unique}} maximal $G_{\{J_{xyz}\}}$-congruence. To see this, suppose $E^*$ is a $G_{\{J_{xyz}\}}$-congruence on $J_{xyz}$ and $E^* \not \subset E_{xyz}$. Without loss of generality, we may suppose $x E^* y$. Let $x' \in J_{xyz}$ with $xE_{xyz}x'$. Then $L(x;y,z) \wedge L(x';y,z)$. It follows by semi-homogeneity that there is a $g \in G$ with $(x,y,z)^g=(x',y,z)$. Then $J_{xyz}^g=J_{xyz}$, and as $y^g=y$, $g $ fixes $E^*(y)$ setwise, so as $x E^* y$ we have $x'E^* y$. Thus $x/E_{xyz}\subset E^*(y)$. Hence $E_{xyz}\subset E^*$ and it follows that $E^*$ is universal, as required.

\end{enumerate}
\end{proof}

\section{Proof of the main theorem}\label{Jordan}
In this section, we prove that $G={\rm Aut}(M)$ is a Jordan group preserving a limit of $D$-relations. The main work is in Section~\ref{findjordan}, where we  first show that every pre-direction is a Jordan set, and then use Lemma~\ref{connected} to identify other Jordan sets. We then show that $G$ does not preserve on $M$ any structure of types (a)-(c) in Theorem~\ref{am-class}.  Finally,
 in  Section \ref{proof of main Th} we prove that $G$ satisfies the requirements of Definition \ref{limits} to obtain our main result, Theorem~\ref{mainthm}.

\subsection{Finding a Jordan set} \label{findjordan}

Recall first that $T=(K^*/R,\leq)$ is a lower semilinear order and meet semilattice. We refer to it as   the {\em structure tree}  of $M$.

\begin{mydef}{\em{ 
     A subset $\hat U$ of $M$ is said to be a {\em{pre-branch}} if there are $x,y,z \in M$ with $L(x;y,z)$ and a branch $ U$ in $R_{xyz}$ such that $\hat U= \{ u\in M: [u] \in  U\} = \bigcup \{[u]: [u]\in  U\}$; that is,  $\hat U$ is the union of all $E_{xyz}$-classes in the single  branch $U$ at some ramification point $r$ of the $D$-set  $(R_{xyz}, D_{xyz})$. In this siruation, we say $\hat U$ is a pre-branch at the ramification point $r$. 
         }}
\end{mydef}
Given a $D$-set $R_{xyz}$ we put $\hat R _{xyz}= \bigcup R_{xyz}$, the union of the pre-directions (see Definition~\ref{longdef}) corresponding to directions of $R_{xyz}$; so $\hat R_{xyz} \subset M$, and in the notation of Section 4,  $\hat R_{xyz}=J_{xyz}$ .

Fix a direction $[n]$ of $M$ (so $n \in M$). Let $\langle pqn \rangle$ be the unique vertex of the structure tree of $M$, whose $D$-set $R_{pqn}$ has $[n]$ as a direction (the uniqueness is noted after Definition~\ref{longdef}); we may suppose $p,q\not\in [n]$.  Note that $L(p;q,n)$ is witnessed in this $D$-set. Define $I:=\{i \in T: i<\langle pqn \rangle\}$, and for each $i \in I$ let $R_i$ be the $D$-set indexed by $i$. Let $D_i$ denote the corresponding $D$-relation $D_{xyz}$, where $R_i= R_{xyz}$. Then $I$ carries a total order $<$ induced from $T$, where $i<j\Leftrightarrow \hat R_j\subset \hat R_i$.
   For each $i \in I$, let $r_i=f_i(\langle pqn \rangle)$, the ramification point of $R_i$ corresponding to the cone (of the structure tree) at $i$ containing $\langle pqn \rangle$.
By Lemma~\ref{gxyz} there is a set $S_i$ of branches at $r_i$ such that $g_{\langle pqn \rangle i}([n])=\bigcup \bigcup S_i$.
  
Since our goal is to show that $[n]$ is a  Jordan set for $G$, we consider the induced structure on $[n]$, viewed as a predirection, i.e. as a subset  of $M$. First, for each $i \in I$, there is an equivalence relation $F_i$ on $[n]$ defined by $$ d_1 F_i d_2 \Leftrightarrow d_1, \ d_2 \  \text{lie in the same pre-branch of}\ \hat R_i\ \text{at}\ r_i.$$
  Also for each $i \in I$, let $E_i$ be the equivalence relation $E_{xyz}$ (restricted to $[n]$), where $R_i =R_{xyz}$.

 \begin{lem}\label{FE}
 Let $i,j \in I$ with $i<j $. Then $E_{i}\subset F_i \subset E_j \subset F_j$.
 \end{lem}
 \begin{proof}
 Take a particular pre-branch at $r_i$ in $\hat{R_i}$ lying in $[n]$, say $\hat U_i$. By the definition of the relation $F_i$ the pre-branch $\hat U_i$ is an $F_i$-class. It is clear that the relation of being in the same pre-direction of $\hat R_i$ refines $F_i$, so $E_i \subset F_i$.  Similarly we have $ E_j \subset F_j$. To show that $F_i \subset E_j$, we see by Lemma~\ref{union of branches} that if $[m]$ is a pre-direction for some $R_j$ where $j \in I$ with $j>i$, then $[m]$ is a union of pre-branches of $\hat R_i$ at $r_i$.
 \end{proof}
 
\begin{lem}\label{intersection}
Given an $F_i$-class $\hat U_i$, the intersection of the $E_j$-classes containing $\hat U_i$ (for $j>i$) is just $\hat U_i$.
\end{lem}
\begin{proof}
We want to show that $F_i =\bigcap\limits_{j>i} E_j$. It is clear from Lemma~\ref{FE} that $F_i \subseteq \bigcap\limits_{j>i} E_j $. Conversely, suppose $u,v \in [n]$ with $\neg u F_i v$. We want to find $j\in I$ with $j>i$ such that $\neg u E_j v$. Let $a \in M$ lie in the special branch of $R_i$  at $r_i$ (so $a\not\in [n]$). Consider a finite structure $A\in \mathscr D$ containing elements $a',u',v',w',s',t',p',q',n'$ in distinct branches at a ramification point $r$ at the root $D$-set (with $a'$ special), such that there is a $\mathscr L$-isomorphism $(a',u',v',p',q',n')\to (a,u,v,p,q,n)$ . We choose $A$ so that also  in a higher $D$-set (below that witnessing $L(p';q',n')$) we have $L(w';s',t')$, with $u',v',w',s',t'$ again in distinct branches at a ramification point.  We may suppose $A\leq M$. By semi-homogeneity there is $g \in G$ with $(a',u',v',p',q',n')^g=(a,u,v,p,q,n)$. The relation $L({w'}^g; s'^g,t'^g)$ will be witnessed in a $D$-set $R_j$ with $j>i$, and we have $\neg u E_j v$. (The role of $p,q,p',q'$ here is to ensure $j\in I$, that is, $j<\langle pqn\rangle$.)

\end{proof}

\begin{lem}\label{largest}
Let $u,v_1,\dots v_m$ be distinct elements of $[n]$. Then there is a greatest $i\in I$ such that $u$ is $E_i$-inequivalent to each of $v_1, \dots , v_m$ and for such $i$ the element $u$ will be $F_i$-equivalent to at least one  $v_j$ with $j \in \{ 1, \dots, m \}$.
\end{lem}
\begin{proof}
Find $i_0\in I$ with $i_0< \langle pqn \rangle$ containing elements $p,q,u,v_1, \dots, v_m$ all lying in distinct branches at the ramification point $r_{i_0}$ of the $D$-set $R_{i_0}$. Consider finite $A\leq M$ with $A \in \mathscr D$ and $p,q,u,v_1, \dots, v_m$ lying in distinct branches at a ramification point of the root $D$-set. 

By considering the structure of $A$, we see that there is $i$ with $i_0<i<\langle pqn \rangle$ such that at $r_i$, $u$ is in the same branch as at least one of the $v_i$, but in a distinct direction to each. (Working in $A$, consider the $D$-sets in the structure tree between the root and the $D$-set $R_{pqn}$, and the corresponding ramification points; there will be a least $D$-set such that $u$ lies in the same branch as some $v_j$ at the relevant ramification point.)  

\end{proof}
\begin{mydef}\label{C and D}{\em{
     For each $i \in I$, define a relation $C_i$ on $\bigcup S_i$ (so on the set of directions of $R_i $ lying in the branches of $S_i$) as follows: if $[x], [y],[z]\in \bigcup S_i$, then $C_i ([z];[x],[y]) \leftrightarrow D_i([x],[y];[z],[w])$ for any direction $[w]$ of ${R_i}$ lying outside $\bigcup S_i$. }}
 \end{mydef}

It is easily seen that for each $i\in I$,  $C_i$ induces a $C$-relation on each $F_i$-class of $[n]$ (considered as a branch at $r_i$, i.e. modulo $E_i$). Furthermore, $C_i$ is invariant under $G_{M\setminus (\bigcup \bigcup S_i)}$, or under the subgroup of $G$ stabilising both $[n]\subset M$ and $\hat R_i$ setwise. 
\begin{lem}\label{g is aut}
Let $g$ be a permutation of $M$ which is the identity on $M\setminus [n]$, and for each $i \in I$ preserves the equivalence relation $E_i$, the relations $L$ and $S$ on $[n]$, and the $C$-relation induced by $C_i$ on each $F_i$-class of $[n]$ (modulo $E_i$). Then $g \in G$.  
\end{lem}
\begin{proof}
By Lemma \ref{intersection} and the assumption that $g $ preserves the relations $E_i$, it follows that $g$ preserves each $F_i{\big|}_{[n]} $ and hence each
$F_i$. By Lemma~\ref{LL'SS'}, it is enough to show that $g$ preserves $L$ and $S$ on $M$.
Below, in Part A we show that $g$ preserves $L$, and in Part B that it preserves $S$.

{\em \bf{Part A.}} To prove that $g\in G$ preserves $L$, we consider four  cases:

{\em{Case $\Romannum{1}$}}. If $x,y,z \notin [n]$, then as $g$ is the identity on $M\setminus [n]$ we have $L(x;y,z) \leftrightarrow L(x^g;y^g,z^g)$, and likewise for the other orderings of $\{x,y,z\}$.

{\em{Case $\Romannum{2}$.}} Let $x \in [n]$, $y,z \in M \setminus [n]$. 
Let $R$ be the $D$-set in which $L\{ x,y,z \}$ is witnessed with $x,y,z$ lying in distinct branches at the ramification point $r$ of $R$.
We need to show that the map $xyz\mapsto x^gyz$ preserves $L$.  
We will consider the possible cases based on where the $D$-set $R$ witnessing $L\{x,y,z\}$ could be.

{\em Sub-case 1.} Assume that the $D$-set $R$ is  $R_{\langle pqn \rangle}$, and let $L\{x,y,z\}$ hold, witnessed in $R$. Now $x, x^g$ lie in the same element of $R$, and $y,z$ lie in two other distinct elements of $R$, fixed by $g$. It is therefore immediate that the map $xyz\mapsto x^gyz$ preserves $L$.

 {\em{Sub-case 2.}} Assume that the $D$-set $R=R_i$ is lower than $R_{\langle pqn \rangle}$ (so $i\in I$). Now $L(x;y,z)$ cannot be witnessed  in this $D$-set at $r_i$, because $x\in  [n]$ so cannot lie in the special branch at $r_i$. However, possibly $L(y;x,z)$ is witnessed at $r_i$ (likewise $L(z;x,y)$) and then since $x^g \in \bigcup\bigcup S_i$  (as $x^g \in [n]$),  the relation $L(y;x^g,z)$ holds
.

  Suppose that as in  Figure \ref{sub2 L(x;y,z)} $L(x;y,z)$ holds in $R_i$ with $x$  special at another ramification point $r_i'$ not within $\bigcup S_i$. Then again because $x^g \in \bigcup\bigcup S_i$,  and since $x$ and $x^g$ lie in the same branch at $r_i'$ we get $L(x^g;y,z)$. Similarly if $L(y;x,z)$ holds at $r_i'$ (likewise for $L(z;x,y)$), then $x$ and $x^g$ will be in the same branch at $r_i'$ and it is readily seen that $L(y;x^g,z)$ holds. It cannot happen that $L\{x,y,z\}$ is witnessed at $r_i'$ within $\bigcup\bigcup S_i$, since $y,z$ are not in $\bigcup\bigcup S_i$ so would lie in the same branch at such $r_i'$.
 \tikzset{middlearrow/.style={
        decoration={markings,
            mark= at position 0.75 with {\arrow{#1}} ,
        },
        postaction={decorate}
    }
}
      
      \begin{figure}[H]
      \centering
 \begin{tikzpicture}[scale=1]
\node (ri) [below] at (0,0){$r_i$};
\node (x) at (1,1) {$x$};
\node (xg) at (1,-1){$x^g$};
\node at (-0.75, 1) {};
\node (z) at (-2,1) {$z$};
\node (y)[below] at (-2,-1){$y$};
\node (ri') [below] at (-1,0){$r_i'$};

\draw (x)--(0,0);
\draw (xg)--(0,0);
\draw  (-0.75, 1)--(0,0);
\draw (z)--(-1,0);
\draw (-2,-1)--(-1,0);
\draw [middlearrow={stealth}](0,0)--(-1,0);
\end{tikzpicture}

      \caption{}
      \label{sub2 L(x;y,z)}
  \end{figure}

{\em{Sub-case 3.}} Assume that the $D$-set $R$ is higher than $R_{\langle pqn \rangle}$. Now the direction containing $x$ in $R$ contains the whole of $[n]$, so is fixed by $g$, as are $y$ and $z$. It follows that $x,y,z$ and $x^g,y,z$ satisfy the same $L$-relation. 

{\em{Sub-case 4.}} Suppose the $D$-set $R$ corresponds to the vertex $k$ of the structure tree with $k$ incomparable with $\langle pqn \rangle$. Let $i=\text{inf}\{\langle pqn \rangle, k\}$, so $i \in I$. We may suppose that the cone of $k$ at $i$ (in the structure tree) corresponds to the ramification point $r'$ of $R_i$; then $r'\neq r_i$. Since $L\{x,y,z\}$ is witnessed in $R$, $x,y,z$ lie in distinct non-special branches at $r'$. Hence, as $y,z \notin \bigcup \bigcup S_i$, it follows that $r'$ cannot be  a ramification point of $\bigcup S_i$, and we have, for example, the picture below in $R_i$.
\begin{figure}[H]
    \centering
    \begin{tikzpicture}[scale=1]
\node (ri) [below] at (0,0){$r_i$};
\node (x) at (1,1) {$x$};
\node (xg) at (1,-1){$x^g$};
\node at (-0.75, 1) {};
\node (z) at (-2,1) {$z$};
\node (y)[below] at (-2,-1){$y$};
\node (ri') [below] at (-1,0){$r'$};

\draw (x)--(0,0);
\draw (xg)--(0,0);
\draw (z)--(-1,0);
\draw (-2,-1)--(-1,0);
\draw (0,0)--(-1,0);
\end{tikzpicture}
    \caption{}
    \label{fig:sub-case 3}
\end{figure}
Now, $x,x^g$ lie in the same branch at $r'$ so in the same pre-direction of $R$, so the same $L$-relation holds among $x,y,z$ and $x^g,y,z$.

{\em{Case $\Romannum{3}$.}} Let $x,y \in [n]$ and $z \in M\setminus [n]$.

{\em {Sub-case 1.}} Suppose that the $D$-set $R$ is $R_{\langle pqn \rangle}$. The relation $L\{x,y,z \}$ is not witnessed here because $x$ and $y$ are in the same direction in $R_{\langle pqn \rangle}$.

{\em{Sub-case 2.}} Suppose that the $D$-set $R$ is lower than $R_{\langle pqn \rangle}$, say $R=R_i$ for  some $i\in I$. If $\neg x F_i y$, then the relation $L(x;y,z)$ or $L(y;x,z)$ cannot be witnessed at $r_i$ because neither $x$ nor $y$ can be special at $r_i$. If $L(z;x,y)$ holds at $r_i$ then $L(z;x^g,y^g)$ is witnessed in $R_i$ (for $x^g , \ y^g$ are in distinct branches at $r_i$ because $F_i$ is preserved on $[n]$). 

If $xF_i y$ and $L(x;y,z)$ is witnessed at $r$ (see Figure \ref{fig:Fig1} below) then we want to see that $L(x^g;y^g,z)$ holds (and likewise if $L(y;x,z)$ holds). There is $t\in [n]$ such that $tF_i x \wedge tF_i y$  and $L(x;y,t)$ holds. Since $g$ preserves $L$  on elements of $[n]$,  we get $L(x^g;y^g,t^g)$, and then $L(x;y,z) \leftrightarrow L(x;y,t)$ and $L(x^g;y^g,z^g) \leftrightarrow L(x^g;y^g,t^g)$ as $g$ preserves $C_i$. Also, $L(x;y,t) \leftrightarrow L(x^g;y^g,t^g)$ as $x,y,t \in [n]$, so $L(x;y,z) \leftrightarrow L(x;y,t) \leftrightarrow L(x^g;y^g,t^g) \leftrightarrow L(x^g;y^g,z^g)$.
\tikzset{middlearrow/.style={
        decoration={markings,
            mark= at position 0.75 with {\arrow{#1}} ,
        },
        postaction={decorate}
    }
}
\begin{figure}[H]
    \centering
    \begin{tikzpicture}[scale=1.5]
\node [below] at (0,0){$r_i$};
\node(v) at (1,1){};
\node(w') at (1,-1){};
\node(w) at (-1,1){$w$};
\node(z) at (-1,-1){$z$};
\node(sg) at (2,0){};
\node(x) at (0.5,1){$x$};
\node(y) at (1,0.5){$y$};
\node(alpha')[below] at (1,0){};
\node(r) [below] at (0.5,0.5){$r$};
\node [below] at (0.25,0.25){};
\node (xg) at (0,1){$x^g$};
\node (yg) at (1,0.25) {$y^g$};
\node at (0.125,0.125){};
\node (t) at (-0.25,1){$t$};

\draw [middlearrow={stealth}](0.25,0.25)--(0,0);
\draw  (v)--(0,0);
\draw (w)--(0,0);
\draw (w')--(0,0);
\draw (z)--(0,0);
\draw (sg)--(0,0);
\draw [middlearrow={stealth}](x)--(0.5,0.5);
\draw [dashed] (0.5,0.5)--(y);
\draw [middlearrow={stealth}] (xg)--(0.25,0.25);
\draw (0.25,0.25)--(yg);
\draw (0.125,0.125)--(t);
\end{tikzpicture}
    \caption{}
    \label{fig:Fig1}
\end{figure}

{\em{Sub-case 3.}} Suppose that the $D$-set $R$ is higher than $R_{\langle pqn \rangle}$. Then $L\{ x,y,z \}$ cannot be witnessed in $R$ because $x,y$ are in the same direction of $R$.

{\em{Sub-case 4.}} Assume that $R$ is the $D$-set of the vertex $k$ incomparable with $\langle pqn \rangle$, and put $i=\text{inf}\{k,\langle pqn \rangle\}$. Then the cone of $k$ at $i$ corresponds to a ramification point $r'$ of $R_i$ distinct from $r_i$, and as $x,y,z$ lie in distinct branches of $R_i$ at $r'$, we must have that $r'$ is a ramification point of $S_i$, as in the diagram,
\begin{figure}[H]
    \centering
    \begin{tikzpicture}[scale=1]
    \node at (0,0)[below]{$r_i$};
    \node at (0.5,0) {};
    \node at (1,0)[below]{$r'$};
    \node (x) at (2,0) [right]{$x$};
    \node (y) at (2,1) {$y$};
    \node (t) at (1,1.5) {$t$};
    \node (z) at (-1.5,0.5) {$z$};
    \node at (2.5,1.25)[right]{$S_i$};
   \draw[decoration={brace,mirror,raise=8pt},decorate,thick]
  (3,0) -- node[right=8pt] {} (1.75,2);
    \draw (z)--(0,0);
    \draw (t)--(0.5,0);
    \draw (y)-- (1,0);
    \draw (x)--(0,0);
    \end{tikzpicture}
    \caption{}
    \label{fig:incomp}
\end{figure}

Choose $t$ as depicted, in the same branch as $z$ at $r'$ and the same branch as $x$ at $r_i$. As $g$ preserves $L$ on $[n]$ and the $C$-relation on $F_i$-classes (considered modulo $E_i$), we have $$L(x;y,z)\Leftrightarrow L(x;y,t) \Leftrightarrow L(x^g;y^g,t^g) \Leftrightarrow L(x^g;y^g,z),$$
and likewise for other permutations of $\{x,y,z\}$.

{\em{Case $\Romannum{4}$.}} If $x,y,z \in [n]$, then $L(x;y,z) \leftrightarrow L(x^g;y^g,z^g)$ follows immediately by the hypothesis that $g$ preserves $L$ on $[n]$.

{\em\bf{Part B.}} To prove that $g$ preserves $S$, we argue as in Part A. 

\end{proof}

\begin{lem}\label{Jordan set}
Each pre-direction $[n]$ is a Jordan set of $G$.
\end{lem}
\begin{proof}
To show this, we want to define a group $K\leq G$ which is transitive on $[n]$ and fixes the complement $M\setminus [n]$. We  construct $K$ as an iterated wreath product of groups of automorphisms of $C$-relations. The argument is similar to the proof of Proposition 5.6 in \cite{bhattmacph2006jordan}, but there is an imprecision there: the map $\chi$ below is not defined precisely in \cite{bhattmacph2006jordan}, leading to problems with the proof of Claim 8. The approach given here works in \cite{bhattmacph2006jordan} too. 

Write $[n]=\{u_i:i \in \omega\}\subset M$. For each $u \in [n]$ and $i \in I$, put $[u]_i:=\{x \in M: x F_i u \}$ (the pre-branch of $M$ at $r_i$ containing $u$). 
For each $i \in I $ and $u \in [n]$, define $A^i(u):=[u]_i/E_i$ (the branch at $r_i$ containing $u$). In particular,
 for each $i \in I$, let $V_i: = [u_0]_i/E_i $, so $V_i=A^i(u_0)$. Let $e_i:=u_0/E_i \in V_i$, where $u_0/E_i$ denotes the $E_i$-class of $u_0$. Define $$\Omega:=\{f:I \rightarrow \bigcup_{i\in I} V_{i} : f(i)\in V_i \  \text{for all }\  i,\ \text{supp}(f) \ \text{finite} \}$$
where supp$(f)=\{i \in I:f(i)\neq e_i\}$.

We aim to find a system of maps $\phi^i_V:V \rightarrow V_i$, where $i \in I$ and $V$ ranges through branches $A^i(u)$ for $u\in [n]$.
Given such maps, define $\chi:[n]\rightarrow \Omega$ by $\chi (u)(i)=\phi^i_{A^i(u)}(u/E_i)$ for all $u \in [n]$ and $i \in I$. We need to define the maps $\phi^i_V$ so that $\chi$ is a bijection. The definition of  $\chi$ is inductive, done in parallel with the definition of the $\phi^i_V$. As the base case,
define $\chi (u_0)$ so that $\chi (u_0)(i)=e_i$ for all $i \in I$. 

Suppose that $\chi(u_0), \dots , \chi (u_{k-1})$ have been defined. We may suppose that each map $\phi^i_{A^i(u_l)}$ has been defined, for all $l<k$, and all $i \in I$. 

Let $i(k)$ be the largest $i \in I$ such that $u_k$ is $E_i$-inequivalent to $u_l$ for each $l<k$. This exists, by Lemma~\ref{largest}, and by that lemma there is some $l<k$ such that $u_l F_{i(k)}u_k$, so $u_lF_i u_k$, that is $A^i(u_l)=A^i(u_k)$,  for all $i\geq i(k)$. Now by assumption $\phi^i_{A^i(u_l)}$ has been defined for all $i \in I$, so $\phi^i_{A^i(u_k)}$ has been defined for all $i\geq i(k)$, but not for $i<i(k)$. For $i <i(k)$, choose $g_i \in G$ such that $(A^i(u_k))^{g_i}=V_i$ and $(u_k/E_i)^{g_i}=e_i$ (this exists, since by Proposition~\ref{aboutG}(v) $G$ is transitive on the set of non-special  branches and induces a transitive group on each branch). Then put $\phi^i_{A^i(u_k)}(u/E_i)=(u/E_i)^{g_i}$, for all $i<i(k)$ and $u$ with  $uF_i u_k$.
Observe that the maps $\phi^i_{A^i(u_l)}$ are now defined for all $l\leq k$ and all $i \in I$.

{\em{Claim 1.}} With the maps $\phi^i_{A^i(u)}$ so defined, we have $\chi(u_k) \in \Omega$ for each $k \in \omega$.

\begin{proof}
This is by induction on $k$. It is immediate that $\chi(u_0) \in \Omega$, so assume it holds for all $l<k$. By construction, as $\phi^i_{A^i(u_k)}$ is a bijection $[u_k]_i/E_i \rightarrow V_i$, we have $\chi (u_k)(i)\in V_i$. We must show supp$(\chi(u_k))$ is finite. There is $l<k$ such that for $i\geq i(k)$, $\chi(u_k)(i)=\chi(u_l)(i)$, so  supp($\chi(u_k) )\cap \{j\in I: j\geq i(k)\}= \text{supp}(\chi(u_l)) \cap \{ j \in I:j\geq i(k)\} $, so by induction is finite. By construction, $\chi(u_k)(i)=e_i$ for all $i<i(k)$, and the claim follows. 
\end{proof}

{\em{Claim 2.}} $\chi:[n] \rightarrow \Omega$ is a bijection.

\begin{proof}
We first show that $\chi$ is injective. So suppose $l<k$. We must show $\chi (u_l)\neq \chi(u_k)$. Pick $i$ such that $u_k F_i u_l$ and $\neg u_k E_i\  u_l$. Then $[u_k]_i=[u_l]_i$, but $[u_k]_i/E_i \neq [u_l]_i/E_i$, so as $A^i(u_k)=A^i(u_l)$, $\chi (u_k)(i)= \phi^i_{A^i(u_k)}(u_k/E_i)\neq \phi^i_{A^i(u_l)}(u_l/E_i)=\chi(u_l)(i)$.

To see surjectivity, suppose for a contradiction that $\chi $ is not surjective, and let $f \in \Omega \setminus \text{Range}(\chi)$ have minimal support, with supp$(f)=\{i_1, \dots, i_t\}$ where $i_1< \dots <i_t$. Define $f' \in \Omega$ where $f'(i_1)=e_{i_1}$, and $f'(j)=f(j)$ for all $j\neq i_1$.

By minimality of supp$(f)$, there is $u \in [n]$ with $\chi(u)=f'$. Let $v=f(i_1)\in V_{i_1}$, and let $k$ be least such that $u_k$ lies in the $E_{i_1}$-class $(\phi^{i_1}_{A^{i_1}(u)})^{-1}(v)$.

To obtain a contradiction and thereby to prove surjectivity, it suffices to prove 

{\em{Sub-claim 1.}} $\chi (u_k)=f$.

\begin{proof}
Certainly $\chi (u_k)(i_1)=\phi^{i_1}_{A^{i_1}(u_k)}(u_k/E_{i_1})=v=f(i_1)$. For $j>i_1$,  $\chi (u_k)(j)=\phi^{j}_{A^{j}(u_k)}(u_k/E_{j})=\phi^{j}_{A^{j}(u)}(u/E_{j})=f(j)$.  Also, $i(k) \geq i_1$, for otherwise there is  $l<k$ such that $u_lE_{i_1} u_k$,  contradicting minimality of $k$. Hence $\chi(u_k)(j)=\phi^j_{A^j(u_k)}(u_k/E_j)=(u_k/E_j)^{g_j}=e_j=f(j)$ for all $j<i_1$, so indeed $\chi(u_k)(j)=f(j)$ for all $j$.
\end{proof}

\end{proof}

For each $i\in I$, Let $H_i$ be the group induced by $G_{\{V_i\}}$ on $V_i$. For each triple $(i,g,h)$, where $i \in I$, $g:(i,\infty) \rightarrow \bigcup\limits_{j>i}V_{j}$ with $g(j) \in V_j$ for all $j$, and $h \in H_i$, define the function $x(i,g,h):\Omega \rightarrow \Omega$ as follows:  
 
 \[  
 f^{x(i,g,h)}(j)=
 \begin{cases}
     f(i)^h  & \text{if} \ j=i \ \text{and} \ f{\big |}_{(i,\infty)}=g, \\
    f(j)  & \text{otherwise} \\
     \end{cases}
\]
 
 
 Now define $K$, the {\em{generalized wreath product}}, to be the subgroup of $\text{Sym}(\Omega)$ generated by permutations $x(i,g,h)$ where $i,g,h$ are as above. By \cite{hall1962wreath}, Lemma 1, the group $K$ is transitive on $\Omega$. Thus, $K$ has an induced transitive action on $[n]$, given by $u^x= \chi^{-1}((\chi (u))^x)$ for all $x \in K$ and $u \in U$. (Note that we  use Cameron's notation for the permutation group $K$, as was  also  used in \cite{bhattmacph2006jordan}.)
 We extend this action to the whole of $M$ by putting $v^x=v$ for all $v \notin [n]$.\\
 
 {\em{Claim 3.}} In this action, $K$ is a subgroup of Aut$(M)$.
 
 \begin{proof}
 It suffices to show that elements $x(i,g,h)$ as above are automorphisms of $M$, and for this we use Lemma \ref{g is aut}. First, observe 
 
 {\em{Sub-claim 2.}} For $u,v \in [n]$, and $i \in I$, $uE_i v \Leftrightarrow \chi(u)(j)=\chi(v)(j)$ for all $j\geq i$.
 
 \begin{proof}
 If $uE_i v$ then $A^j(u)=A^j(v)$ for all $j\geq i$, so $\chi (u)(j)= \phi^j_{A^j (u)}(u/E_j)=\phi^j_{A^j (v)}(v/E_j) = \chi(v)(j)$ for all $j\geq i$. Conversely, if $\neg u E_i v$, then by Lemma~\ref{largest} there is $j \geq i$ such that $uF_j v$ and $\neg uE_j v$. Then $A^j(u)= A^j(v)$, so $\chi (u)(j)= \phi^j_{A^j (u)}(u/E_j)\neq \phi^j_{A^j (v)}(v/E_j) = \chi(v)(j)$, as required.
 \end{proof}
 
 Since $x(i',g,h)$ acts as a permutation in the single coordinate $i'$, in its action on $\Omega$, it is clear that for $u,v \in [n]$ and $i \in I$, we have $\chi(u)(j)=\chi(v)(j)$ for all $j\geq i$ if and only if $\chi(u)^{x(i',g,h)}(j)=\chi(v)^{x(i',g,h)}(j)$ for all $j\geq i$. Thus, $uE_i v$ if and only if $u^{x(i',g,h)} E_i v^{x(i',g,h)}$, so the maps $x(i',g,h)$ preserve all the equivalence relations $E_i$. 
Thus, by Lemma~\ref{intersection}, the maps  $x(i',g,h)$ also preserve all the $F_i$.

 For $u,v,w \in [n]$, put
 $$\sigma (u,v,w)=\text{Max}\{i:u/E_i, v/E_i, w/E_i\  \text{are all distinct} \}.$$
 $$\mu (u,v,w)=\text{Max}\{i:u/E_i, v/E_i, w/E_i\  \text{are not  all equal} \}.$$
 Then $\mu (u,v,w) \geq \sigma(u,v,w) $, and $\mu(u,v,w)=\sigma (u,v,w)$ if and only if there is $i$ (namely $\sigma(u,v,w)$) such that $u,v,w$ are $F_i$-equivalent but not $E_i$-equivalent.
 
 Suppose $\mu(u,v,w)=\sigma (u,v,w)=i$. Let $C_i$ be as in Definition \ref{C and D} with the invariance properties noted there, and note that $C_i$ induces a $C$-relation on $V_i$. Then since the map $\phi^i_{A^i (u)}$ is induced by an element of $G$, we have $$C_i(u/E_i;v/E_i,w/E_i)\leftrightarrow C_i(\phi^i_{A^i (u)}(u/E_i); \phi^i_{A^i (v)}(v/E_i), \phi^i_{A^i (w)}(w/E_i)).$$ 
 It follows that under the assumption $\mu(u,v,w)= \sigma(u,v,w)=i$, the fact that $C_i(u/E_i;v/E_i,w/E_i)$ holds depends just on $\chi(u)(i), \chi(v)(i), \chi(w)(i)$. Similarly, the fact that $L(u;v,w)$ holds depends just on $\chi(u)(i), \chi (v)(i), \chi(w)(i)$. And if $u,v,w,z$ are all $F_i$-equivalent but $E_i$-inequivalent, the fact that $S(u,v;w,z)$ holds depends just on $\chi(u)(i)$, $\chi (v)(i)$, $\chi(w)(i)$ and $ \chi(z)(i)$. We call this phenomenon {\em{tail-independence}}.
 
 {\em{Sub-claim 3.}} The group $K$ preserves the $C$-relation induced by $C_i$ on the branches at $r_i$.
 
 \begin{proof}
 Suppose $u,v,w$ lie in the same $F_i$-class but distinct $E_i$-classes, so $\mu(u,v,w)= \sigma(u,v,w)=i$, and assume $C_i(u;v,w)$ holds in this branch. Let $x=x(i',g,h) \in K$. If $i'>i$, then $\chi(u)(i)=\chi(u^x)(i)$, $\chi(v)(i)=\chi(v^x)(i)$ and $\chi(w)(i)=\chi(w^x)(i)$, so $C_i(u^x;v^x, w^x)$ by tail-independence. If $i=i'$, then $C_i(u^x;v^x,w^x)$ since the action of $x$ in the $i^{\text{th}}$ coordinate is induced by an element of $G^{V_i}$ which preserves the $C$-relation on $V_i$. If $i'<i$ then $C_i(u^x;v^x,w^x)$ holds by tail-independence.
 \end{proof}
 
 {\em{Sub-claim 4.}} The group $K$ preserves the $L$-relation and $S$-relation on the branches at $r_i$. That is, if $\mu(u,v,w)=\sigma(u,v,w)=i$, then for $x \in K$ we have $L(u;v,w)\Leftrightarrow L(u^x;v^x,w^x)$, and similarly for $S$.
 
 \begin{proof}
 This is similar to Sub-claim 3.
 \end{proof}
 
 {\em{Sub-claim 5.}} The group $K$ preserves $L$ on $[n]$. 
 
 \begin{proof}
 Let $u,v, w \in [n]$ be distinct with $L(u;v,w)$. By Sub-claim 4, we may suppose $i=\sigma(u,v,w)< \mu(u,v,w)$. Thus, two of $u,v,w$ are $F_i$-equivalent and the other $F_i$-inequivalent to these. We suppose $uF_iv$ and $\neg uF_i w$ (the other cases are similar). Pick $z \in A^i(u)$ with $C_i(z;u,v)$, as shown in Figure \ref{fig:K pres L}.
 \begin{figure}[H]
     \centering
     \begin{tikzpicture}[scale=1]
     \node at (0,0){};
     \node (w) at (2,0){$w$};
     \node (z) at (0.5,1){$z$};
     \node (u) at (1,1) {$u$};
     \node (v) at (2,1){$v$};
     
     \draw (w)--(0,0);
     \draw (v)--(0,0);
     \draw (u)--(1,0.5);
     \draw (z)--(0.5,0.25);
     \end{tikzpicture}
          \caption{}
     \label{fig:K pres L}
 \end{figure}
  Then for $x \in K$, $L(u;v,w) \Leftrightarrow L(u;v,z)\xLeftrightarrow{\text{by Sub-claim 4}}  L(u^x; v^x, z^x) \Leftrightarrow  L(u^x;v^x,w^x) $ (since $x$ preserves the relations $E_j, F_j$ and $C$).
 \end{proof}
 
 {\em{Sub-claim 6.}} The group $K$ preserves $S$ on $[n]$.
 
 \begin{proof}
 Let $u,v,w,z \in [n]$ be distinct. Let $i$ be greatest such that $u/E_i, v/E_i, w/E_i, z/E_i$ are distinct. Then at least two of $u,v,w,z$ are $F_i$-equivalent. If all are $F_i$-equivalent, then $K$ preserves any $S$-relation among these by Subclaim 4. If just three of $u,v,w,z$ are $F_i$-equivalent, then $K$ preserves any $S$-relation among them by the proof of Subclaim 5. If say $uF_i v$ and $\neg u F_i w \wedge \neg u F_i z$, then as $K$ preserves $F_i$, if $x \in K$ we have $u^xF_i v^x \wedge \neg u^x F_i w^x \wedge w^x F_i z^x$. We now see $S(u,v;w,z) \wedge S(u^x,v^x;w^x,z^x)$ as required.
 \end{proof}
 
 By the sub-claims, the conditions of Lemma \ref{g is aut} are satisfied, completing the proof of Claim 4.
 
 \end{proof}
 
 It follows that $[n]$ is a Jordan set for $G$.
 
\end{proof}

\begin{rem} \label{grouponbranch}\rm
In the above proof, in view of the way the group $K$ is built from the groups $H_i$, it follows that for each  $r_i$ and branch $V_i$ at $r_i$, the group $G_{(M\setminus [n]),\{V_i\}}$ induces the whole group induced by $G_{\{V_i\}}$ on $V_i$.
\end{rem}

\begin{prop}\label{prebranch is Jo}
Each pre-branch is a Jordan set for $G$ in its action on $M$.
\end{prop}
\begin{proof}
Let $R$ be a $D$-set of $M$, and let $U$ be a branch of $R$ at a ramification point $r$, with corresponding pre-branch $\hat U\subset M$. Pick $z$ lying in a pre-branch at $r$ other than $\hat U$. We may choose a sequence $(r_i: i \in \mathbb N)$ of ramification points which is coinitial in $U$, that is, for each ramification point $r'$ in $U$ there is $i \in \mathbb N$ such that for all $j\geq i$, $r_j$ lies between $r$ and $r'$.
\tikzset{middlearrow/.style={
        decoration={markings,
            mark= at position 0.75 with {\arrow{#1}} ,
        },
        postaction={decorate}
    }
}
\begin{figure}[H]
    \centering
    \begin{tikzpicture}[scale=0.75]

\node [below] at (0,0){$r$};
\node(a) at (4,4){};
\node(u) at (1,2){};
\node(alpha')[below] at (1,0){};
\node(alpha) [below] at (1,1){$r_2$};
\node at (0.25,0.25){};
\node at (0.75, 0.75){};
\node (w) at (2,0.5){};
\node at (0.5,0.5){};
\node (b0) [below] at (3,3) {$r_0$};
\node at (4,3){};
\node at (2.5,4){};
\node at (2.5,2.5){};
\node at (2,3){};
\node  [below] at (2,2) {$r_1$};
\node at (4,2){};
\node at (1,3){};
\node at (2,1){};
\node [left] at (-2,0){$z$};
\node at (3.75,4.25)[above]{$\hat U$};
\draw[decoration={brace,mirror,raise=8pt},decorate,thick]
  (4.5,3.5) -- node[above=8pt] {}(2.5,4.5) ;

\draw (1,1)--(2,1);
\draw[middlearrow={stealth}] (0,0)--((1,1);
\draw (2,2)--(1,3);
\draw (2,2)--(4,2);
\draw [middlearrow={stealth}](0,0)--(2,2);
\draw [middlearrow={stealth}] (2,2)--(3,3);
\draw (3,3)--(2.5,4);
\draw  (a)--(0,0);
\draw (u)--(1,1);
\draw (0,0)--(0.5,0.5);
\draw  (0.5,0.5)--(1,1);
\draw (3,3)-- (4,3);
\draw (0,0)--(-2,0);

\end{tikzpicture}
    \caption{}
    \label{fig:union of branches}
\end{figure}
We may suppose in addition that $r_{i+1}$ lies between $r_i$ and $r$ for each $i$, and that $z$ lies in the special branch at $r_i$ for each $i$. For each $i$, there is a union $T_i$ of pre-branches at $r_i$ which is a pre-direction of a higher $D$-set. We may choose the $T_i$ so that  for each $i$, $r_i$ is a ramification point of one of the branches of $T_{i+1}$.

It follows that $T_i\subseteq T_{i+1}$ for each $i$ and that $ \bigcup\limits_{i\in \mathbb N}T_{i}= \hat U$. Since pre-directions are Jordan sets by Lemma \ref{Jordan set}, each $T_i$ is a Jordan set, so $\hat U$ is a Jordan set by Lemma \ref{typicalpair}.
\end{proof}

Recall from Definition~\ref{longdef} that given a $D$-set $R$ of $M$, the corresponding {\em pre-$D$-set} is the union of the predirections of $R$. 
\begin{lem}\label{pre branch Jordan}
Each pre-$D$-set $\hat{R_i}$  is a Jordan set for $G$. 
\end{lem}
\begin{proof}
Consider two distinct ramification points $r_1, r_2$ of $R$. Let $U_{r_1}$ be the branch at $r_1$ which includes  $r_2$,  and $U_{r_2}$ be the branch at $r_2$ containing $r_1$. We know by Proposition \ref{prebranch is Jo} that the corresponding pre-branches are Jordan sets and they form a typical pair, hence by Lemma \ref{typicalpair} their union is a Jordan set and is the whole pre-$D$-set.
\end{proof}
\begin{lem}\label{G_s-invariant C}
Let $s\in M$. Then there is a $G_s$-invariant $C$-relation on $M\setminus \{s \}$.
\end{lem}
\begin{proof}
Consider all the pre-$D$-sets that contain $s$ and the pre-branches $\hat U$ in these pre-$D$-sets that do not contain $s$, with the property that $s$ lies in the special branch at the ramification point at which $U$ is a branch.  Let $\mathscr K$ be this collection of pre-branches. The elements of this collection are all Jordan sets (by Proposition \ref{prebranch is Jo}). Now we check that $\mathscr K$ satisfies $(\romannum{1})$- $(\romannum{5})$ of  Lemma $2.2.2$ of \cite{adeleke1996classification}, applied to $G_s$ acting on $M \setminus \{s\}$. 

    (i) and (ii) are trivial; that is, each element of $\mathscr K$ has size greater than 1, and $\mathscr K$ is $G_s$-invariant. 

    (iii) $\mathscr K$ has no typical pair (Definition \ref{connected}(a) above). First, suppose that $\hat U, \hat V \in \mathscr K$ are pre-branches of the same $D$-set. Since $\hat U, \hat V$ both omit the element $s$ of this $D$-set, it is immediate that $\hat U, \hat V$ do not form a typical pair.
    
    Next, suppose $\hat U, \hat V\in \mathscr K$ are pre-branches of distinct but comparable $D$-sets $R$ and $R'$ respectively with $R'$ below $R$. We may suppose that $R$ lies in a cone of the structure tree corresponding to the ramification point $r$ of $R'$, and that $V$ is a branch at the ramification point $r'$ of $R'$. If $r=r'$, then $\hat U$ is a union of pre-branches at $r'$ omitting $s$, so contains $\hat V$ or is disjoint from $\hat V$. If $r$ lies in the  branch at $r'$ containing $s$, then again, $\hat U$ either contains $\hat V$ or $\hat U \cap \hat V=\emptyset$. If $r$ is a ramification point lying in $\hat V$, then $\hat U \subset \hat V$. And if $r$ lies in a branch at $r'$ other than $V$ or that containing $s$, then $\hat U \cap \hat V= \emptyset$.
    
    Finally, suppose that $\hat U$ and $\hat V$ are pre-branches of $D$-sets $R_1, R_2$ labelling incomparable vertices $\nu_1, \nu_2$ of the structure tree. Let $\mu :=\text{inf}\{\nu_1, \nu_2\}$, and $R$ be the $D$-set of $\mu$, and suppose $R_i$ corresponds to the ramification points $r_i$ of $R$, for $i=1,2$. Thus, $\hat U$ and $\hat V$ correspond to unions of pre-branches at $r_1$ and $r_2$ respectively of $R$, omitting $s$. Note that $s$ lies in both $R_1$ and $R_2$, and hence also in $R$.  If, say, $r_2$ is a ramification point of $ U$, then $r_1$ is not a ramification point of $V$, (otherwise $s \in \hat U \cup \hat V)$, and $\hat V \subset \hat U$; likewise with $r_1,r_2$ reversed. Alternatively, $r_2$ is not a ramification point of $\hat U$, and $r_1$ is not a ramification point  of $\hat V$, and in this case $\hat U \cap \hat V=\emptyset$.

    (iv) We must show that given distinct $u,v\in M\setminus \{s\}$ there is a member of $\mathscr K$ containing $u,v$.  Choose a $D$-set $R$ such that the pre-$D$-set $\hat R$ contains $u,v,s$ in distinct pre-directions. There is a ramification point $r$ at $R$ such that $s$ lies in the special pre-branch at $r$, and $u,v$ lie in the same other pre-branch $\hat U$ at $r$. Then $\hat U \in \mathscr K$ and contains $u,v$.

    (v) We  show that given distinct $u,v\in M\setminus \{s\}$ there is a member of $\mathscr K$ containing $u$ but not $v$. Choose a $D$-set $R$ such that $\hat R$ contains $u,v,s$ in distinct pre-directions, meeting at ramification point $r$. There is a ramification point $r'$ in the branch at $r$ containing $u$, such that the branch at $r'$ containing $s$ is special. Let $\hat U$ be the pre-branch at $r'$ containing $u$. Then $\hat U \in \mathscr K$ and contains $u$ but not $v$.
   
Now define a ternary relation $C_s$ such that for every $x,y,z \in M \setminus \{s \}$, the relation $C_s(x;y,z)$ holds if and only if $(\exists U \in \mathscr K) ( y,z \in U \wedge x \notin U)$. Then $C_s$ is a $G_s$-invariant $C$-relation by Lemma $2.2.2$ of \cite{adeleke1996classification}.  

\end{proof}

\begin{lem}\label{no sep}
There is no $G$-invariant separation relation on $M$.
\end{lem}
\begin{proof}
Choose a configuration in $M$ as depicted in Figure~\ref{fig:no sep}, in some $D$-set. 
\tikzset{middlearrow/.style={
        decoration={markings,
            mark= at position 0.75 with {\arrow{#1}} ,
        },
        postaction={decorate}
    }
}
\begin{figure}[H]
    \centering
\begin{tikzpicture}
[scale=1.5]
\node at (-0.5,0){};
\node(x) at (-1,0.5){$x$};
\node(y) at (-1,-0.5){$y$};
\node at (0.5,0){}; 
\node at (0,0){};
\node (u) at (1,0.5)[above]{$u$};
\node (v) at (1,-0.5)[below]{$v$};
\node (z) at (0,0.5){$z$};

\draw  [middlearrow={stealth}] (z)--(0,0);
\draw (x)--(-0.5,0);
\draw (-0.5,0)--(y);
\draw (u)--(0.5,0);
\draw (v)--(0.5,0);
\draw (-0.5,0)--(0.5,0);
\draw [middlearrow={stealth}] (0,0)--(-0.5,0);
\draw [middlearrow={stealth}] (0,0)--(0.5,0);
\end{tikzpicture}
    \caption{}
    \label{fig:no sep}
\end{figure}
By semi-homogeneity there is $g \in G$ inducing $(x)(y)(z)(uv)$. It is easily seen that a permutation of $M$ with such cycle structure cannot preserve a separation relation on $M$.

  \end{proof}
\begin{lem}\label{no steiner}
There is no $G$-invariant Steiner system on $M$.
\end{lem}
{\bf{Note}}: We use the idea of the proof of Lemma 6.5 in \cite{bhattmacph2006jordan}.

\begin{proof}
For a contradiction, suppose there is a $G$-invariant Steiner $n$-system on $M$. Let $s_1, \dots, s_{n+1}$ be distinct elements of  a block $\mathfrak B$ of the Steiner system. Since we may choose a $D$-set in which all $s_i$ lie in different branches at a ramification point, there is a pre-branch $V$ containing $s_{n+1}$ and omitting $s_1, \dots, s_n$. 
Let $t \in V$. Since $V$ is a Jordan set, there is $g \in G_{(M\setminus V)}$ with $s_{n+1}^g = t$. As $g$ fixes $s_1, \dots, s_n$, it fixes setwise the unique block $\mathfrak B$ containing $s_1, \dots, s_n$, so as $s_{n+1}\in \mathfrak B$, also $t \in \mathfrak B$; that is, $V \subseteq \mathfrak B$. 

Let $s^*$ be an element of $M\setminus \mathfrak B$ (hence not in $V$) and $\mathfrak B'$ be the block containing $s_1, \dots, s_{n-2}, s_{n+1}, s^*$. As $\lvert \mathfrak B '\rvert \geq n+1$, there is $s^{**} \in \mathfrak B '$ distinct from $s_1, \dots, s_{n-2},s_{n+1}, s^*$ with $s^{**} \notin \mathfrak B$, so as $V \subseteq \mathfrak B$ then $s^{**} \notin V$. But $s_1, \dots, s_{n-2}, s^*, s^{**} $ are all in $\mathfrak B '$ so determine $\mathfrak B'$. So as $s_{n+1} \in V \cap \mathfrak B'$, by the above  argument using the Jordan property of  $V$,  we obtain $V \subseteq \mathfrak B'$. So $V \subseteq \mathfrak B \cap \mathfrak B'$, a contradiction as  $V$ is infinite and $\mid \mathfrak B \cap \mathfrak B'\mid = n-1$. 
\end{proof}

\begin{lem}\label{No D}
There is no $G$-invariant $D$-relation on $M$.
\end{lem}
\begin{proof}
Suppose, for a contradiction, that there is a $G$-invariant $D$-relation $D$ 
defined on $M$. Fix $x,y, z_0\in M$. Find $u_1 \in M \setminus \{x,y,z_0\}$ with $D(u_1,z_0;x,y)$. Note that in the argument below, we should not confuse $D$ with the various $D$-sets in $M$ coded by the structure tree.  

Find a $D$-set $R_1$ of $M$ containing  $u_1, z_0, x,y$ in distinct branches at the same ramification point $r_1$, and pick $v_1 \in M$ lying in the pre-branch at $r_1$ containing $z_0$, with $L(z_0;v_1,x)$ witnessed in this $D$-set. See Figure \ref{fig:The $D$-set $R_1$}.
\tikzset{middlearrow/.style={
        decoration={markings,
            mark= at position 0.75 with {\arrow{#1}} ,
        },
        postaction={decorate}
    }
}
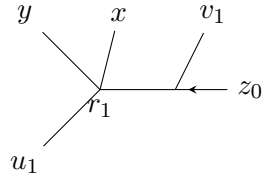
\begin{figure}[H]
    \centering
    \begin{tikzpicture}[scale=1]
    \node at (0,0)[below]{$r_1$};
    \node (y) at (-1,1){$y$};
    \node (x) at (0.25,1){$x$};
    \node (u1) at (-1,-1){$u_1$};
    \node at (1,0){};
    \node (z0) at (2,0) {$z_0$};
    \node (v1) at (1.5,1){$v_1$};
    
    \draw (y)--(0,0);
    \draw (x)--(0,0);
    \draw (u1)--(0,0);
    \draw (v1)--(1,0);
    \draw [middlearrow={stealth}](z0)--(1,0);
    \draw (0,0)--(1,0);
    \end{tikzpicture}
    \caption{The $D$-set $R_1$}
    \label{fig:The $D$-set $R_1$}
\end{figure}
Let $z_1 \in M \setminus \hat R_1$. Choose $h_1,k_1 \in G_{z_0, z_1}$ with $(x,v_1)^{h_1} =(v_1,x)$ and $(u_1, v_1)^{k_1}=(v_1,u_1)$ - these exist by semi-homogeneity.

In the $D$-relation on $M$, consider the regions $P,Q,R,S$ as depicted (here $x\in R, u_1\in S, z_0\in P$).
\tikzset{middlearrow/.style={
        decoration={markings,
            mark= at position 0.75 with {\arrow{#1}} ,
        },
        postaction={decorate}
    }
}
\begin{figure}[H]
    \centering
    \begin{tikzpicture}[scale=1]
    \node (r0) at (2,0)[right]{$z_0$};
    \node at (1,0){};
    \node at (0,0){};
    \node at (-1,0){};
    \node (x) at (-2,0)[left]{$x$};
    \node (R) at (-2,1)[above]{$R$};
     \node (Q) at (-0.75,1)[above]{$Q$};
     \node (P) at (0.25,1)[above]{$P$};
     \node (u1) at (-1.5,-1)[below]{$u_1$};
     \node at (-0.75,-0.5){};
     \node (S) at (0,-1)[right]{$S$};
     
     \draw [->](-0.75,-0.5)--(0,-1);
     \draw (2,0)--(-2,0);
     \draw [->](-1,0)--(-2,1);
     \draw [->] (0,0)--(-0.75,1);
     \draw (0,0)--(-1.5,-1);
     \draw [->](1,0)--(0.25,1);
     
    \end{tikzpicture}
    \caption{}
    \label{fig:PQRS}
\end{figure}

Let supp$\langle h_1,k_1\rangle$ denote the set of elements of $M$ moved by some element of the subgroup $\langle h_1, k_1 \rangle$ of $G$ generated by $h_1$ and $k_1$. If say $v_1 \in R$, then we see that $R\cup S\subseteq \text{supp}(k_1) \subseteq \text{supp}\langle h_1, k_1 \rangle.$  If $v_1 \in S$ then
$R\cup S\subseteq \text{supp}(h_1) \subseteq \text{supp}\langle h_1, k_1 \rangle.$ If $v_1 \in Q$ then 
$R\subseteq \text{supp}(h_1) \subseteq \text{supp}\langle h_1, k_1 \rangle$, and $S \subseteq \text{supp}(k_1) \subseteq \text{supp}\langle h_1, k_1 \rangle$. Finally, if $v_1 \in P$  then
$R, S\subseteq \text{supp}(h_1) \subseteq \text{supp}\langle h_1, k_1 \rangle.$ Thus, wherever $v_1$ lies,  $R\cup S \subseteq \text{supp}\langle h_1, k_1 \rangle,$ 
so as $h_1, k_1$ fix $z_1$, so $z_1 \notin R\cup S$. Thus, $z_1\in P \cup Q$. Since $D(u_1,z_0;x,y)$, $y \in R$, so we have the following picture.
\begin{figure}[H]
\centering
\begin{tikzpicture}[scale=1]
\node (z0) at (3,0){$z_0$};
    \node (y) at (-3,0){$y$};
    \node (z1) at (2,1){$z_1$};
    \node (x) at (-2,1){$x$};
    
    \draw (z0)--(y);
    \draw (z1)--(1,0);
    \draw (x)--(-1,0);
\end{tikzpicture}
\caption{} \label{fig:D(x,y;z_1,z_0}
\end{figure}

Now we iterate this argument with  $(z_0,x, z_1)$ in place of $(z_0,x,y)$. Pick $u_2\in M\setminus \{x,z_0,z_1\}$ with $D(u_2,z_0;x,z_1)$. Find a $D$-set $R_2$ of $M$ containing  $u_2,z_0,x,z_1$ in distinct branches at the same ramification point $r_2$, and pick $v_2\in M$ lying in the pre-branch at $r_2$ containing $z_0$, with $L(z_0;v_2,x)$ witnessed in this $D$-set. Let $z_2 \in M\setminus \hat R_2$. By semi-homogeneity there are $h_2,k_2 \in G_{z_0, z_2}$ with $(x,v_2)^{h_2} =(v_2,x)$ and $(u_2,v_2)^{k_2}=(v_2,u_2)$. Let $x, z_0,u_2,P',Q',R',S'$ replace $x,z_0,u_1,P,Q,R,S$ above. We see that $z_2 \in P'\cup Q'$, and thus the $D$-relation on $M$ satisfies the following picture.
\begin{figure}[H]
\centering
\begin{tikzpicture}[scale=1]
\node (z0) at (3,0){$z_0$};
    \node (x) at (-3,0){$x$};
    \node (z1) at (0,1){$z_2$};
    \node (y) at (-2,1){$y$};
    \node at (0,0){};
    \node (z2) at (-1,1){$z_1$};
    
    \draw (z2)--(0,0);
    \draw (z0)--(x);
    \draw (z1)--(1,0);
    \draw (y)--(-1,0);
\end{tikzpicture}
\caption{} \label{fig:D(x,y;z_1,z_2)}
\end{figure}
Observe that as $z_1\not\in \hat R_1$ and $z_2\not\in \hat R_2$,  we have $L(z_1;x,z_0) \wedge L(z_2;x,z_0)\wedge L(z_2;x,z_1) \wedge L(z_2;z_0,z_1)$. Thus, by semi-homogeneity, there is $g \in G_{z_1, z_2}$ inducing $(x,z_0)$. Such $g$ does not preserve the $D$-relation on $M$, a contradiction.
\end{proof}

\subsection{Proof of Main Theorem}\label{proof of main Th}
In this section, we show that $G=\text{Aut}(M,L,S)$ is an infinite primitive Jordan group preserving a limit of $D$-relations (Definition \ref{limits}).

We may view $M$ as an $\mathscr L$-structure, or as a structure in just the language with symbols $L$ and $S$, since, by Lemma \ref{LL'SS'}, the other $\mathscr L$-symbols are $\emptyset$-definable in terms of $L$ and $S$.

Let $\hat R$ be a pre-$D$-set with $D$-set $R$, let $H:=G_{(M\setminus \hat R)}$ and let $E$ be the equivalence relation on $\hat R$ corresponding to being in the same direction (the equivalence relation identified in Definition \ref{Dsetdef}(ii)). Let $D$ be the induced $D$-relation on $R=\hat R/E$. 
\begin{lem}\label{H_i}
In the above notation,
\begin{enumerate}[(i)]
    \item $H$ preserves $E$ and the relation $D$;
    \item $H$ is transitive on $\hat{R}$;
    \item $H$ is $2$-transitive but not $3$-transitive on $ R$; and
    \item $E$ is the unique maximal $H$-congruence on $\hat{R}$.
\end{enumerate}
\end{lem}
\begin{proof}
\begin{enumerate}[(i)]
    \item  $H$ preserves $E$ as $H < G_{\{M \setminus \hat {R}\}}$, which preserves $E$ as noted after Definition~\ref{defE}. 
    Also, the assertion that $H$ preserves $D$ follows from Lemma \ref{4}(\ref{D preserved}). 
    
    \item This follows from \ref{pre branch Jordan}.
    \item Fix $x_0\in \hat{R}$, and let $x_0/E$ denote the $E$-class of $x_0$. We show that $H_{x_0 /E}$ is transitive on $ R \setminus \{x_0 /E\}$. Let $u,v$ be $E$-inequivalent elements of $ R \setminus \{x_0\}$. Choose a ramification point $r$ such that there is a branch $U$ at $r$ containing $u/E,v/E$ and omitting $x_0 /E$. It is known that $\hat{U}$ is a Jordan set (pre-branches are Jordan sets) so there is $g \in G_{(M\setminus \hat U)}<H$ with $u^g=v$ and hence $(u/E)^g=v/E$. However, $H$ is not $3$-transitive; for if $u,v,w \in  \hat{R}$ and meet at a ramification point $r$ with $L(u;v,w)$ then there is no element of $H$ inducing $(u/E,v/E)(w/E)$. 
    \item The invariance of $E$ follows from (i), and its maximality from (iii).  For the uniqueness, suppose $E^*$ is an $H$-congruence on $\hat R$ and there are $u,v \in \hat R$ with $\neg u E v$ and $uE^* v$. Since pre-directions are Jordan sets, for $v' \in \hat R$ if $v' E v$ there is $g \in H$ fixing $M \setminus ( v/E)$ pointwise with $v^g= v'$. As $u^g=u$, $g$ fixes $E^*(u)$, so $v E^* v'$, so $v/E \subset v/E^*$, so $E^*$ contains $E$ properly, hence is universal by maximality of $E$.
\end{enumerate}
\end{proof}

\begin{thm}\label{main Th}
$G$ preserves a limit of $D$-relations on $M$.
\end{thm}
\begin{proof}
Let $G=\text{Aut}(M)$. Then $G$ is an infinite Jordan group acting on $M$. Let $T$ be the structure tree of $M$ (so $T=(K^*/R,\leq)$, as identified in Lemma~\ref{semi2}).
 Let $J$ be a maximal chain from $T$. Then $J$ is  linearly ordered by $\leq$. Let $R_j$ be the $D$-set indexed by $j$, for $j \in J$. Then by the paragraph below Lemma \ref{QRIJ}, for $i,j \in J$ we have $i<j \Leftrightarrow \hat R_j \subset \hat R_i$. Thus $(\hat{R}_j:j\in J)$ is  a strictly increasing chain of subsets of $M$, where the ordering under inclusion is the reverse of that induced from the index set $J$. Let $\hat R_j$ be the pre-$D$-set corresponding to $R_j$, let $H_j:=G_{(M\setminus \hat R_j)}$, and let $E_j$ be the unique maximal $H_j$-congruence on $\hat R_j$ as in Lemma \ref{H_i}(iv). Then $\{H_j: j \in J\}$ is an increasing chain of subgroups of $G$, with the ordering under inclusion reversed from that of $J$. We must check the conditions (\romannum{1})-(\romannum{8}) in Definition \ref{limits}.

\begin{enumerate}[(i)]
    \item This follows from (\romannum{2}) and (\romannum{4}) in the Lemma \ref{H_i} above.
    \item This is (\romannum{1}) and (\romannum{3}) in  Lemma \ref{H_i} above. Note that since pre-branches and pre-directions are Jordan sets of $G$, branches are Jordan sets of each $(H_j,R_j)$, so the latter are Jordan groups.
    \item It is clear that  $\bigcup (\hat R_i: i \in J)=M$.
    \item Let $ H:=\bigcup\limits_{j\in J}H_{j}$. Then $H$ is a Jordan group on $M$, since each $R_j$ is a Jordan set for $H$.  The group $G$ is not $3$-transitive since it preserves the relation $L$ (and $L(u;v,w)\rightarrow \neg L(v;u,w)$), hence $H$ is not $3$-transitive.
    
    We now show that $H$ is $2$-primitive on $M$.
We first observe a point from Lemma \ref{Jordan set}. In the proof of that lemma (see also Remark~\ref{grouponbranch}), if $[n]$ is a pre-direction of the $D$-set labelled by the  vertex $j_n$, then for each $j<j_n$ there is a $D$-set $R_j$ and ramification point $r_j$ such that $[n]=\bigcup\bigcup S_j$ for some set  $ S_j$ of branches at $r_j$. It follows from that proof that for each branch $U \in S_j$ at $r_j$, the pointwise stabiliser of the complement of $[n]$ induces $G^U$ on $U$.

Now let $x_0 \in M$, and let $\rho$ be a nontrivial $H_{x_0}$-congruence on $M \setminus \{x_0\}$. We must show that $\rho$ is universal. Pick distinct $u,v\in M \setminus \{x_0\}$ with $u\neq v$. Choose $j \in J$ such that $x_0,u,v$ lie in distinct pre-directions of $R_j$.  Let $\mathfrak B$ be the $\rho$-class containing $u$. For a contradiction, we suppose that $\rho$ is not universal, so may suppose that $\mathfrak B$ does not contain each pre-direction of $R_j$ other than that of $x_0$. In particular by (iii) and Lemma~\ref{Jordan set}, it follows that $\mathfrak B$ is a proper subset of $\hat R_j$ omitting at least two pre-directions, including that of $x_0$.

Let $r$ be a ramification point of $R_j$ such that $u,v$ lie in the same pre-branch $\hat U$ at $r$, and $x_0$ in a different pre-branch. Let $C$ be the $C$-relation induced on the corresponding branch $U$ at $r$. Suppose there are distinct $u',v',w' \in \hat U$ such that $C(u'/E_j;v'/E_j,w'/E_j)$ and $u'\rho w'$. Let $V$ be the largest branch in $U$ containing $v',w'$ and omitting $u'$. Then the pre-branch $\hat V$ is a Jordan set, so there is $g \in G_{(M\setminus \hat V)}<H_{x_0}$ with $(u',w')^g=(u',v')$. Since $g$ fixes $u'$, it follows that $v' \rho w'$. Thus $\mathfrak B \cap \hat R_j$ is a pre-branch of $R_j$, the union of a nested sequence of pre-branches of $R_j$, or a union of more than one pre-branch at some fixed vertex. By choosing $j$ sufficiently low in the structure tree, we may assume that the last one holds, i.e. $\mathfrak B \cap \hat R_j$ is the union of more than one pre-branch at a ramification point $r_j$ of $R_j$.

Pick a ramification point $r^*$ of $R_j$ such that elements of $\mathfrak B$ and $x_0$ lie in distinct pre-branches at $r^*$ with the pre-branch containing elements of $\mathfrak B$ non-special, and that containing $x_0$ special. There is a pre-direction $[n]$ which is a union of pre-branches at $r^*$ including the pre-branch $\hat V$ at $r^*$ containing $\mathfrak B$, and excluding that containing $x_0$. Now by the observation above (i.e. Remark~\ref{grouponbranch}), since $G_{(M \setminus [n])}\leq H$, $H$ induces the full group $G^V$ on $V$. In particular, using semi-homogeneity there is a ramification point $r$ between $r^*$ and $r_j$ such that $H_{x_0}$ contains an element $h$ with $u^h=u$ and $r_j^h=r$. It follows that $\mathfrak B ^h\supset \mathfrak B$, contradicting that $\mathfrak B$ is a block of $H_{x_0}$.

\item $E_j\big |_{\hat{R_i}} \subseteq E_i$ if $i>j$, by Lemma \ref{FE}.
    \item We claim that $\bigcap (E_i: i \in J)$ is equality. Let $u,v \in M$ be distinct. By $2$-transitivity of $G$, there is a $D$-set $R$ such that $u,v$ lie in distinct directions of $R$. Choose $j \in J$ such that the corresponding $D$-set $R_j$ labels a vertex of the structure tree below that of $R$. Then $u,v$ lie in distinct directions of $R_j$, so $\neg u E_j v$. 
    \item Given $g \in G$, choose an initial  segment $I$ of $J$ which lies in the common part of $J$ and $J^g$. Let $i_0 \in I^{g^{-1}}\subseteq J^{g^{-1}}\cap J$. Then for any $i<i_0$ we have $i^g<i_0^g$ and so $i^g\in I$. Thus $i^g=j$ for some $j\in J$. Hence $g^{-1}H_i g= H_j$ and $R_i ^g = R_j$.
    \item This is by \ref{G_s-invariant C}. 
\end{enumerate}

\end{proof}

\begin{thm}
There is a ternary relation $L$ and a quaternary relation $S$ on a countably infinite set $M$, such that if $G:= \text {\em{Aut}}(M,L,S)$, then $G$ is oligomorphic, $3$-homogeneous, $2$-primitive but not $3$-transitive or $4$-homogeneous on $M$, and is a Jordan group preserving a limit of $D$-relations on $M$, and not preserving any of the structures of types $(i)-(iii)$ in Theorem \ref{am-class}.
\end{thm}
\begin{proof}
This is by Lemma \ref{olig}, Lemma \ref{Aut(M)}, Lemma \ref{no sep}, Lemma \ref{no steiner}, Lemma \ref{No D} and Theorem \ref{main Th}. The group $G$ is not 4-homogeneous as some but not all quadruples satisfy $S$ under some ordering. Note that $G$ cannot preserve a linear or circular order or a linear betweenness relation since it does not preserve a separation relation, $G$ cannot preserve a $C$-relation since it does not preserve a $D$-relation, and cannot preserve a semilinear order or general betweenness relation since it is 2-primitive.
\end{proof}

\section{Further Questions}
We have a number of questions around the construction in this paper, its companion in \cite{bhattmacph2006jordan}, and the exact statement of Theorem~\ref{am-class} which was proved in \cite{adeleke1996classification}. We also have questions concerning the flexibility of our construction, 
and how it fits in the developing theory of homogeneous and $\omega$-categorical structures.

\begin{problem} \label{prob1}\rm
Axiomatise a concept of {\em $(L,S)$-structure}. The idea here is to identify a set, probably finite, of axioms for a ternary relation $L$ and quaternary relation $S$, from which can be derived the basic combinatorics of Sections 2 and 3 above. In particular, it should be possible from the axioms to interpret in any $(L,S)$-structure a semilinear order (the `structure tree'), a family of $D$-sets in bijection with the vertices of the semilinear order, a concept of special branch at a ramification point of a $D$-set, the maps $f_\nu$ associating cones at the vertex $\nu$ of the structure tree with ramification points of the associated $D$-set $D(\nu)$, and the corresponding maps $g_{\mu\nu}$. 
There is need for an analogous axiomatisation of the corresponding ternary relation (also denoted by $L$) in \cite{bhattmacph2006jordan} -- there is an initial discussion of this in the last section of that paper. This should also be done for limits of Steiner systems.
\end{problem}

\begin{problem} \label{prob2}\rm
Sharpen Theorem~\ref{am-class} above (the main result of   \cite{adeleke1996classification}), and its proof there, so that in Case (iv) the notion of limit of betweenness or $D$-relation (and possibly of Steiner system) is replaced by the concept identified in Problem~\ref{prob1}. At the very least, it should be possible to replace the total order $I$ in Definition~\ref{limits} by an {\em invariant} semilinear order, with a corresponding modification of the proof of Theorem~\ref{am-class}. 
\end{problem}

\begin{problem} \label{prob3} \rm
Clarify the connection between a limit of $D$-relations and a limit of general betweenness relations. For example, is the structure constructed in \cite{bhattmacph2006jordan} interpretable in the structure constructed in this paper (a question asked by Peter Cameron). Note that any $D$-relation interprets a general betweenness relation. 
\end{problem}

In his PhD thesis \cite{bradleywilliams}, David Bradley-Williams initiated a construction of a limit of betweenness relations based on a discrete rather than a dense semilinear order. It has not yet been shown that the associated automorphism group is a Jordan group.

\begin{problem} \label{prob4} \rm
Show that the constructions in this paper and in \cite{bhattmacph2006jordan} can be carried out with a wide class of semilinear orders as structure tree, yielding structures whose automorphism groups are Jordan groups. Show that Adeleke's constructions in \cite{adeleke2013irregular} can be incorporated into this framework. Can the betweenness relations and $D$-sets in these structures be replaced by other kinds of relational structures?
\end{problem}

Recall that a relational structure $M$ is {\em homogeneous} (in the sense of Fra\"iss\'e) if it is countably infinite and any isomorphism between finite substructures of $M$ extends to an automorphism of $M$. We say $M$ is {\em homogenisable} if there is a homogeneous structure $N$ on the same domain as $M$, such that the language of $N$ is finite relational, and ${\rm Aut}(M)={\rm Aut}(N)$ (as permutation groups). Recall also the model-theoretic notion of an {\em NIP} structure (or theory) -- see for example \cite{simon}.

\begin{problem} \label{prob5} \rm
Show that the structure $M$ constructed in this paper  is not homogeneous, is homogenisable, and is NIP.
\end{problem}

\begin{problem} \label{prob6} \rm
With $G={\rm Aut}(M)$ as in this paper, let $n_k(G)$ be the number of orbits of $G$ on the set of $k$-element subsets of $M$. Find the asymptotic growth rate of the sequence $(n_k(G))$.
\end{problem}
Regarding the last problem, we know by the main theorem of \cite{macpherson-orbits} that $n_k(G)$ is bounded below by an exponential function. There are very few known examples of oligomorphic primitive permutation groups for which the growth is bounded above exponentially. Most of these examples are associated with treelike structures. 

There is a well-known connection between valued fields and treelike structures. For example, given a field $F$ equipped with a non-trivial  valuation map $v:F \to \Gamma \cup \{\infty\}$ where $\Gamma$ is an ordered abelian group, there is a $C$-relation on  $F$, invariant under addition and multiplication by non-zero elements, give by $C(x;y,z) \Leftrightarrow (v(x-y)<v(y-z))$; see for example \cite{macpherson-steinhorn}. The well-known graph-theoretic tree on which ${\rm SL}_2({\mathbb Q}_p)$ acts, as described in Chapter II of Serre \cite{serre}, is associated with this. There is a $D$-relation on the projective line ${\rm PG}_1(F)$ defined by putting $D(x,y;z,w)$ if and only the cross ratio $[x,y;z,w]$ lies in $1+\mathcal{M}$, where $\mathcal{M}$ is the maximal ideal of the corresponding valuation ring -- see \cite[Theorem 30.4]{adeleke1998relations}. It is also well-known that the set of all valuations on a field is lower semilinearly ordered under reverse inclusion of the corresponding valuation rings. This suggests the following problem.

\begin{problem} \label{prob7} \rm 
Show that the structure $M$ in this paper, or more generally an $(L,S)$-structure as in Problem~\ref{prob1} `lives' on a field, in the sense that the structure tree can be identified with a set of valuation rings of the field. 
\end{problem}

\bibliographystyle{plain}
\bibliography{references}

\begin{thebibliography}{10}

\bibitem{adeleke1995semilinear}
S.~Adeleke.
\newblock {Semilinear tower of Steiner systems}.
\newblock {\em Journal of Combinatorial Theory, Series A}, 72(2):243--255,
  1995.

\bibitem{adeleke2013irregular}
S.~Adeleke.
\newblock {On irregular infinite Jordan groups}.
\newblock {\em Communications in Algebra}, 41(4):1514--1546, 2013.

\bibitem{adeleke1996classification}
S.~Adeleke and H.D. Macpherson.
\newblock {Classification of infinite primitive Jordan permutation groups}.
\newblock {\em Proceedings of the London Mathematical Society}, 3(1):63--123,
  1996.

\bibitem{adeleke1996infinite}
S.~Adeleke and P.M. Neumann.
\newblock Infinite bounded permutation groups.
\newblock {\em Journal of the London Mathematical Society}, 53(2):230--242,
  1996.

\bibitem{adeleke1996primitive}
S.~Adeleke and P.M. Neumann.
\newblock {Primitive permutation groups with primitive Jordan sets}.
\newblock {\em Journal of the London Mathematical Society}, 53(2):209--229,
  1996.

\bibitem{adeleke1998relations}
S.~Adeleke and P.M. Neumann.
\newblock {\em Relations Related to Betweenness: Their Structure and
  Automorphisms}, volume 623.
\newblock Memoirs, American Mathematical Soc., 1998.

\bibitem{asmathesis}
A.~Almazaydeh.
\newblock {\em Infinite Jordan permutation groups}.
\newblock PhD thesis, University of Leeds, 2019.

\bibitem{bhattmacph2006jordan}
M.~Bhattacharjee and H.D. Macpherson.
\newblock Jordan groups and limits of betweenness relations.
\newblock {\em Journal of Group Theory}, 9(1):59--94, 2006.

\bibitem{bodirsky-macpherson}
M.~Bodirsky and H.D. Macpherson.
\newblock Reducts of structures and maximal-closed permutation groups.
\newblock {\em Journal of Symbolic Logic}, 81(3):1087--1114, 2016.

\bibitem{bradleywilliams}
D.~Bradley-Williams.
\newblock {\em Jordan groups and homogeneous structures}.
\newblock PhD thesis, University of Leeds, 2014.

\bibitem{cameron1976transitivity}
P.J. Cameron.
\newblock Transitivity of permutation groups on unordered sets.
\newblock {\em Mathematische Zeitschrift}, 148(2):127--139, 1976.

\bibitem{cameron1987some}
P.J. Cameron.
\newblock Some treelike objects.
\newblock {\em The Quarterly Journal of Mathematics}, 38(2):155--183, 1987.

\bibitem{cameron1990oligomorphic}
P.J. Cameron.
\newblock {\em Oligomorphic permutation groups}, volume 152.
\newblock Cambridge University Press, 1990.

\bibitem{evans1994examples}
D.M. Evans.
\newblock Examples of aleph-zero categorical structures.
\newblock {\em Automorphisms of first-order structures (Eds. R.W. Kaye and H.D.
  Macpherson)}, pages 33--72, 1994.

\bibitem{hall1962wreath}
P.~Hall.
\newblock Wreath powers and characteristically simple groups.
\newblock In {\em Mathematical Proceedings of the Cambridge Philosophical
  Society}, volume~58, pages 170--184. Cambridge University Press, 1962.

\bibitem{hodges1997shorter}
W.~Hodges.
\newblock {\em A shorter model theory}.
\newblock Cambridge University Press, 1997.

\bibitem{johnson2002constructions}
K.~Johnson.
\newblock {Constructions of semilinear towers of Steiner systems}.
\newblock {\em Tits buildings and the model theory of groups (Ed. K. Tent),
  London Math. Soc. Lecture Notes}, 291:235--278, 2002.

\bibitem{kantor1985homogeneous}
W.M. Kantor.
\newblock Homogeneous designs and geometric lattices.
\newblock {\em Journal of Combinatorial Theory, Series A}, 38(1):66--74, 1985.

\bibitem{kaplan2016affine}
I.~Kaplan and P.~Simon.
\newblock The affine and projective groups are maximal.
\newblock {\em Transactions of the American Mathematical Society},
  368(7):5229--5245, 2016.

\bibitem{cherlin1985aleph}
G.~Cherlin L.~Harrington and A.H. Lachlan.
\newblock $\aleph_0$-categorical, $\aleph_0$-stable structures.
\newblock {\em Annals of Pure and Applied Logic}, 28(2):103--135, 1985.

\bibitem{macpherson-orbits}
H.D. Macpherson.
\newblock Orbits of infinite permutation groups.
\newblock {\em Proceedings of the London Mathematical Society (3)},
  51(2):246--285, 1985.

\bibitem{macpherson-praeger-cycle}
H.D. Macpherson and C.E. Praeger.
\newblock Cycle types in infinite permutation groups.
\newblock {\em Journal of Algebra}, 175(1):212--240, 1995.

\bibitem{macpherson-steinhorn}
H.D. Macpherson and C.~Steinhorn.
\newblock On variants of o-minimality.
\newblock {\em Annals of Pure and Applied Logic}, 79(2):165--209, 1996.

\bibitem{neumann1985some}
P.M. Neumann.
\newblock Some primitive permutation groups.
\newblock {\em Proceedings of the London Mathematical Society}, 3(2):265--281,
  1985.

\bibitem{serre}
J-P. Serre.
\newblock {\em Trees}.
\newblock Springer, 1986.

\bibitem{simon}
P.~Simon.
\newblock {\em A guide to NIP theories}, volume Lecture Notes in Logic 44.
\newblock Cambridge University Press, 2015.

\end{thebibliography}
\end{document}